\documentclass{amsart}
\usepackage{graphicx} 
\usepackage{amsmath,amstext,amssymb,bm}
\usepackage{leftidx}
\numberwithin{equation}{section}
\usepackage{xcolor}
\usepackage{soul}
\usepackage{tikz}
\usetikzlibrary{shapes,arrows}
\usepackage{mathrsfs}

\newtheorem{definition}{Definition}[section]
\newtheorem{proposition}{Proposition}[section]
\newtheorem{remark}{Remark}[section]
\newtheorem{theorem}{Theorem}[section]
\newtheorem{corollary}{Corollary}[section]
\newtheorem{lemma}{{Lemma}}[section]

\allowdisplaybreaks[4]

\newcommand{\eps}{\varepsilon}
\newcommand{\Ome}{{\Omega}}

\newcommand{\supp}{{\text{\rm supp}}}

\newcommand{\tH}{\widetilde{H}}
\newcommand{\tW}{\widetilde{W}}

\newcommand{\R}{\mathbb{R}}

\newcommand{\N}{\mathbb{N}}

\begin{document}

\title[A NEW THEORY OF WEAK FRACTIONAL CALCULUS]{A new theory of fractional differential calculus and fractional Sobolev spaces: One-dimensional case
}

\author{Xiaobing Feng\dag}
\thanks{\dag Department of Mathematics, The University of Tennessee, Knoxville, TN 37996, U.S.A. (xfeng@math.utk.edu).
	The work of this author was partially supported by the NSF grant: DMS-1620168.}

\author{Mitchell Sutton\ddag}
\thanks{\ddag Department of Mathematics, The University of Tennessee, Knoxville, TN 37996, U.S.A. (msutto11@vols.utk.edu).
	The work of this author was partially supported by the NSF grant: DMS-1620168.}
	
%



\begin{abstract}
    This paper presents a self-contained new theory of weak fractional differential calculus and fractional Sobolev spaces in one-dimension. The crux of this new theory is the introduction of a  weak fractional derivative notion which is a natural generalization of integer order weak derivatives; it also helps to unify multiple existing fractional derivative definitions   and characterize what functions are fractionally differentiable. Various calculus rules including a fundamental theorem of calculus, product and chain rules, and integration by parts formulas are established for weak fractional derivatives and relationships with classical fractional derivatives are also obtained. Based on the weak fractional  derivative notion, new fractional order Sobolev spaces are introduced and many important theorems and properties, such as density/approximation theorem, extension theorems, trace theorem, and various embedding theorems in these Sobolev spaces are established. Moreover, a few relationships with existing fractional Sobolev spaces are also established. Furthermore,  the notion of weak fractional derivatives is also systematically extended to general distributions instead of only to some special distributions. 
    The new theory lays down a solid theoretical foundation for systematically and rigorously developing a fractional calculus of variations theory and a fractional PDE theory as well as their numerical solutions in subsequent works. 
\end{abstract}



\maketitle

\tableofcontents

 
\section{Introduction}\label{sec-1}

    Similar to the classical integer order calculus, the classical fractional order calculus also consists of two integral parts, namely, the fractional order integral calculus and the fractional order differential calculus. It is concerned with studying their properties/rules and the interplay between the two notions, which is often characterized by the so-called {\em Fundamental Theorem of Calculus}. Fractional calculus also has had a long history, which can be traced back to L'H\^{o}pital (1695), Wallis (1697), Euler (1738), Laplace (1812), Lacroix (1820), Fourier (1822), Abel (1823), Liouville (1832), Riemann (1847), Leibniz (1853), Gr\"unwald (1867), Letnikov (1868) and many others. We refer the reader to \cite{Guo, Podlubny, Samko} and the references therein for a detailed exposition about the history of the classical fractional calculus. 
 
    In the past twenty years fractional calculus and fractional (order) differential equations have garnered a lot of interest and attention both from the PDE community (in the name of nonlocal PDEs) and in the applied mathematics and scientific communities. Besides the genuine mathematical interest and curiosity, this trend has also been driven by intriguing scientific and engineering applications which give rise to fractional order differential equation models to better describe the (time) memory effect and the (space) nonlocal phenomena (cf. \cite{Du19,Guo,Hilfer,Kilbas, Meerschaert} and the references therein). It is the rise of these applications that revitalizes the field of fractional calculus and fractional differential equations and calls for further research in the field, including to develop new numerical methods for solving various fractional order problems. 
 
    Although a lot of progress has been achieved in the past twenty years in the field of fractional calculus and fractional differential equations, many fundamental issues remain to be addressed. For a novice in the field, one would immediately be clogged and confused by many (non-equivalent) definitions of fractional derivatives/operators. The immediate ramification of the situation is the difficulty for building/choosing ``correct" fractional models to study analytically and to solve numerically. Moreover, compared to the classical integer order calculus, the classical fractional calculus still has many missing components. For example, many basic calculus rules (such as product and chain rules) are not completely settled, the physical and geometric meaning of fractional derivatives are not fully understood, and a thorough characterization of the fractional differentiability seems still missing. Furthermore, at the differential equation (DE) level, the gap between the integer order and fractional order cases is even wider. For a given integer order DE, it would be accustomed for one to interpret the derivatives in the DE as weak derivatives and the solution as a weak solution. However, there is no parallel weak derivative and solution concept/theory in the fractional order case so far. As a result, various solution concepts and theories, which may not be equivalent, have been used for fractional order DEs, especially, for fractional (order) partial differential equations (FPDEs). The non-equivalence of solution concepts may cause confusions and misunderstanding of the underlying fractional order problem.
 
    The primary goal of this paper is to develop a weak fractional  differential calculus theory and its  corresponding fractional Sobolev space theory; all of which are analogous to the well understood theory and spaces in the integer order case. It provides a first step/attempt in achieving the overreaching goal of providing the missing components to, and to expand the reach of, the classical fractional calculus and fractional differential equations, which will be continued in subsequent works \cite{Feng_Sutton2,Feng_Sutton3}.
 
    The remainder of this paper is organized as follows. In Section \ref{sec-2}, we give a holistic summary/survey of the notions and results from classical fractional calculus. Much of this information can be found in the beginning chapters of the reference \cite{Samko}. In addition to reviewing the components needed later for developing a new fractional calculus theory, we begin to build a Rolodex of results related to the classical fractional calculus operators that will be utilized in later sections for the construction and analysis of our new weak fractional differential calculus theory. Special attention is given to drawing the differences between the components of the fractional calculus and ones in the integer order calculus. Among them, we especially mention inherent/intrinsic  properties such as non-locality and domain and directional dependence. 
    In Section \ref{sec-3}, we present two new perspectives of the classical fractional derivatives, the first one is to introduce the so-called {\em Fundamental Theorem of Classical Fractional Calculus} (FTcFC) and a new interpretation/definition of the classical Riemann-Liouville fractional derivatives. The second is to present a transformation point of view for the classical Riemann-Liouville fractional integrals and derivatives which allows us to interpret the fractional derivative as the rate of change of the transformed function in the ``frequency domain".
    In Section \ref{sec-4}, we first introduce the notion of   weak fractional derivatives using integration by parts and special test functions, which is analogous to the notion of integer order weak derivatives. It is proved that these weak derivatives inherit the fundamental properties of classical fractional derivatives and the generality allows us to avoid the need for numerous definitions as seen in the classical theory. After having proved a characterization result, we then establish a weak fractional differential calculus theory including a \textit{Fundamental Theorem of Weak Fractional Calculus} (FTwFC), product and chain rules, and integration by parts formulas. Many of these results and their proof techniques will look familiar to the informed reader because they are adapted and refined  from those used in the integer order weak differential calculus theory (cf. \cite{Adams, Evans, Meyers}). The desired differential calculus components are proved for both left and right weak fractional derivatives and covers both finite and infinite domain cases. 
    In Section \ref{sec-5}, we present a new definition of fractional Sobolev spaces using weak fractional  derivatives. Unlike the  existing definitions of fractional Sobolev spaces, the new definition is in the exact same spirit as that for the integer order Sobolev spaces defined  through the weak derivative. The rest of Section \ref{sec-5} is devoted to the study of new fractional Sobolev spaces and to the establishment of a fractional Sobolev space theory that is analogous to the theory found in the integer order case, which consists of proving a density/approximation theorem, extension theorems, a trace theorem, and various embedding theorems. Again, special attention is given to explaining the main differences between the new fractional Sobolev spaces and the integer Sobolev spaces. Moreover, a few connections between the new fractional Sobolev spaces and existing fractional Sobolev spaces are also established.
    In Section \ref{sec-6}, we extend the notion of weak fractional derivatives to distributions (or generalized functions). Unlike the existing fractional derivative definitions which only apply to a certain subset of distributions, we aim to define weak fractional derivatives for general distributions in both finite and infinite domain cases. Due to the pollution effect of the fractional derivatives, the main difficulty to overcome is to design a good domain extension for a given distribution, which is achieved by using a partition of the unity idea in this section. 
    Finally, the paper is concluded by a short summary and a few concluding remarks given in Section \ref{sec-7}.

\section{Elements of Classical Fractional Calculus Theory}\label{sec-2}
    In this section we collect many well-known facts about classical fractional calculus including the definitions and their properties. Most of them are stated without proofs, but relevant references will be cited as sources of the proofs for the interested reader, in particular, \cite{Samko} is cited as the primary reference for the missing proofs. 
    On the other hand, proofs will be provided for a few lesser known properties of classical fractional derivatives which play an important role in developing our fractional  weak derivative and calculus in the later sections. One such property is the behavior of classical fractional derivatives of compactly supported smooth functions. 
    Since the definitions of fractional integrals and derivatives are domain-dependent, it will be imperative for us to separate the cases when the domain is finite or infinite. In this section and in the later sections, we shall consider both finite interval $\Omega=(a,b)$ for $\infty < a < b < \infty$ and the infinite interval $\Omega=\R:=(-\infty, \infty)$. 

    Throughout this paper, $\Gamma: \R\to \R$ denotes the standard Gamma function and $\N$ stands for the set of all positive integers. In addition, $C$ will be used to denote a generic positive constant which may be different at different locations and $f^{(n)}$ denotes the $n$th order derivative of $f$ for a given positive integer $n$. Unless stated otherwise, throughout this section all integrals $\int_a^b \varphi(x)\, dx$ are understood in the sense of Riemann integrals.
    
    \subsection{Definitions of Classical Fractional Integrals and Derivatives}\label{sec-2.1}
        In this subsection, we recall the definitions and some basic properties of classical fractional integrals and derivatives, such as Riemann-Liouville, Liouville, Caputo, and Fourier fractional integrals and derivatives. 

        \subsubsection{\bf Definitions on a Finite Interval}\label{sec-2.1.1}
            Historically, the integral calculus was invented before the differential calculus in the classical Newton-Leibniz (integer) calculus  and the two are intimately connected through the well known {\em Fundamental Theorem of Calculus} (or Newton-Leibniz Theorem). It is interesting to note that the same is true for the classical fractional calculus. Indeed, in order to give a definition of fractional derivatives,  we first need to recall the definition of fractional integrals.

            \begin{definition}[cf. \cite{Samko}] \label{def2.1}
                Let $\sigma>0$ and $f:[a,b] \rightarrow \R$. The $\sigma$  order left Riemann-Liouville fractional integral of $f$ is defined by
                \begin{align} \label{left_RL_int}
                    {_{a}}{I}{^{\sigma}_{x}} f(x) : = \dfrac{1}{\Gamma(\sigma)} \int_{a}^{x} \dfrac{f(y)}{(x - y)^{1 - \sigma}} \, dy \qquad\forall x\in [a,b], 
                \end{align}
                and the $\sigma$ order right Riemann-Liouville fractional integral of $f$ is defined by 
                \begin{align}\label{right_RL_int}
                    {_{x}}{I}{^{\sigma}_{b}} f(x) : = \dfrac{1}{\Gamma(\sigma)} \int_{x}^{b} \dfrac{f(y)}{(y-x)^{1 - \sigma}}\, dy\qquad \forall x\in [a,b].
                \end{align}
            \end{definition}
            ${_{a}}{I}{^{\sigma}_{x}}$ and $ {_{x}}{I}{^{\sigma}_{b}}$ are respectively called the left and right Riemann-Liouville fractional integral operators. We also set ${^{-}}{I}{^{\sigma}} := {_{a}}{I}{^{\sigma}_{x}}$ and ${^{+}}{I}{^{\sigma}} := {_{x}}{I}{^{\sigma}_{b}}$. 

            \begin{remark}
                It is well known (\cite{Samko}) that both $ {_{a}}{I}{^{\sigma}_{x}}$ and ${_{x}}{I}{^{\sigma}_{b}}$ are convolution-type operators (with different kernel functions) because 
                \begin{align}\label{convolution}
                    {_{a}}{I}{^{\sigma}_{x}} f(x) = \kappa_a^\sigma * f(x),\qquad {_{x}}{I}{^{\sigma}_{b}} f(x) = \kappa_b^\sigma * f(x) ,
                \end{align} 
                where the convolution-kernel functions  are defined by 
                \begin{align} \label{kernel}
                    \kappa_a^\sigma (x)=\frac{1}{\Gamma(\sigma)} x_+^{\sigma-1},\quad \kappa_b^\sigma (x)= \frac{1}{\Gamma(\sigma)} (-x)_+^{\sigma-1} \qquad \forall x\in [a,b],
                \end{align} 
                where $x_+$ denotes the positive part of $x$. More integral operators can be defined by using other kernel functions. 
            \end{remark} 

            With the help of the above fractional integrals,  the definitions of two popular Riemann-Liouville fractional derivatives are given below.  

            \begin{definition}[cf. \cite{Samko}] \label{def2.2} 
                Let $n-1 < \alpha < n$ and $f : [a,b] \rightarrow \R$. The  $\alpha$ order left Riemann-Liouville fractional derivative of $f$ is defined by
                \begin{align}\label{left_RL_derivative}
                    {_{a}}{{D}}{_{x}^{\alpha}} f(x):&= \dfrac{d^{n}}{dx^n} \Bigl( {_{a}}{I}{^{n-\alpha}_{x}}f(x) \Bigr) \\
                    &=\dfrac{1}{\Gamma(n - \alpha)} \dfrac{d^{n}}{dx^n}\int_{a}^{x} \dfrac{f(y)}{(x - y)^{1  + \alpha -n} } \, dy    \qquad \forall x\in [a,b],
                \end{align} 
                and the  $\alpha$ order right Riemann-Liouville fractional derivative of $f$ is defined by
                \begin{align}\label{right_RL_derivative}
                    {_{x}}{{D}}{^{\alpha}_{b}}f(x) :&= (-1)^n  \dfrac{d^n}{dx^n} \Bigl({_{x}}{I}{^{n-\alpha}_{b}} f(x) \Bigr)  \\
                    &= \dfrac{(-1)^n}{\Gamma(n - \alpha)} \dfrac{d^n}{dx^n} \int_{x}^{b}  \dfrac{f(y)}{(y - x)^{1 + \alpha -n} }  \,dy     \qquad \forall x\in [a,b].
                \end{align}
                ${_{a}}{{D}}{_{x}^{\alpha}} $ and ${_{x}}{{D}}{^{\alpha}_{b}}$ are called the left and right Riemann-Liouville fractional derivative(or differential) operators, respectively. 
            \end{definition}
    
            Another fractional derivative notion is the Caputo fractional derivative, which is widely used in initial value problems of factional order ODEs; particularly for fractional differentiation in time.
 
            \begin{definition}[cf. \cite{Samko}] \label{def2.3}
                Let $n-1 < \alpha <n$ and $f :[a,b] \rightarrow \R$.  The $\alpha$ order left Caputo fractional derivative of $f$ is defined by
                \begin{align}\label{left_Caputo}
                    {^{C}_{a}}{{D}}{_{x}^{\alpha}} f(x) := \dfrac{1}{\Gamma(n- \alpha)} \int_{a}^{x} \dfrac{f^{(n)}(y)}{(x- y)^{ 1+ \alpha -n}} \,dy \qquad \forall x\in [a,b],
                \end{align}
                and the $\alpha$ order right Caputo fractional derivative of $f$ is defined by
                \begin{align}\label{right_Caputo}
                    {^{C}_{x}}{{D}}{_{b}^{\alpha}} f(x) := \dfrac{(-1)^n}{\Gamma(n - \alpha)}  \int_{x}^{b} \dfrac{f^{(n)}(y)}{(y - x)^{1+\alpha -n}} \, dy \qquad\forall x\in [a,b].
                \end{align}
            \end{definition}

            \begin{remark}
                It is easy to see that the definitions of the $\alpha$ order Caputo fractional derivatives require that $f^{(n)}$  exists almost everywhere. This very fact may cause some confusion, even misunderstanding about the Caputo fractional derivatives. The relationship between Riemann-Liouville and Caputo fractional derivatives is given by the following identities: 
                \begin{align}\label{RL_Caputo_l}
                    {^{C}_{a}}{{D}}{_{x}^{\alpha}} f(x) &:= {_{a}}{{D}}{_{x}^{\alpha}} f(x) -\sum_{k=0}^{n-1} \frac{f^{(k)}(a)}{\Gamma(k+1-\alpha)} (x-a)^{k-\alpha},\\
                    {^{C}_{x}}{{D}}{_{b}^{\alpha}} f(x) &:={_{x}}{{D}}{^{\alpha}_{b}} f(x) -\sum_{k=0}^{n-1}\frac{f^{(k)}(b)}{\Gamma(k+1-\alpha) } (b-x)^{k-\alpha}.
                    \label{RL_Caputo_r} 
                \end{align}
                Using the above relationship and the definition of the Riemann-Liouville fractional derivatives, we can derive the following ``weak" definitions of the Caputo fractional derivatives in the case $0<\alpha<1$: 
                \begin{align}\label{weak_RL_Caputo_l} 
                    {^{C}_{a}}{{D}}{_{x}^{\alpha}} f(x) &:= {_{a}}{{D}}{_{x}^{\alpha}} \bigl[ f(x) -f(a) \bigr] =\frac{d}{dx} \Bigl[ {_{a}}{I}{^{1-\alpha}_{x}} \Bigl( f(x)-f(a) \Bigr) \Bigr],\\
                    {^{C}_{x}}{{D}}{_{b}^{\alpha}} f(x) &:={_{x}}{{D}}{^{\alpha}_{b}} \bigl[ f(x)-f(b) \bigr] =-\frac{d}{dx} \Bigl[{_{x}}{I}{^{1-\alpha}_{b}} \Bigl( f(x) -f(b) \Bigr) \Bigr],
                    \label{weak_RL_Caputo_r} 
                \end{align}
                which do not require the existence of $f'(x)$, instead, they require the existence of $f(a)$ and $f(b)$, respectively.
            \end{remark} 

            Unlike the Riemann-Liouville and Caputo derivatives which use an integral operator to induce a fractional order derivative, a natural question is if a fractional derivative can be defined as a limit of some difference quotient similar to the definition of the integer order derivative. Although there have been a number of attempts in this direction (cf. \cite{Khalil2014}), for the purpose of having a complete calculus theory, we only recall the well-known Gr\"{u}nwald-Letnikov fractional derivatives as they are related to the Riemann-Liouville derivatives.
	
            \begin{definition} [cf. \cite{Samko}] 
                Let $0 < \alpha <1$ and $f:[a,b] \rightarrow \R$. The \textit{left} Gr\"{u}nwald-Letnikov fractional derivative of $f$ is defined by
                \begin{align*}
                    {^{GL}_{a}}{D}{^{\alpha}_{x}} f(x) : = \lim_{h \rightarrow 0^+} \dfrac{1}{h^{\alpha}} \sum_{k = 0 }^{[(x-a)/h]} \dfrac{(-1)^{k} \Gamma(1 + \alpha)}{\Gamma(k+1) \Gamma(\alpha - k +1)} f(x - kh) \qquad \forall x \in [a,b ]
                \end{align*}
                and the \textit{right} Gr\"{u}nwald-Letnikov fractional derivative of $f$ is defined by
                \begin{align*}
                    {^{GL}_{x}}{D}{^{\alpha}_{b}} f(x) : = \lim_{h \rightarrow 0^+} \dfrac{1}{h^{\alpha}} \sum_{k = 0 }^{[(b-x)/h]} \dfrac{(-1)^{k+1} \Gamma(1 + \alpha)}{\Gamma(k+1) \Gamma(\alpha - k +1)} f(x + kh) \qquad \forall x \in [a,b].
                \end{align*}
            \end{definition}

            Clearly, for the fractional derivative, the difference quotients are much more complicated. As alluded above, an equivalence between the Gr\"{u}nwald-Letnikov fractional derivatives and  the Riemann-Liouville derivatives will be stated in a subsequent subsection. 

        \subsubsection{\bf Definitions on an Infinite Interval}\label{sec-2.1.2} 
            We can also define fractional integrals over unbounded intervals in the same way; here we only consider the whole real line case, that is,  $(a,b)=(-\infty, \infty)$. The main reason we separate the infinite and finite interval cases is that there are two different definitions of fractional order derivatives in the infinite interval case which were proved to be equivalent. The first three definitions are direct generalizations of Definitions \ref{def2.1}--\ref{def2.3}.

            \begin{definition}[cf. \cite{Samko}] \label{def2.4} 
                Let $\sigma > 0$ and $f:\R \rightarrow \R$. The $\sigma$ order left Liouville fractional integral of $f$ is defined by
                \begin{align*}
                    {}{I}{^{\sigma}_{x}} f(x) : =   \dfrac{1}{\Gamma(\sigma)} \int_{-\infty}^{x} \dfrac{f(y)}{(x - y)^{1 - \sigma}} \, dy \qquad \forall x \in \R
                \end{align*}
                and the $\sigma$ order right Liouville fractional integral of $f$ is defined by 
                \begin{align*}
                    {_{x}}{I}{^{\sigma}} f(x) : =  \dfrac{1}{\Gamma(\sigma)} \int_{x}^{\infty} \dfrac{f(y)}{(y-x)^{1 - \sigma}}\, dy \qquad \forall x \in \R.
                \end{align*}
            \end{definition}


            \begin{definition}[cf. \cite{Samko}] \label{def2.5} 
                Let $n-1 < \alpha < n$ and $f:\R \rightarrow \R$. The $\alpha$ order left Liouville fractional derivative of $f$ is defined by 
                \begin{align*}
                    {}{{D}}{_{x}^{\alpha}} f(x):= \dfrac{1}{\Gamma(n - \alpha)} \dfrac{d^{n}}{dx^n}\int_{-\infty}^{x} \dfrac{f(y)}{ (x - y)^{1+\alpha -n}} \,dy \qquad \forall x \in \R
                \end{align*}
                and the $\alpha$ order right Liouville fractional derivative of $f$ is defined by
                \begin{align*}
                    {_{x}}{{D}}{^{\alpha}}f(x) :=  \dfrac{(-1)^{n}}{\Gamma(n - \alpha)} \dfrac{d^n}{dx^n} \int_{x}^{\infty} \dfrac{f(y)}{(y - x)^{1+ \alpha -n}} \,dy \forall x \in \R.
                \end{align*}
            \end{definition}

            \begin{definition}[cf. \cite{Samko}] \label{def2.6} 
                Let $n-1 <\alpha< n $ and $f :\R \rightarrow \R$. The $\alpha$ order left Caputo fractional derivative of $f$ is defined by 
                \begin{align*}
                    {^{C}}{{D}}{_{x}^{\alpha}} f(x) &:= \dfrac{1}{\Gamma(n- \alpha)} \int_{- \infty}^{x} \dfrac{f^{(n)}(y)}{(x- y)^{1+ \alpha -n} } \,dy \qquad \forall x \in \R
                \end{align*}
                and the $\alpha$ order right Caputo fractional derivative of $f$ is defined by
                \begin{align*}
                    {^{C}_{x}}{{D}}{^{\alpha}} f(x) &:= \dfrac{(-1)^n}{\Gamma(n - \alpha)}  \int_{x}^{\infty} \dfrac{f^{(n)}(y)}{ (y - x)^{1+\alpha -n} } \, dy \qquad \forall x \in \R.
                \end{align*}
            \end{definition}

            Note that in the above three definitions we do not write $\pm \infty$ in the operator notation to signify the interval; we leave the variable notation to distinguish between directions. It should also be noted that all integrals over the infinite domain are understood as standard improper integrals. 

            Similar to the finite interval case, we also can define the Gr\"unwald-Letnikov fractional derivatives for functions defined on the whole real line. In this case, notice that the sums are infinite sums in the above definition. 
 
            \begin{definition}[cf. \cite{Samko}] 
                Let $0 < \alpha <1$ and $f : \R \rightarrow \R$. The left Gr\"unwald-Letnikov fractional derivative of $f$ is defined by 
                \begin{align*}
                    {^{GL}}{D}{^{\alpha}_{x}} f(x) := \lim_{h \rightarrow 0^+} \dfrac{1}{h^{\alpha}} \sum_{k = 0}^{\infty} \dfrac{(-1)^{k}\Gamma(1+\alpha)}{\Gamma(k+1) \Gamma(\alpha - k +1)} f(x- kh) \qquad \forall x \in \R
                \end{align*}
                and the right Gr\"unwald-Letnikov fractional derivative of $f$ is defined by
                \begin{align*}
                    {^{GL}_{x}}{D}{^{\alpha}} f(x) := \lim_{h \rightarrow 0^+} \dfrac{1}{h^{\alpha}} \sum_{k = 0}^{\infty} \dfrac{(-1)^{k+1}\Gamma(1+\alpha)}{\Gamma(k+1) \Gamma(\alpha - k +1)} f(x+ kh) \qquad \forall x \in \R.
                \end{align*}
            \end{definition}

            Next, we recall another definition of fractional derivatives that are based on the Fourier transforms.  

            \begin{definition}[cf. \cite{Samko}] \label{def2.7}
	            Let $\alpha > 0$ and $f: \R \rightarrow \R$. The $\alpha$ order Fourier fractional derivative is defined by
	            \begin{align*}
	                {^{\mathcal{F}}}{{D}}{^{\alpha}}f(x) &:= \mathcal{F}^{-1} \left[ (i \xi)^{\alpha} \mathcal{F} [f](\xi)\right](x) \qquad \forall x \in \R
	            \end{align*}
                where $\mathcal{F}[\cdot]$ and $\mathcal{F}^{-1}[\cdot]$ denote respectively the Fourier transform and its inverse transform which are defined as follows: for any $x,\xi\in \R$
                \begin{align*}
                    \mathcal{F}[f](\xi) : = \int_{\R} e^{-i \xi x}f(x)\,dx,  \qquad \mathcal{F}^{-1}[f](x) := \int_{\R} e^{i\xi x}f(\xi)\,d\xi .
                \end{align*}
            \end{definition}

            \begin{remark}
	            The above Fourier fractional order derivative notion is based on the following well-known property of the Fourier transform:
                \[
                    \mathcal{F}[f^{(n)}](\xi)= (i \xi)^n \mathcal{F}[f](\xi), \qquad f^{(n)}(x)= \mathcal{F}^{-1}[ (i \xi)^n \mathcal{F}[f] ] (x)
                \]
                for any positive integer $n$.
            \end{remark}
   
    \subsection{Equivalences}\label{sec-2.2}
        In this subsection we state some conditions under which the above fractional derivative notions are equivalent. We note that this is only a subset of sufficient conditions. Other relationships between these fractional derivative notions can also be established (\cite{Samko}). 
    
        \begin{proposition}[cf. \cite{Samko}]\label{RLC} 
	        Let $0<\alpha<1$, then there hold the following identities and equalities:
            \begin{itemize}
                \item[{\rm (i)}] If $f \in AC([a,b])$, then there holds 
                \begin{align}\label{LRLC}
                    {_{a}}{{D}}{^{\alpha}_{x}} f(x) = \dfrac{1}{\Gamma (1 - \alpha)} \left(\dfrac{f(a)}{(x - a)^{\alpha}} + \int_{a}^{x} \dfrac{f'(y)}{(x - y)^{\alpha}} \,dy \right),\\ 
                    {_{x}}{{D}}{^{\alpha}_{b}} f(x) = \dfrac{1}{\Gamma (1 - \alpha)} \left( \dfrac{f(b)}{(b -x)^{\alpha}} - \int_{x}^{b} \dfrac{f'(y)}{(y - x)^{\alpha}} \, dy \right).\label{RRLC}
                \end{align}
                \item[{\rm (ii)}] If $f \in AC([a,b])$ and $f(a) =0$, then ${_{a}}{D}{^{\alpha}_{x}}f(x) = {^{C}_{a}}{D}{^{\alpha}_{x}} f(x)$. 
                \item[{\rm (iii)}] If $f \in AC([a,b])$ and $f(b) = 0$, then ${_{x}}{D}{^{\alpha}_{b}} f(x) = {^{C}_{x}}{D}{^{\alpha}_{b}}f(x).$
                \item[{\rm (iv)}] If $f\in C^{1}([a,b])$, then ${_{a}}{D}{^{\alpha}_{x}} f(x) = {^{GL}_{a}}{D}{^{\alpha}_{x}}f(x)$ and ${_{x}}{D}{^{\alpha}_{b}}f(x) ={^{GL}_{x}}{D}{^{\alpha}_{b}} f(x)$. 
            \end{itemize}
        \end{proposition}
        \medskip
        \begin{proposition}[cf. \cite{Ervin, Samko}] \label{EquivalencesonR} 
            Let $0<\alpha<1$, then there hold the following equivalence relationships: for any $f \in C^{1}_{0}(\R)$ 
            \begin{itemize}
                \item[{\rm (i)}] ${}{D}{^{\alpha}_{x}}f(x) = {^{C}}{D}{^{\alpha}_{x}}f(x)$ and ${_{x}}{D}{^{\alpha}}f(x) = {^{C}_{x}}{D}{^{\alpha}}f(x)$,
                \item[{\rm (ii)}] ${}{D}{^{\alpha}_{x}} f(x) = {^{\mathcal{F}}}{D}{^{\alpha}} f(x)$ and ${_{x}}{D}{^{\alpha}} f(x) = (-1)^{\alpha} {^{\mathcal{F}}}{D}{^{\alpha}}f(x)$.
            \end{itemize}
        \end{proposition}
        
        It is essential to our study that all derivative defined in Section \ref{sec-2.1} are equivalent on $C^{\infty}_{0}(\R)$. Specific mapping properties on this space will be explored in Section \ref{sec-2.7}


    \subsection{Elementary Properties}\label{sec-2.3}
        In this subsection we quote a number of the basic properties of fractional integrals and derivatives. These results will also illuminate an important characteristic of the fractional calculus, namely, the loss of some elementary properties and rules from the classical (integer) calculus. In order to be concise, we shall use $D^{\alpha}$ as a notation placeholder for any of the fractional derivatives defined in Section \ref{sec-2.1}.
    
        \begin{proposition}[cf. \cite{Samko}] 
    	    There hold the following properties for $D^{\alpha}$:
            \begin{itemize}
                \item[{\rm (a)}] Linearity: $D^{\alpha}(c_1f+ c_2g) = c_1 D^{\alpha}f + c_2 D^{\alpha}g$.
                \item[{\rm(b)}] Inclusivity: If $0 < \alpha < \beta <1$ and $D^{\beta} f$ exists in  $L^{1}$, then $D^{\alpha}f$ exists in $L^{1}$.
                 \item[{\rm (c)}] Consistency: For sufficiently smooth function $f$, there hold $\lim_{\alpha \rightarrow 0^+} D^{\alpha} f = f$ and $\lim_{\alpha \rightarrow 1^-} D^{\alpha} f = f'$;
                \item[{\rm (d)}] Semigroup: $D^{\alpha + \beta} \neq D^{\alpha} D^{\beta}$ in general, however, 
                    \begin{itemize}
                        \item[{\rm (i)}]  if $f\in C([a,b])$, then ${_{a}}{I}{^{\sigma}_{x}} {_{a}}{I}{^{\nu}_{x}} f = {_{a}}{I}{^{\sigma + \nu}_{x}} f$;
                        \item[{\rm (ii)}] if $f \in L^{p}(\R)$ and $\sigma + \nu <\frac{1}{p}$, then ${I}{^{\sigma}_{x}}{I}{^{\nu}_{x}}f = {I}^{\sigma + \nu}_{x} f$;
                        \item[{\rm (iii)}] let $f \in C^{1}(\R)$ and $0 < \alpha , \beta <1$ such that $\alpha + \beta <1$, then ${^{\mathcal{F}}}{D}{^{\alpha}} {^{\mathcal{F}}}{D}{^{\beta}} f = {^{\mathcal{F}}}{D}{^{\alpha + \beta}}f$;
                        \item[{\rm (iv)}] if $f \in C^{1}_{0}(\R)$ and $0 <\alpha ,\beta <1$ such that $\alpha + \beta <1$, then ${}{D}{^{\alpha}_{x}}{}{D}{^{\beta}_{x}} f = {}{D}{^{\alpha + \beta}_{x}} f$. 
                    \end{itemize}

            \end{itemize}
        \end{proposition}
        
        \begin{remark}
            Unlike the integer order case, the inclusivity result is nontrivial. We will prove this result for the Riemann-Liouville derivative(s) in Section \ref{sec-3.1.1}. In particular it is our inclination to employ several mapping properties of ${^{\pm}}{I}{^{\alpha}}$ and ${^{\pm}}{D}{^{\alpha}}$ in conjunction with the Fundamental Theorem of Fractional Calculus in Section \ref{sec-3.1}. The reader should then refer to the equivalences stated in Section \ref{sec-2.2} to verify the conditions for other derivative definitions. 
        \end{remark}

        We end this subsection by citing the following negative result on the possibility of having a clean product and chain rule.
    
        \begin{theorem}[cf. \cite{Tarasov}] \label{NoProductRule}
	        For all fractional derivatives, if ${D}^{\alpha}$ satisfies the product rule: ${D}^{\alpha} (fg) = \left({D}^{\alpha} f \right) g + f \left({D}^{\alpha} g \right)$, then there must have $\alpha = 1$. 
        \end{theorem}

        \begin{theorem}[cf. \cite{Tarasov2}] \label{NoChainRule}
            Fractional derivatives $D^{\alpha}$, which satisfy the chain rule $D^{\alpha}(f(g)) = D^{\alpha}f(g)D^{\alpha}g$, must have integer order $\alpha =1$.
        \end{theorem}

    \subsection{Product Rules, Chain Rules, and Integration By Parts}\label{sec-2.4}
        In this subsection we first recall a set of product rules for the Riemann-Liouville fractional derivatives, we then derive a set of chain rules based on these product rules. As one may expect based on the two previous negative results, both product rules and chain rules become much more complicated for fractional order derivatives. 
    
        \begin{theorem}[cf. \cite{Bassam, Samko}] \label{ProductRule}
            Let $0<\alpha <1$ and $f \in AC([a,b])$ and $g \in C^{m+1}((a,b))$. Then there hold
            \begin{align}\label{LPR}
                {_{a}}{D}{^{\alpha}_{x}} (fg)(x)  = \sum_{k=0}^{m} \dfrac{\Gamma(1 + \alpha)}{\Gamma(k +1) \Gamma(1 - k + \alpha)} {_{a}}{D}{^{\alpha-k}_{x}}f(x) D^{k} g(x) 
                + {^{-}}{R}{_{m}^{\alpha}}(f,g)(x), \\
                \label{RPR}
                {_{x}}{D}{^{\alpha}_{b}} (fg)(x) = \sum_{k=0}^{m} \dfrac{\Gamma(1 + \alpha)}{\Gamma(k +1) \Gamma(1 - k + \alpha)} {_{x}}{D}{^{\alpha-k}_{b}}f(x) D^{k} g(x)+ {^{+}}{R}{_{m}^{\alpha}}(f,g)(x), 
            \end{align}
            for every $x \in [a,b]$ where
            \begin{align}\label{LeftProductRuleRemainder}
                {^{-}}{R}{^{\alpha}_{m}}(f,g)(x) = \dfrac{(-1)^{m+1} }{m! \Gamma(- \alpha)} \int_{a}^{x} \dfrac{f(y)}{(x- y)^{1 + \alpha}} \, dy \int_{y}^{x} g^{(m+1)}(z) (x-z)^{m} \, dz , \\
             \label{RightProductRuleRemainder}
                {^{+}}{R}{^{\alpha}_{m}} (f,g)(x) = \dfrac{(-1)^{m+1}}{m!\Gamma(-\alpha)} \int_{x}^{b} \dfrac{f(y)}{(y-x)^{1+\alpha}}\,dy \int_{x}^{y} g^{(m+1)}(z) (z-x)^{m}\,dz,
            \end{align}
            where and throughout this paper we use the convention ${_{a}}{D}{^{\alpha-k}_{x}}:= {_{a}}{I}{^{k-\alpha}_{x}}$ and ${_{x}}{D}{^{\alpha-k}_{b}}:= {_{x}}{I}{^{k-\alpha}_{b}}$ for all $k>\alpha$. 
        \end{theorem}
    
        \begin{corollary}
    	    Let $0<\alpha <1$ and $f,g \in AC([a,b])$. Then there hold
    	    \begin{align}\label{LPR_0}
    	        {_{a}}{D}{^{\alpha}_{x}} (fg)(x)  =   g(x) {_{a}}{D}{^{\alpha}_{x}}f(x)  + {^{-}}{R}{_{0}^{\alpha}}(f,g)(x), \\
                \label{RPR_0}
    	        {_{x}}{D}{^{\alpha}_{b}} (fg)(x) =   g(x) {_{x}}{D}{^{\alpha}_{b}}f(x) + {^{+}}{R}{_{0}^{\alpha}}(f,g)(x), 
    	    \end{align}
    	    for every $x \in [a,b]$ where 
    	    \begin{align}\label{LeftProductRuleRemainder_0}
    	        {^{-}}{R}{^{\alpha}_{0}}(f,g)(x) = \dfrac{-1}{\Gamma(- \alpha)} \int_{a}^{x} \dfrac{f(y) [g(x)-g(y)]}{(x- y)^{1 + \alpha}} \, dy , \\
    	        \label{RightProductRuleRemainder_0}
    	        {^{+}}{R}{^{\alpha}_{0}} (f,g)(x) = \dfrac{-1}{\Gamma(-\alpha)} \int_{x}^{b} \dfrac{f(y) [g(x)-g(y)]}{(y-x)^{1+\alpha}}\,dy .
    	    \end{align}
        \end{corollary}
   
        \begin{remark}
	        (a) Theorem  \ref{NoProductRule} implies that the standard Leibniz rule does not hold for fractional derivatives,  so the above product rules are probably the neatest one can have for fractional derivatives (of Riemann-Liouville type). 
	
	        (b) It is not known if similar formulas to \eqref{LPR}  and \eqref{RPR} hold for fractional derivatives defined on $\R$. On the other hand, under stronger assumptions of analyticity and higher order Riemann-Liouville differentiability, a formula of the type ${}{D}{^{\alpha}_{x}}(fg) = \sum_{k=-\infty}^{\infty} {}{D}{^{\alpha - \beta -k}}f {}{D}{^{\beta +k}_{x}} g$ is known to hold \cite{Podlubny, Samko}. However, for our purposes, we are interested in the minimal smoothness assumptions on the function(s), hence, we prefer the product rules with remainder formula.
    
            (c) The positions of $f$ and $g$ in \eqref{LPR_0} and \eqref{RPR_0} can be swapped, averaging those formulas then yields the following symmetric versions of fractional product rules:
            \begin{align}\label{SPR}
                {^{\pm}}{D}{^{\alpha}} (fg)(x) & =  \frac12 \Bigl[  {^{\pm}}{D}{^{\alpha}}f(x)\,  g(x) + f(x)\, {^{\pm}}{D}{^{\alpha}}\,g(x) \Bigr] \\
                &\quad +  \frac12 \Bigl[ {^{\pm}}{R}{_{0}^{\alpha}}(f,g)(x) +   {^{\pm}}{R}{_{0}^{\alpha}}(g,f)(x)  \Bigr]. \nonumber
            \end{align}
   
            (d) It is easy to check that for all $x\in [a,b]$ there holds
            \[
                {^{\pm}}{R}{^{\alpha}_{0}}(f,g)(x) - {^{\pm}}{R}{^{\alpha}_{0}} (g,f)(x) = g(x) \, {^{\pm}}{I}{^{-\alpha}}f (x) - f(x)\, {^{\pm}}{I}{^{-\alpha}}g (x).
            \]
        \end{remark}

        Utilizing \eqref{LPR_0} and \eqref{RPR_0} we can obtain some chain rules for the fractional Riemann-Liouville derivatives. 

        \begin{theorem}\label{thm_CR}
	        Suppose that $\varphi\in C^1(\R)$ such that $\varphi(0)=0$ and $f\in AC([a,b])$. Then there hold
	        \begin{align}\label{chain_rule}
	            {^{\pm}}{D}{^{\alpha}}  \varphi(f)(x) =   \frac{\varphi(f)(x)}{f(x)} {^{\pm}}{D}{^{\alpha}} f(x) + {^{\pm}}{R}{_{0}^{\alpha}}\Bigl(f, \frac{\varphi(f)}{f} \Bigr)(x) \qquad \forall x \in [a,b].
	        \end{align}
	    \end{theorem}

        \begin{proof}
	        {\em Step 1:}  Assume $\varphi(f) = f^m$ for any integer $m\geq 1$.  By \eqref{LPR_0} and \eqref{RPR_0} we get 
	        \begin{align*}
	            {^{\pm}}{D}{^{\alpha}}  f^m  =  	{^{\pm}}{D}{^{\alpha}}  (f f^{m-1})
	            &=  {^{\pm}}{D}{^{\alpha}}f\,  f^{m-1} + {^{\pm}}{R}{_{0}^{\alpha}}(f, f^{m-1}) \\
	            &= \frac{\varphi(f)}{f} {^{\pm}}{D}{^{\alpha}} f + {^{\pm}}{R}{_{0}^{\alpha}}\Bigl(f, \frac{\varphi(f)}{f} \Bigr).
	        \end{align*}
	        Hence, \eqref{chain_rule} holds when $\varphi$ is a power function. 

            {\em Step 2:} Assume $\varphi(f)= a_r f^r+a_{r-1} f^{r-1} +\cdots + a_1 f$ for any integer $r\geq 1$. Again, by \eqref{LPR_0} and \eqref{RPR_0}  we immediately obtain
            \begin{align*}
                {^{\pm}}{D}{^{\alpha}} \varphi(f)  &=  	{^{\pm}}{D}{^{\alpha}} \bigl(f\, (a_r f^{r-1} +a_{r-1} f^{r-2} +\cdots + a_1) \bigr) \\
                &=  {^{\pm}}{D}{^{\alpha}}f\,  \bigl(a_r f^{r-1} +\cdots + a_1\bigr) + {^{\pm}}{R}{_{0}^{\alpha}}(f, a_r f^{r-1} +\cdots + a_1 ) \\
                &= \frac{\varphi(f)}{f} {^{\pm}}{D}{^{\alpha}} f + {^{\pm}}{R}{_{0}^{\alpha}}\Bigl(f, \frac{\varphi(f)}{f} \Bigr).
            \end{align*}
   
            {\em Step 3:}  For any $\varphi\in C^1(\R)$ such that $\varphi(0)=0$, By Weierstrass Theorem we know that there exists a sequence of polynomials $\{P_r(s)\}_{r = 1}^{\infty}$ with $P_r(0)=0$ such that $\lim_{r\to \infty } P_r =\varphi$ uniformly on every finite closed interval on $\R$.  By the result of {\em Step 2} we have
            \begin{align}\label{eq2.30}
                {^{\pm}}{D}{^{\alpha}} P_r(f) = \frac{P_r(f)}{f} {^{\pm}}{D}{^{\alpha}} f + {^{\pm}}{R}{_{0}^{\alpha}}\Bigl(f, \frac{P_r(f)}{f} \Bigr).
            \end{align}
            On noting that ${^{\pm}}{D}{^{\alpha}}$ are linear operators and ${^{\pm}}{R}{_{0}^{\alpha}}(\cdot, \cdot) $ are bilinear operators, setting $r\to 0$ in \eqref{eq2.30} yields 
            \begin{align*} 
                {^{\pm}}{D}{^{\alpha}} \varphi(f)(x) = \frac{\varphi(f)(x)}{f(x)} {^{\pm}}{D}{^{\alpha}} f(x) + {^{\pm}}{R}{_{0}^{\alpha}}\Bigl(f, \frac{\varphi(f)}{f} \Bigr)(x) \qquad \forall x \in [a,b].
            \end{align*}
            The proof is complete. 
        \end{proof} 

        \begin{remark}
	        (a) At the end of the proof we have used  $\lim_{r\to \infty } {^{\pm}}{D}{^{\alpha}} P_r(f) = {^{\pm}}{D}{^{\alpha}} \varphi(f) $ which can be easily checked using the relationship between the Riemann-Liouville derivatives and Caputo derivatives, the definition of Caputo derivatives, and the uniform convergence of $\lim_{r\to \infty } P_r =\varphi$.
	
	        (b) The expression $\frac{\varphi(f)}{f}$ is understood as a limit at point $x\in [a,b]$ where if $f(x)=0$, the limit is well defined since $f\in AC([a,b])$ and  $\varphi'(0)$ exists.  
	
	        (c) The $C^1$ regularity on $\varphi$ can be relaxed, for example, into  $\varphi\in AC([a,b])$. 
	
	        (d) In case $\varphi(0)\neq 0$, applying \eqref{chain_rule} to $\Psi(s):=\varphi(s)-\varphi(0)$, a chain rule can be easily obtained in that case, which will involve the unbounded (kernel) functions $\kappa^{\alpha}_{-}(x):=(x-a)^{\alpha -1}$ and $\kappa^{\alpha}_{+}(x):=(b-x)^{\alpha -1}$.
        \end{remark}

        The next theorem collects a set of integration by parts formulas for fractional integral and derivative operators, which  are as clean as one could expect. Details are omitted, but the interested reader can derive each of the following formulas using Fubini-Tonelli's theorem in conjunction with $L^p$ mapping properties from Section \ref{sec-2.5} and classical integration by parts.
    
        \begin{theorem}[cf. \cite{Samko}]\label{IBP}
    	    Let $f \in L^{p}((a,b))$ and $g\in L^{q}((a,b))$. If $p^{-1} + q^{-1} \leq 1 + \sigma$, $p \geq 1$, $q \geq 1$ (and $p^{-1} + q^{-1}= 1 + \sigma$  when $p\neq 1$ and $q\ne 1$) and $0<\sigma<1$, then for finite interval $(a,b)$, there holds 
    	    \begin{align}\label{IBP1}
    	        \int_{a}^{b} f(x) {_{a}}{I}{^{\sigma}_{x}} g(x) \, dx = \int_{a}^{b}  {_{x}}{I}{^{\sigma}_{b}} f (x) g(x)  \, dx.
    	    \end{align}
    	    If $p^{-1}+ q^{-1} = 1 + \sigma$, $p >1$, and $q > 1$ and $0<\sigma <1$, then there holds 
    	    \begin{align}\label{IBP2}
    	        \int_{\R} f(x) {I}{^{\sigma}_{x}} g(x) \, dx = \int_{\R}   {_{x}}{I}{{^\sigma}} f (x) g(x)  \, dx.
    	    \end{align}
    	    Moreover, suppose $0<\alpha <1$, $f \in AC([a,b])$,   $g \in L^{p}((a,b))$ for $\frac{1}{1-\alpha} < p \leq \infty$ and ${^{\pm}}{{D}}{^{\alpha}} g \in L^1((a,b))$, then there holds 
    	    \begin{align}\label{IBP3}
    	        \int_{a}^{b} f(x) {^{\pm}}{{D}}{^{\alpha}} g(x) \, dx = \int_{a}^{b} {^{\mp}}{{D}}{^{\alpha}} f(x) g(x)\, dx.
    	    \end{align}
    	    and \eqref{IBP3} also holds if $f \in L^{1}((a,b))$, ${^{\mp}}{I}{^{1-\alpha}}f \in AC([a,b])$ and $g \in C^{\infty}_{0}([a,b])$. Furthermore, \eqref{IBP3} holds for $(a,b)$ being replaced by $\R$ under the same assumptions. 
    	%
        \end{theorem}

         \subsection{Mapping Properties of ${^{\pm}}{I}{^{\alpha}}$ and ${^{\pm}}{D}{^{\alpha}}$ on $L^p$ Spaces} \label{sec-2.5}
        In this subsection we cite the well known mapping properties of the Riemann-Liouville integral operator ${^{\pm}}{I}{^{\alpha}}$ on $L^p$ spaces.  These mapping properties are fundamental to understanding classical fractional integral and differential operators. It should be noted that there is a parallel set of mapping properties of ${^{\pm}}{I}{^{\alpha}}$ on H\"older spaces $C^\lambda$ (cf. \cite{Samko}). We do not list them here to save space and because they will not be used in this paper.             

        \begin{theorem}[cf. \cite{Samko}]\label{LpMappings}
	        Let $0 < \alpha <1$, $1\leq p < \infty $, and $q = \frac{p}{1-\alpha p}$. Then the following properties hold:
	        \begin{itemize}
		        \item[{\rm (a)}] ${_{a}}{I}{^{\alpha}_{x}}, {_{x}}{I}{^{\alpha}_{b}} : L^{p}((a,b)) \rightarrow L^{p}((a,b))$ are bounded and satisfy $$ \Bigl\{ \left\|{_{x}}{I}{^{\alpha}_{b}} f\right\|_{L^{p}((a,b))},\,\, \left\|{_{a}}{I}{^{\alpha}_{x}} f\right\|_{L^{p}((a,b))} \Bigr\} \leq \dfrac{(b-a)^{\alpha}}{\alpha \Gamma(\alpha)} \|f\|_{L^{p}((a,b))}.$$
		        \item[{\rm (b)}] If $1 < p < \frac{1}{\alpha}$, then ${_{a}}{I}{^{\alpha}_{x}}$, ${_{x}}{I}{^{\alpha}_{b}}:L^{p}((a,b)) \rightarrow L^{q}((a,b))$ are bounded and satisfy $$ \Bigl\{ \|{_{a}}{I}{^{\alpha}_{x}} f\|_{L^{r}((a,b))}, \,\, \| {_{x}}{I}{^{\alpha}_{b}}\|_{L^{r}((a,b))} \Bigr\} \leq C\|f\|_{L^{p}((a,b))} \qquad \forall \, 1\leq r \leq q.$$ 
		        \item[{\rm (c)}] ${^{}}{I}{^{\alpha}_{x}}, {_{x}}{I}{^{\alpha}} : L^{p}(\R) \rightarrow L^{q}(\R)$ are bounded if and only if $1 < p <  \frac{1}{\alpha}$.               
		        \item[{\rm (d)}] If $f \in AC([a,b])$, then ${_{a}}{D}{^{\alpha}_{x}}f, {_{x}}{D}{^{\alpha}_{b}}f \in L^{r}((a,b))$ for $1 \leq r < \frac{1}{\alpha} $.
            \end{itemize}
        \end{theorem}
        \medskip
        \begin{remark}
	        For any $\alpha> 0$, when $\alpha p>1$,  it can be shown that ${_{a}}{I}{^{\alpha}_{x}}$ and  ${_{x}}{I}{^{\alpha}_{b}}$ are bounded from $L^p((a,b))$ to $C^{\alpha-\frac{1}{p}}((a,b))$ and ${_{a}}{I}{^{\alpha}_{x}} \varphi(x) = o\bigl(|x-a|^{\alpha-\frac{1}{p}}\bigr)$ as $x\to a^+$, ${_{x}}{I}{^{\alpha}_{b}} \varphi(x) = o\bigl(|b-x|^{\alpha-\frac{1}{p}} \bigr)$ as $x\to b^-$. We refer the reader to \cite[Theorem 3.6]{Samko} for the detailed statement and a proof.
        \end{remark}

        At this point, we can ask (and answer) a natural question: {\em what functions would be $\alpha$ order Riemann-Liouville differentiable?} Assume $0<\alpha <1$, by the definition we have ${^{\pm}}{D}{^{\alpha}}f(x) = \frac{d}{dx} {^{\pm}}{I}{^{1-\alpha}}f(x)$. Hence, in order for the right-hand side function to exist almost everywhere in $(a,b)$, we must require ${^{\pm}}{I}{^{1-\alpha}}f$ to be absolutely continuous. Hence, it is sufficient to know what function $f$ would make ${^{\pm}}{I}{^{1-\alpha}}f$ be Lipschitz. 


        The following theorem gives a sufficient condition for fractional differentiability.

        \begin{theorem}[cf. \cite{Samko}]\label{differentibility}
	        Let $0 < \alpha <1$.  Then $f\in {^{-}}{I}{^{\alpha}} [L^1((a,b))]$ (resp. $f\in {^{+}}{I}{^{\alpha}} [L^1((a,b))]$) if and only if ${^{-}}{I}{^{1-\alpha}}f$ (resp. ${^{+}}{I}{^{1-\alpha}}f$) belongs to $AC([a,b])$ and ${^{-}}{I}{^{1-\alpha}}f(a)=0$ (resp. ${^{+}}{I}{^{1-\alpha}}f(b)=0$). Where ${^{\pm}}{I}{^{\alpha}} [L^1((a,b))]$ denote the range spaces of $L^1((a,b))$ under the mappings/operators ${^{\pm}}{I}{^{\alpha}}$.
        \end{theorem}

        \begin{corollary}
	        ${^{-}}{D}{^{\alpha}}f$ exists almost everywhere for any $f\in{^{-}}{I}{^{\alpha}} [L^1((a,b))]$ and ${^{+}}{D}{^{\alpha}}f$ exists for any $f\in{^{+}}{I}{^{\alpha}} [L^1((a,b))]$. 
        \end{corollary} 

        \begin{remark}
	        The condition $f\in{^{\pm}}{I}{^{\alpha}} [L^1((a,b))]$ not only imply that ${^{+}}{I}{^{1-\alpha}}f\in AC([a,b])$, but also infers that ${^{-}}{I}{^{1-\alpha}}f(a)=0$ and ${^{+}}{I}{^{1-\alpha}}f(b)=0$. Hence, they are only sufficient conditions for $\alpha$ order differentiability. 
        \end{remark}

   \subsection{Kernel Spaces}\label{sec-2.6}
        As expected, the kernel functions of ${^{\pm}}{I}{^{\alpha} }$ and ${^{\pm}}{D}{^{\alpha} }$ play an important role in the classical fractional calculus theory because the behavior of fractional integrals and derivatives  is heavily influenced by the behavior of the kernel functions. As a brief glimpse into some of the behavior characteristics of the kernel functions, we cite the following results. 
   
        \begin{lemma}[cf. \cite{Samko}]\label{Abel}
            Let $0 < \alpha <1$. Then ${^{\pm}}{I}{^{\alpha} } f = 0$ has only the trivial solution in $L^1((a,b))$. Hence, the kernel space of ${^{\pm}}{I}{^{\alpha} }$ is trivial.
        \end{lemma}
    
        \medskip
        \begin{lemma}[cf. \cite{Samko}]
   	        Let $0 < \alpha,\beta <1$ and $f(x) = (x-a)^{-\beta}$ and $g(x) = (b-x)^{- \beta}$. Then 
   	        \begin{align}\label{d_formula1}
   	            {_{a}}{D}{^{\alpha}_{x}}f(x) &=  \dfrac{\Gamma(1-\beta)}{\Gamma(1-\beta -\alpha)} (x-a)^{-(\beta +\alpha)}, \\
   	            {_{x}}{D}{^{\alpha}_{b}} g(x) &= \dfrac{\Gamma(1 - \beta)}{\Gamma(1- \beta -\alpha)} (b-x)^{-(\beta +\alpha)}. \label{d_formula2}
   	        \end{align}
        \end{lemma}
 
   	    \begin{corollary} 
   		    Let $0 < \alpha <1$, $f(x) = (x-a)^{\alpha -1}$ and $g(x) = (b-x)^{\alpha -1}$. Then ${_{a}}{D}{^{\alpha}_{x}} f(x) \equiv 0$ and ${_{x}}{D}{^{\alpha}_{b}}g(x) \equiv 0$. The converse statement is also true.
   	    \end{corollary}
   	    \medskip
   	    \begin{remark}
   		    (a) It should be noted that ${_{a}}{D}{^{\alpha}_{x}} (x-a)^{\alpha -1} \equiv 0$ and ${_{x}}{D}{^{\alpha}_{b}} (b-x)^{\alpha -1} \equiv 0$ because $\lim_{z\to 0} \Gamma(z)^{-1}=0$; hence the Gamma function plays a major role in the mapping properties of the Riemann-Liouville derivative. 
   		
   		    (b) In the Riemann-Liouville fractional calculus, the kernel functions  $\kappa^{\alpha}_{-}(x):=(x-a)^{\alpha -1}$ and $\kappa^{\alpha}_{+}(x):=(b-x)^{\alpha -1}$ play the role that constant functions do in the Newton-Leibniz calculus. This point of view will be further elaborated in the next section. 
   	    \end{remark}
   	
   	    The above corollary motivates us to introduce the following two kernel spaces. 
   	    \begin{align}\label{kernel_space} 
   	        &\mathcal{K}^-:= \{f: \, {_{a}}{D}{^{\alpha}_{x}} f =0\} = \mbox{span}\{\kappa^{\alpha}_{-}\}, \\
   	        &\mathcal{K^+}:= \{ f:\, {_{x}}{D}{^{\alpha}_{b}}f = 0 \} = \mbox{span}\{\kappa^{\alpha}_{+}\}.
   	    \end{align}
   
    \subsection{Action on Smooth Functions with Compact Support}\label{sec-2.7}
        The action of the Riemann-Liouville integral and differential operators on smooth functions with 
        compact support is of special interest for our study in this paper. The need for understanding 
        these behaviors will become evident in the later sections. For now, we simply establish a few results 
        about this action which will be used later. 
        
        \begin{proposition} \label{Smooth}
            Let $\sigma >0$, $\Omega\subseteq \R$. If $\varphi \in C^{\infty}_{0} ( \Omega)$, then for any integer 
            $n\geq 1$ there holds
            \begin{equation}\label{eq2.18}
                \dfrac{d^n}{dx^n} {}{I}{^{\sigma}} \varphi(x) = {}{I}{^{\sigma}} \frac{d^n\varphi}{dx^n}(x) \qquad \forall x\in \Ome. 
            \end{equation}
            Consequently, for all $0<\alpha<1$ and $n\geq 1$ there holds
            \begin{equation}\label{eq2.19} 
                \dfrac{d^{n-1}}{dx^{n-1}} {}{D}{^{\alpha}} \varphi(x) = {}{I}{^{1-\alpha}} \frac{d^n\varphi}{dx^n}(x) \qquad \forall x\in \Ome. 
            \end{equation}
        \end{proposition}
        
        \begin{proof}
            We only prove the assertion for $\Omega = (a,b)$, $0 < \sigma <1$, and the left direction because all other cases follow similarly. In particular, the case for $\Omega =\R$ follows by definition of the indefinite integral and the equivalence of ${^{\mathcal{F}}}{D}{^{\alpha}}$ to ${^{\pm}}{D}{^{\alpha}}$.
            
            Noticing that ${_{a}}{D}{^{\alpha}_{x}} \varphi = {^{C}_{a}}{D}{^{\alpha}_{x}} \varphi$ for any $0<\alpha<1$. Explicitly for $\alpha=1-\sigma$, we get
           \begin{align*}
                \dfrac{1}{\Gamma(\sigma)} \dfrac{d}{dx} \int_{a}^{x} \dfrac{\varphi(y)}{(x-y)^{1 - \sigma}} \, dy = \dfrac{1}{\Gamma(\sigma)} \int_{a}^{x} \dfrac{\varphi'(y)}{(x-y)^{1 - \sigma}} \, dy
            \end{align*}
            which is exactly 
            \begin{align*}
                \dfrac{d}{dx} {}{I}{^{\sigma}} \varphi(x)  = {}{I}{^{\sigma}} \varphi '(x)\qquad \forall 
                x\in \Ome.
            \end{align*}
            Hence \eqref{eq2.18} holds for $n=1$.  On noting that $\varphi^{(n)}:=\frac{d^n\varphi}{dx^n} \in C^{\infty}_{0} ( \Omega)$, then \eqref{eq2.18} follows from an application of the induction argument. \eqref{eq2.19} immediately follows from \eqref{eq2.18} with $\sigma=1-\alpha$ and the definition of $D^\alpha$. The proof is complete.      
         \end{proof}

        \begin{corollary} \label{coro_Smooth}
            Let $\Omega \subseteq \R$, $0 < \alpha <1$. If $\varphi \in C^{\infty}_{0}(\Omega)$, then ${^{\pm}}{D}{^{\alpha}}\varphi \in C^{\infty} (\Omega)$. 
        \end{corollary}
   
        \begin{proof}
            The assertion is an immediate consequence of \eqref{eq2.19}. 
       \end{proof}
        
        Next, let us have a closer look at the support and the tail behavior of ${^{\pm}}{D}{^{\alpha}} \varphi$ for $\varphi \in C^{\infty}_{0}(\R)$ so that $\mbox{supp}(\varphi)\subset (a,b)$. To that end, a direct computation yields 
        \begin{align}\label{eq2.20a}
            {^{-}}{{I}}{^{\sigma}_{x}} \varphi (x) = 
                \begin{cases} 
                    0 &\text{if } x \in (-\infty , a ),\\
                    {_{a}}{{I}}{^{\sigma}_{x}} \varphi(x) & \text{if } x \in [a,b],\\
                    L(x) & \text{if } x \in (b ,\infty),
                \end{cases}
        \end{align}
        where
        \begin{align}\label{2.20b}
            L(x) = \dfrac{1}{\Gamma( \sigma)} \int_{a}^{b} \dfrac{\varphi(y)}{(x -y)^{1 - \sigma}} \, dy.
        \end{align}
        Taking the first derivative and letting $ \sigma=1-\alpha$ yields
        \begin{align} \label{LAction}
            {}{{D}}{^{\alpha}_{x}} \varphi(x) = 
                \begin{cases} 
                    0 &\text{if } x \in ( -\infty , a),\\
                    {_{a}}{{D}}{^{\alpha}_{x}} \varphi(x) & \text{if } x \in  [a,b],\\
                    L'(x) &\text{if } x \in  (b, \infty),
                \end{cases}
        \end{align}
        where 
        \begin{align}\label{LPollution}
            L'(x) = \dfrac{1}{\Gamma(1- \alpha)} \dfrac{d}{dx} \int_{a}^{b} \dfrac{\varphi(y)}{(x-y)^{\alpha}} \, dy. 
        \end{align}
        Assuming that $\varphi \not\equiv 0$, it is easy to see that $L'(x) \neq 0$. 
        
        Similarly, taking the right fractional integral of $\varphi$, we obtain 
        \begin{align}\label{eq2.25a}
            {_{x}}{{I}}{^{\sigma}} \varphi (x) = 
                \begin{cases} 
                    R(x) &\text{if } x \in (-\infty , a ),\\
                    {_{x}}{{I}}{^{\sigma}_{b}} \varphi(x) & \text{if } x \in [a,b],\\
                    0& \text{if } x \in (b ,\infty),
                \end{cases}
        \end{align}
        where
        \begin{align}\label{eq2.25b}
            R(x) =- \dfrac{1}{\Gamma(\sigma)} \int_{a}^{b} \dfrac{\varphi(y)}{(y -x)^{1-\sigma}} \, dy.
        \end{align}
        Take the first derivative and letting $\sigma = 1 - \alpha$,
        \begin{align}\label{RAction}
        {_{x}}{{D}}{^{\alpha}} \varphi(x) = 
            \begin{cases} 
                R'(x) &\text{if } x \in ( -\infty , a),\\
                {_{x}}{{D}}{^{\alpha}_{b}} \varphi(x) & \text{if } x \in [a,b],\\
                0 &\text{if } x \in  (b, \infty),
            \end{cases}
        \end{align}
        where 
        \begin{align}\label{RPollution}
            R'(x) = - \dfrac{ 1}{\Gamma(1- \alpha)}  \dfrac{d}{dx} \int_{a}^{b} \dfrac{\varphi(y)}{(y - x)^{\alpha}} \, dy .
        \end{align}
        It is easy to check that $R' \neq 0$ if $\varphi\not\equiv 0$. To sum up, we have
          
        \begin{proposition} \label{Pollution}
            If $\varphi \in C^{\infty}_{0} (\R)$ with $\supp(\varphi) \subset (a,b)$, then $\supp({}{{D}}{^{\alpha}_{x}} \varphi)\subset (a, \infty)$ and $\supp({_{x}}{{D}}{^{\alpha}} \varphi)\subset (-\infty, b)$. 
        \end{proposition}

        \begin{remark}
    	    (a) Riemann-Liouville fractional differential (and integral) operators have a pollution effect on the support when acting on functions in $C^\infty_0 (\Ome)$. Left derivatives pollute the support to the right and right derivatives pollute the support to the left. 
    	    This pollution effect is a consequence of the nonlocal characteristics of fractional order differential and integral operators; in particular, the ``memory" effect.  
    	 
    	    (b) When $x\to \pm \infty$, the integrands in \eqref{LPollution} and \eqref{RPollution} are shrinking. Moreover, $\lim_{x \rightarrow \infty} |L'(x)| = 0$ and $\lim_{x \rightarrow -\infty} |R'(x)| = 0$. 
        \end{remark}
       
        The next theorem states an integrability property of $D^\alpha \varphi$ for $\varphi \in C^{\infty}_{0} (\R)$. 
        
        \begin{theorem}
           Let $0<\alpha< 1$.  If $\varphi \in C^{\infty}_{0}(\Omega)$, then ${^{\pm}}{D}{^{\alpha}} \varphi \in L^{p}(\Omega)$ for each $1 \leq p \leq \infty$.
        \end{theorem}

        \begin{proof} 
            We only give a proof for $\Omega = \R$ and the left derivative 
            ${}{{D}}{^{\alpha}_{x}} \varphi$ because the other cases follow similarly.
            
            First, we show that the assertion is true for $p = \infty$. Let $\varphi \in C^{\infty}_{0}(\R)$ and
            $\mbox{supp}(\varphi) \subset (a,b)$, by Corollary \ref{coro_Smooth}, we only need to consider $x \geq a$ to bound  
            the $L^\infty$-norm of ${}{{D}}{^{\alpha}_{x}} \varphi$.
            Since ${}{D}{^{\alpha}_{x}} \varphi\in C^{\infty}([a,\infty))$, of course,  ${}{D}{^{\alpha}_{x}} \varphi\in C^{\infty}([a,b])$. Hence, ${}{D}{^{\alpha}_{x}} \varphi$ is continuous over the compact 
            set $[a,b]$. Then there exists $M_1 > 0$ so that $\left\|{}{D}{^{\alpha}_{x}} \varphi \right\|_{L^{\infty}([a,b])} \leq M_1$.  
            It remains to show that there exists $M_2 > 0$ so that $\left|L'(x) \right| \leq M_2$ almost everywhere in $(b , \infty)$. Noticing that $\left|L'(x) \right| \rightarrow 0$ continuously as $x \rightarrow \infty$. 
            Then for any $\varepsilon >0$, then there exists $N>0$ so that $\left|L'(x) \right| <\eps$ for every $x > N$.  Since $\mbox{supp}(\varphi) \subset (a,b)$, there exists $c \in (a,b)$ so that $\varphi(x) \equiv 0$ for every $x \in [c, N]$. Thus, 
            \begin{align*}
                L'(x) &= \dfrac{- \alpha }{\Gamma(1 - \alpha)} \int_{a}^{c} \dfrac{\varphi(y)}{(x-y)^{ 1+ \alpha}}\, dy.
            \end{align*}
            It follows that for all $x>b$
            \begin{align*}
                \left|L'(x)\right| \leq \dfrac{ \alpha}{\Gamma(1 - \alpha)} \int_{a}^{c} \dfrac{|\varphi(y)|}{(x - y)^{1  +\alpha}}\, dy 
                \leq \dfrac{\alpha}{\Gamma(1 - \alpha)} \int_{a}^{c} \dfrac{\left|\varphi(y) \right|}{(b - y)^{1 +  \alpha}}\, dy.
            \end{align*}
            Clearly, the integrand is bounded for all $y \in [a,c]$. Hence, there exists some $M_2 > 0$ such that $\left|L'(x) \right| \leq M_2$ for all $x>b$. Taking $M = \max \left\{M_1 , M_2 \right\}$, we have $\left|{}{D}{^{\alpha}_{x}} \varphi \right| \leq M$. Thus, ${}{D}{^{\alpha}_{x}} \varphi \in L^{\infty}(\R)$.

            Next, we consider the case when $1 \leq p <\infty$. Let $\mbox{supp}(\varphi) \subseteq [c,d] \subset (a,b)$, and $q$ be the H\"older conjugate of $p$; namely, $p^{-1}+q^{-1}=1$. It follows 
            \begin{align*}
                \left\| {}{D}{^{\alpha}_{x}} \varphi \right\|_{L^{p}(\R)}^{p}
                &= \int_{\R} \left| \dfrac{1}{\Gamma(1 - \alpha)} \dfrac{d}{dx} \int_{-\infty}^{x} \dfrac{\varphi(y)}{(x-y)^{\alpha}}\,dy \right|^{p} \,dx \\ 
                &= \left(\dfrac{1}{\Gamma(1-\alpha)}\right)^{p}\left(\int_{a}^{b} \left|\dfrac{d}{dx} \int_{c}^{x} \dfrac{\varphi(y)}{(x-y)^{\alpha}}\,dy \right|^{p}\,dx \right. \\
                &\qquad\qquad \left. + \int_{b}^{\infty} \left|\dfrac{d}{dx} \int_{c}^{d} \dfrac{\varphi(y)}{(x-y)^{\alpha}}\,dy \right|^{p}\,dx \right)\\
                &=: C^p(\alpha) \bigl(I + II \bigr),
                \end{align*}
                and 
                \begin{align*}
                I &:= C^p(\alpha) \int_{a}^{b} \left|\dfrac{d}{dx} \int_{c}^{x} \dfrac{\varphi(y)}{(x-y)^{\alpha}}\,dy \right|^{p}\,dx \\
                &= C^p(\alpha)  \int_{a}^{b} \left| \int_{c}^{x} \dfrac{\varphi'(y)}{(x-y)^{\alpha}} \,dy \right|^{p}\,dx \\
                &\leq C^p(\alpha)
                \| \varphi'\|_{L^{\infty}(\R)}^{p} \int_{a}^{b} \left(\int_{c}^{x} \dfrac{dy}{(x-y)^{\alpha}} \right)^{p}\,dx \\
                &\leq C^p(\alpha) \|\varphi'\|_{L^{\infty}(\R)}^{p} \int_{a}^{b} \left(\dfrac{(x-c)^{1-\alpha}}{1-\alpha}\right)^{p} \,dx \\
                &=C^p(\alpha) \left(\dfrac{\|\varphi'\|_{L^{\infty}(\R)}}{(1-\alpha)} \right)^{p}
                   \left(\dfrac{(b-c)^{2-\alpha}}{2-\alpha} -\dfrac{(c-a)^{2-\alpha}}{2-\alpha} \right)<\infty
                \end{align*}
                and 
                \begin{align*}
                II &:= C^p(\alpha) \int_{b}^{\infty} \left|\dfrac{d}{dx} \int_{c}^{d} \dfrac{\varphi(y)}{(x-y)^{\alpha}}\,dy \right|^{p}\,dx \\
                &= + \alpha^p\int_{b}^{\infty} \left| \int_{c}^{d} \dfrac{\varphi(y)}{(x-y)^{1 + \alpha}}\,dy \right|^{p}\,dx\\
                &\leq C^p(\alpha) \alpha^{p}\|\varphi\|_{L^{q}(\R)}^{p} \int_{b}^{\infty} \int_{c}^{d} \dfrac{1}{(x-y)^{p(1+\alpha)}}\,dydx \\
                &\leq C^p(\alpha) \alpha^{p}\|\varphi\|_{L^{q}(\R)}^{p} 
                \int_{c}^{d} \int_{b}^{\infty} \dfrac{1}{(x-y)^{p(1+\alpha)}}\,dxdy\\
                &= C^p(\alpha)  \alpha^{p}\|\varphi\|_{L^{q}(\R)}^{p}  
                    \int_{c}^{d} \dfrac{dy}{(1-p(1+\alpha)) (b-y)^{p(1+\alpha)-1}}\\
                &=C^p(\alpha)  \alpha^{p}\|\varphi\|_{L^{q}(\R)}^{p}  \left(\dfrac{1}{(1-p(1+\alpha))(2-p(1+\alpha))}\right)  \\
                &\qquad \times 
                 \left(\dfrac{1}{(b-d)^{p(1+\alpha)-2}} - \dfrac{1}{ (b-c)^{p(1+\alpha)-2}}\right) < \infty. 
            \end{align*}
            This completes the proof. 
        \end{proof}
       
       \subsection{Fractional Derivatives of Mollified Functions}\label{sec-2.8}
        In this subsection, we consider a class of special smooth functions with compact supports, 
        which are obtained through a mollification process (i.e., through a convolution with a 
        compactly supported mollifier).  We are interested in relating fractional derivatives of such 
        a function to the fractional derivatives of the original function (before mollification). 
        
        Let $\eta_{\eps}$ be a standard mollifier with support $B_\eps(0)=\{x;\, |x|\leq \eps\}$ (cf. \cite{Evans}), recall that $f^\eps:=\eta_\eps * f$,  the convolution of $\eta_{\eps}$ with $f$, 
        is called the mollification of $f$ by $\eta_{\eps}$. 
        
        \begin{lemma}\label{FourierMollifier}
            Let $0< \alpha <1$, suppose $f \in L^{1}(\R)$ is Fourier fractionally differentiable such that  ${^{\mathcal{F}}}{D}{^{\alpha}} f\in L^1(\R)$ . Then ${^{\mathcal{F}}}{D}{^{\alpha}}f^{\eps} = \eta_{\eps} * {^{\mathcal{F}}}{D}{^{\alpha}} f$.
        \end{lemma}

        \begin{proof}
            By a well-known property of Fourier transform, we get 
            \begin{align*}
                \mathcal{F} \left[ f^{\eps} \right] (\xi) = \mathcal{F} \left[ \eta_{\eps} * f \right] (\xi) 
                = \mathcal{F} \left[ \eta_{\eps} \right] (\xi) \cdot \mathcal{F}[f] (\xi) . 
            \end{align*}
            Then 
            \begin{align*}
                (i \xi)^{\alpha} \mathcal{F} \left[ f^{\eps} \right] (\xi) = \mathcal{F} \left[ \eta_{\eps} \right] (\xi) \cdot (i \xi)^{\alpha} \mathcal{F} [ f] (\xi) 
                = \mathcal{F} \bigl[ \eta_{\eps} \bigr] (\xi) \cdot \mathcal{F} \bigl[ {^{\mathcal{F}}}{D}{^{\alpha}} f \bigr](\xi)
            \end{align*}
            where we have used the fact that $       
                \mathcal{F}\bigl[  {^{\mathcal{F}}}{D}{^{\alpha}} f \bigr](\xi) 
                = (i \xi)^{\alpha} \mathcal{F} [f](\xi)$. 
            Hence, 
            \begin{align*}
                (i \xi)^{\alpha} \mathcal{F}\left[ f^{\eps} \right] (\xi) = \mathcal{F} \bigl[ \eta_{\eps} * {^{\mathcal{F}}}{D}{^{\alpha}} f \bigr](\xi)
            \end{align*}
            Taking the inverse Fourier transform and using the definition of ${^{\mathcal{F}}}{D}{^{\alpha}} f^{\eps}$ yields the desired result.
            The proof is complete. 
        \end{proof}

        \begin{corollary}\label{RLMollifier}
            Let $0 < \alpha <1$ and $f \in C_{0}^{1}(\R)$, then ${}{D}{^{\alpha}_{x}} f^{\eps} = \eta_{\eps} * {}{D}{^{\alpha}_{x}} f$
            and ${_{x}}{D}{^{\alpha}} f^{\eps} = \eta_{\eps}* {_{x}}{D}{^{\alpha}}f .$
        \end{corollary}

        \begin{proof}
            Since $f \in C^{1}_{0} ( \R)$,  by  Lemma \ref{FourierMollifier},  we have
            \begin{align*}
            {}{D}{^{\alpha}_{x}} f= {^{\mathcal{F}}}{D}{^{\alpha}} f, \,\, 
            {}{D}{^{\alpha}_{x}} f^\eps= {^{\mathcal{F}}}{D}{^{\alpha}} f^\eps; \quad
             {_{x}}{D}{^{\alpha}} f= (-1)^{\alpha} {^{\mathcal{F}}}{D}{^{\alpha}} f,\,\,
              {_{x}}{D}{^{\alpha}} f^{\eps} = (-1)^{\alpha} {^{\mathcal{F}}}{D}{^{\alpha}} f^{\eps},
            \end{align*}
            which   immediately infers the assertions. 
        \end{proof}

        \begin{corollary}\label{CaputoMollifier}
            If $f \in C^{1}_{0} (\R)$, then ${^{C}}{D}{^{\alpha}} f^{\eps} = \eta_{\eps} * {^{C}}{D}{^{\alpha}}f.$
        \end{corollary}

        \begin{proof}
            It follows immediately from Corollary \ref{RLMollifier} and Proposition \ref{RLC}. 
        \end{proof}

\section{Two Alternative Perspectives of Classical Fractional Derivatives}\label{sec-3}
In this section we present two alternative understandings of the classical fractional order 
integrals and derivatives. The first perspective is to interpret the fractional differentiation 
as a by-product of the fractional integration through the so-called {\em  Fundamental Theorem 
of Classical Fractional Calculus}. The second one is to interpret the fractional integration as  
a (function) transform and the fractional differentiation as its inverse transform. 
Both interpretations shed some light on understanding the true meaning of these  fractional 
operators as well as the fractional calculus. 

\subsection{Fundamental Theorem of Classical Fractional Calculus (FTcFC)}\label{sec-3.1}
 
\subsubsection{\bf FTcFC on Finite Intervals $(a,b)\subset \R$}\label{sec-3.1.1}
We begin this subsection by recalling the following properties of the fractional operators 
${^{\pm}}{I}{^{\alpha} }$ and $  {^{\pm}}{D}{^{\alpha} }$.
        
    \begin{lemma} [cf. \cite{Samko}]\label{lemma3.1}
        Let $0 < \alpha <1$. The  following properties hold:
        \begin{itemize}
            \item[(a)] ${_{a}}{D}{^{\alpha}_{x}}{_{a}}{I}{^{\alpha}_{x}} f(x) = f(x)$ and ${_{x}}{D}{^{\alpha}_{b}}{_{x}}{I}{^{\alpha}_{b}}f(x) = f(x)$ for any $f \in L^{1}_{loc}((a,b))$.
            \item[(b)] If ${_{a}}{I}{^{1 - \alpha}_{x}}f \in AC([a,b])$, then 
            \begin{equation}\label{FTFC_1}
            f(x) = c^{1-\alpha}_{-} \kappa^{\alpha}_{-}(x) + {_{a}}{I}{^{\alpha}_{x}}{_{a}}{D}{^{\alpha}_{x}}f(x),
            \end{equation} 
            and if ${_{x}}{I}{^{1 - \alpha}_{b}}f \in AC([a,b])$, then 
            \begin{equation}\label{FTFC_2}
            f(x) = c^{1-\alpha}_{+} \kappa^{\alpha}_{+}(x) + {_{x}}{I}{^{\alpha}_{b}} {_{x}}{D}{^{\alpha}_{b}}f(x),
            \end{equation}
            where 
           \begin{align} \label{FTFC_3}
            c_{-}^{\sigma}  := \frac{  {_{a}}{I}{^{\sigma}_{x}} f(a) }{\Gamma(\sigma)},\qquad
            c_{+}^{\sigma} := \frac{   {_{x}}{I}{^{\sigma}_{b}} f(b) }{\Gamma(\sigma)}. 
            \end{align}
        \end{itemize}
    
    \end{lemma}

On noting the fact that ${^{\pm}}{D}{^{\alpha}}F(x) = f(x)$ implies that ${^{\pm}}{I}{^{1-\alpha}}F \in AC([a,b])$, then the above lemma immediately infers the 
following theorem. 

    \begin{theorem} \label{FTFC} 
        Let $0<\alpha <1$, $f, F \in L^{1}((a,b))$. 
        Then ${^{\pm}}{D}{^{\alpha}}F(x) = f(x)$  on $(a,b)$ if and only if 
        \begin{equation}\label{FTFC_4} 
        F(x) = c_{\pm}^{1-\alpha} \kappa^{\alpha}_{\pm}(x) + {^{\pm}}{I}{^{\alpha}}f(x).
        \end{equation}
  
    \end{theorem} 
     
\begin{remark}
 The analogue of Theorem \ref{FTFC} in the Newton-Leibniz (integer order) calculus is the well-known 
	{\em Fundamental Theorem of Calculus (or Newton-Leibniz Theorem)} which says that $F'(x):=\frac{dF}{dx}(x)=f(x)$ if 
		and only if 
		\[
		F(x)= F(a) + \int_a^x f(y)\, dy =  F(a) + \int_a^b H(x-y) f(y)\, dy\qquad \forall x\in [a,b],
		\]
		where the kernel function $\kappa(x,y)= H(x-y)$, the Heaviside function. Since 
		the kernel space of the derivative $\frac{d}{dx}$ operator is $\R$, this is why the first term 
		on the right-hand side must be a constant, because it must belong to the kernel space of $\frac{d}{dx}$. 
		Due to the above analogue, we shall call Theorem \ref{FTFC} {\em Fundamental Theorem of Classical Fractional Calculus} on finite intervals in the rest of this paper.
		
	To further explain this point, we recall that the Fundamental Theorem of Calculus (or Newton-Leibniz Theorem) actually consists of the following two parts (often called the first 
		and second FTC): 
		\begin{itemize}
		    \item [{\rm (i)}] if $F(x) = \int_{a}^{x} f(y)\,dy + C$, then $F'(x) = f(x)$,
		    \item [{\rm (ii)}] if $F'(x) = f(x)$, then $F(x) = f(a) + \int_{a}^{x} f(y)\,dy$.
		\end{itemize}
		In the statement of Theorem \ref{FTFC}, the backward direction is the fractional analogue to 
		part {\rm (i)} where $c^{1-\alpha}_{\pm}$ can be replaced by any constant, and the 
		forward direction is the analogue to part {\rm (ii)} where the constant $c^{1-\alpha}_{\pm}$
		 must be related to $F$ in such a specific way. Rather than separate the FTcFC into two parts, 
		we adopt the form as presented, this is because we shall be mostly interested in the 
		forward direction of the theorem and therefore the relaxing of the given constant 
		is a mute point.
\end{remark}

In fact, given integral operators ${^{\pm}}{I}{^{\alpha}}$,   \eqref{FTFC_4} can be used to define 
the corresponding Riemann-Liouville derivatives as follows.

  \begin{definition} \label{FTFC_def} 
	Let $0<\alpha <1$ and $f , F\in L^{1}((a,b))$. Then $f$ is called  the $\alpha$ order left/right  
	Riemann-Liouville fractional derivative of $F$, and write ${^{\pm}}{D}{^{\alpha}}F(x) = f(x)$, 
	if \eqref{FTFC_4} holds.  

\end{definition} 

By Lemma \ref{Abel} it is easy to check that the $\alpha$ order fractional derivative of $f$, if it exists, 
is uniquely 
defined.  In light of Theorem \ref{FTFC},   we know that the original Riemann-Liouville derivatives
satisfy \eqref{FTFC_4}. Hence, the original definition and the above definition coincide. 

\begin{remark}
	In this paper we emphasize the above FTFC approach of using a given integral operator (i.e.,  its kernel function is given) to define the corresponding derivative notion by the FTFC identity.  There are many benefits/advantages of this approach. It is systematic (not ad hoc) and quite general, because
	it is done in the same way for any given integral operator (see the definition below). The FTFC is 
	built into the definition; we regard that having such a FTFC is essential for any fractional calculus theory
	(namely, no FTFC, no fractional calculus).  Indeed, using the FTFC as a criterion, some fractional derivative notions (such as Caputo) would be considered incomplete. 
\end{remark}

We now give the alluded definition of fractional derivatives for general kernels (and their associated integral operators). 

\begin{definition}
Given any kernel function $\tau\in L^1((a,b)\times (a,b))$, let 
$I_\tau$ denote the subordinate (Riemann or Lebesgue) integral operator, namely,
\begin{align}
I_{\tau} f(x):= \int_{a}^b \tau(x,y) f(y)\, dy    \qquad \forall x\in [a,b].
\end{align}
Let $f, F\in L^1(\Ome)$, then $f$ is called the fractional/nonlocal derivative of $F$, and written
$D_{\tau} F = f$, there exists some $c\in [a,b]$ such that 
       \begin{align} \label{general_def_a}
           F(x) &=   C_{F,c} \tau(x,c) +  I_{\tau}  f(x) \qquad \forall x\in [a,b] 
       \end{align}
       for some constant $C_{F,c}$ depending on both $F$ and $c$.

\end{definition}

At this point, we would like to circle back to the inclusivity result and its accompanying remark given in Section \ref{sec-2.3}. We are now ready to give a precise statement of this result.

\begin{proposition}\label{ClassicConsistency}
    Let $0 < \alpha < \beta < 1$, suppose  that $f$ and ${^{\pm}}{D}{^{\beta}} f$ exist and  belong to $L^{1}((a,b))$. Then ${^{\pm}}{D}{^{\alpha}}f$ exists and belongs to $L^{1}((a,b))$. 
\end{proposition}

\begin{proof}
    We only prove the result for the left direction since the right direction follows similarly. 
    By Theorem \ref{FTFC}, 
    \[
    f(x) = c^{1-\beta}_{-} \kappa^{\beta}_{-}(x) + {^{-}}{I}{^{\beta}}{^{-}}{D}{^{\beta}} f(x).
    \]
    Taking the $\alpha$ order derivative on the right-hand side, it follows by \eqref{d_formula1}, the semigroup property of ${^{-}}{I}{^{\beta}}$, and Lemma \ref{lemma3.1} that 
    \begin{align*}
         {^{-}}{D}{^{\alpha}}  \bigl[ c^{1-\beta}_{-}  \kappa^{\beta}_{-}(x) +{^{-}}{I}{^{\beta}} {^{-}}{D}{^{\beta}} f(x) \bigr](x)
        = c^{1-\beta}_{-} \kappa^{\beta -\alpha}_{-}(x)+ {^{-}}{I}{^{\beta - \alpha}} {^{-}}{D}{^{\beta}}f(x).
    \end{align*}
    A direct calculation confirms that the above function belongs to $L^1(\Omega)$. Hence, ${^{-}}{D}{^{\alpha}}f$ exists as a member of $L^{1}(\Omega)$. This completes the proof.
\end{proof}

\subsubsection{\bf FTcFC on the Infinite Interval $\R$}\label{sec-3.1.2}
The case for a FTcFC on the entire line is quite different, but simpler because of the decay properties of 
kernel functions $\kappa^{\alpha}_{\pm}$ when $|x|\to \infty$. Similarly, we start by recalling the following properties 
of the fractional operators ${^{\pm}}{I}{^{\alpha} }$ and $  {^{\pm}}{D}{^{\alpha} }$.  

        \begin{lemma}[cf. \cite{Samko}] \label{lem3.4}
        	Let $0 < \alpha < 1$. The following properties hold:
        \begin{itemize}
            \item[(a)] $D^{\alpha}_{x} I^{\alpha}_{x} f(x) = f(x)$ and ${_{x}}{D}{^{\alpha}} {_{x}}{I}{^{\alpha}}f(x) = f(x)$ for any $f \in L^{1}(\R)$,
            \item[(b)] $I^{\alpha}_{x} D^{\alpha}_{x} f(x) = f(x)$ and ${_{x}}{I}{^{\alpha}} {_{x}}{D}{^{\alpha}} f(x) =  f(x)$ for any $I^{1-\alpha}f \in AC(\R)$ so that $f(x) \rightarrow 0$ as $|x| \rightarrow \infty$.
        \end{itemize} 
    \end{lemma}
    
    We then have
    
    \begin{theorem}\label{FTFCa}
        Let $0 < \alpha < 1$, and $f, F \in L^{1}(\R)$. If
        \begin{equation}\label{FTFC_1a}
               F(x)= {^{\pm}}{I}{^{\alpha}} f(x),
        \end{equation} 
        then ${^{\pm}}{D}{^{\alpha}} F(x) = f(x)$. The converse is true under the additional assumption $F(x)\to 0$ as $|x|\to \infty$.
    \end{theorem}

For the same reason as explained in Subsection \ref{sec-3.1.1}, we shall call 
Theorem \ref{FTFCa} {\em the Fundamental Theorem of Classical Fractional Calculus on $\R$}
in the rest of this paper. 

Similarly, we also introduce the following definition.

  \begin{definition} \label{FTFC_defa} 
	Let $0<\alpha <1$ and $f , F\in L^{1}(\R)$. Then $f$ is called  the $\alpha$ order left/right  
	Riemann-Liouville fractional derivative of $F$ on $\R$, and write ${^{\pm}}{D}{^{\alpha}}F(x) 
	= f(x)$ (abusing the notation), if \eqref{FTFC_1a} holds.  
\end{definition} 

It is easy to show that ${^{\pm}}{D}{^{\alpha}}F$ is well defined and it coincides with 
the original definitions of Riemann-Liouville derivatives in the case of the infinite interval $\R$. 
This FTcFC interpretation of fractional derivatives will be emphasized in this paper.

The following result, which is analogous to that of Proposition \ref{ClassicConsistency},  holds. We omit its proof to save space.
	
\begin{proposition}
    Let $0 < \alpha < \beta < 1$. Suppose that $f$ and ${^{\pm}}{D}{^{\beta}}f$ belong to $L^{1}(\R)$. Then ${^{\pm}}{D}{^{\alpha}}f$ exists and belongs to $L^{1}(\R)$.
\end{proposition}

\subsection{Transform Characterization of Fractional Integration and Differentiation}\label{sec-3.2}
In this section we shall present another (but related)  view point for understanding Riemann-Liouville
fractional integration and differentiation as a pair of forward and inverse transforms. This function
transform view point will also provide a geometric interpretation for Riemann-Liouville fractional 
derivatives. 

\begin{definition} 
	Let $0<\alpha <1$ and $\Ome=(a,b)$ or $\R$.  For any $f\in L^1(\Omega)$, we define the left/right Riemann-Liouville 
	transforms $ {_\pm}{\mathcal{R}}{_{\alpha}} [f]$ of $f$ by
\begin{align}\label{R_transform-1}
{_{-}}{\mathcal{R}}{_{\alpha}} [f](\xi):= \frac{1}{\Gamma(\alpha)} \int_{a^*}^{\xi} \frac{f(x)}{(\xi-x)^{1-\alpha}}\, dx 
	\qquad \forall \xi\in \Ome, \\
{_{+}}{\mathcal{R}}{_{\alpha}} [f](\xi):= \frac{1}{\Gamma(\alpha)} \int^{b^*}_\xi\frac{ f(x)}{(x-\xi)^{1-\alpha}}\, dx 
\qquad \forall \xi\in \Ome,
\end{align}
where  $a^*=a$ or $-\infty$ and $b^*=b$ or $\infty$.  
\end{definition} 

\begin{remark}
	Clearly, $\widehat{f}^\pm_\alpha := {_\pm}{\mathcal{R}}{_{\alpha}} [f] ={^{\pm}}{I}{^{\alpha}} f$ are 
		the left and right Riemann-Liouville fractional integrals of $f$. 
	We intentionally use a different notation $\xi$ to denote the independent variable for the transformed function  $\widehat{f}^\pm_\alpha$ to indicate that it is defined in 	the ``frequency" domain $(\hat{a},\hat{b})=(a,b)$. 
\end{remark}

By Theorem \ref{LpMappings} we know that ${_\pm}{\mathcal{R}}{_{\alpha}} $ map $L^1(\Ome)$
into itself.  We are interested in knowing  inverse transforms of ${_\pm}{\mathcal{R}}{_{\alpha}} $ which are  
defined below. 

\begin{definition} 
	Let $0<\alpha <1$ and $\Ome=(a,b)$ or $\R$.  For any $f\in L^1(\Ome)$, we define the left/right  inverse Riemann-Liouville transforms ${_\pm}{\mathcal{R}}{_{\alpha}} ^{-1} $  as mappings 
	from $L^1(\Ome)$ into itself such that 
	\begin{align}\label{R_transform-2}
	{_{\pm}}{\mathcal{R}}{_{\alpha}}^{-1}  \bigl[ {_{\pm}}{\mathcal{R}}{_{\alpha}}[f](\xi)\bigr] (x) =  f(x) 
	\qquad \forall  x\in \Ome.
	\end{align}
	 
\end{definition}  

By Lemmas \ref{lemma3.1} and \ref{lem3.4}  we immediately get

\begin{proposition}
	The left/right inverse Riemann-Liouville transforms  ${_\pm}{\mathcal{R}}{_{\alpha}}^{-1}$ 
	are uniquely defined and 
	${_\pm}{\mathcal{R}}{_{\alpha}}^{-1} = {^\pm}{D}{^{\alpha}}$. 
\end{proposition} 

\begin{remark}
(a)The above transform interpretation of the fractional Riemann-Liouville integration and differentiation 
	is consistent with the FTFC interpretation. Moreover, it reveals some more insights about these fractional 
	derivatives in the sense that the integration was done in the ``physical domain" $\Ome$, on the other hand,
	the differentiation is performed in the ``frequency domain" $\widehat{\Ome}$, which is topologically the
	 same as $\Ome$, but metrically not equivalent.  

(b) To give a geometric interpretation of $\alpha$ order fractional Riemann-Liouville derivatives, 
recall that ${^{\pm}}{D}{^{\alpha}}f(\xi) = \frac{d}{d\xi} \bigl( {_{\pm}}{\mathcal{R}}{_{1-\alpha}}[f](\xi) \bigr)$,
hence ${^{\pm}}{D}{^{\alpha}}f(\xi)$ measures the rate of change (or the slope of the tangent) 
of ${_{\pm}}{\mathcal{R}}{_{1-\alpha}}[f]$ at $\xi$ in the ``frequency domain" of the mapping 
${_{\pm}}{\mathcal{R}}{_{1-\alpha}}$. 
\end{remark}
 
%
%
 
\section{A Weak Fractional Differential Calculus Theory}\label{sec-4}
As having been seen in the previous sections, the classical fractional calculus theory has several difficulties 
arising from the change to non-integer order integration and differentiation. Unlike the well formulated and understood 
integer order calculus, the basic notion of fractional derivatives is domain-dependent 
and has several different (and nonequivalent) 
definitions; familiar calculus rules do not hold or become fairly complicated and restricted;
fractionally differentiable functions are difficult to characterize;  there is no local characterization of 
non-local fractional integral and derivative operators; more importantly, although the Riemann integration can be 
generalized to the Lebesgue integration in the definitions of all fractional integrals, unlike the 
integer order case, there is no weak derivative concept/theory, and the underlying Sobolev space 
theory for fractionally differentiable functions is not well developed nor well understood, which in turn has caused   
a lot of difficulties and confusions for understanding/studying/interpreting fractional order differential 
equations. 

The primary goal of this section (and this paper) is to develop a weak fractional differential calculus theory and its 
corresponding Sobolev space theory, which are respectively parallel to the integer order weak derivative theory and its corresponding 
Sobolev space theory (cf. \cite{Adams, Brezis, Evans}). The anticipated weak fractional theories will lay down a solid theoretical 
foundation and  pave the way for a systematical and thorough study of initial value, boundary value and initial-boundary 
value problems for fractional order differential equations and fractional calculus of variations problems as well as 
their numerical solutions in the subsequent works \cite{Feng_Sutton2,Feng_Sutton3}. 

Although it may be semantics, we find it insightful to point out the difference 
	(and the lack there of) between two terminologies \textit{``fractional weak"} and \textit{``weak fractional"}. First and foremost, we regard them to be interchangeable and have done so in the paper. However, we suspect that one may prefer one to the other depending on their prospective. 
	For example, from the viewpoint of building a calculus theory (i.e. weak calculus built on top of classical calculus), it seems more accurate to use \textit{weak fractional}. On the other hand, through the lens of generalizing the existing integer order weak derivative theory to fractional order, one may prefer the terminology \textit{fractional weak}. Since the order difference  
	plays no role in this paper, we simply choose (with some ambiguity) to use \textit{``weak fractional"} in the remainder of the paper.

In this section, unless it is stated otherwise, all integrals are understood in the Lebesgue sense. We use ${^{-}}{D}{^{\alpha}}$ and ${^{+}}{D}{^{\alpha}}$ to denote respectively any left and right $\alpha$ order classical derivative introduced in Section \ref{sec-2}. ${^{\pm}}{D}{^{\alpha}}$ denotes either ${^{-}}{D}{^{\alpha}}$ or ${^{+}}{D}{^{\alpha}}$. $\Omega$ denotes either a finite interval $(a,b)$ or the whole real line $\R$.  In the case $\Ome=(a,b)$, for any $\varphi \in C^{\infty}_{0}(\Omega)$, $\tilde{\varphi}$ is used to denote the zero extension of $\varphi$ to $\R$.

        
\subsection{Definitions of Weak Fractional Derivatives}\label{sec-4.1}
%
Like in the integer order case, the idea of defining {\em weak} fractional derivative ${^{\pm}}{ \mathcal{D}}{^{\alpha}} u$ of a function $u$ is to specify its action on any smooth compactly supported function $\varphi \in C^{\infty}_{0}(\Omega)$, 
instead of knowing its pointwise values as done in the classical fractional derivative definitions.

    \begin{definition}\label{RWFD}
        For $\alpha> 0$, let $[\alpha]$ denote the integer part of $\alpha$. For $u \in L^{1}(\Omega)$, 
       \begin{itemize} 
       \item[{\rm (i)}] a function $v \in L_{loc}^{1} (\Omega)$ is called the left weak fractional derivative of $u$ if 
        \begin{align*}
            \int_{\Omega} v(x) \varphi(x) \,dx = (-1)^{[\alpha]} \int_{\Omega} u(x) {^{+}}{D}{^{\alpha}} \tilde{\varphi}(x) \, dx
             \qquad \forall \varphi \in C_{0}^{\infty} (\Omega),
        \end{align*}
        we write ${^{-}}{ \mathcal{D}}{^{\alpha}} u:=v$; 
     \item[{\rm (ii)}] a function $w\in L_{loc}^{1} (\Omega)$ is called the right weak fractional derivative of $u$ if 
      \begin{align*}
       \int_{\Omega} w(x) \varphi(x) \,dx = (-1)^{[\alpha]} \int_{\Omega} u(x) {^{-}}{D}{}^{\alpha} \tilde{\varphi}(x) \,dx
       \qquad \forall \varphi \in C_{0}^{\infty} (\Omega), 
      \end{align*}
       and we write ${^{+}}{\mathcal{D}}{^{\alpha}} u:=w$. 
      \end{itemize}
   \end{definition}

The next proposition shows that weak fractional derivatives are well-defined. 

\begin{proposition}
	Let $u \in L^1 (\Omega)$. Then a weak fractional derivative of $u$, if it exists, is uniquely defined.
\end{proposition}

\begin{proof}
	Let $v_1, v_2 \in L^1_{loc} ( \Omega)$ be two left (resp. right) weak fractional derivatives of $u$, then
	\begin{align*}
	\int_{\Omega} v_1(x) \varphi(x) \, dx = (-1)^{[\alpha]} \int_{\Omega} u(x) {^{\pm}}{D}{^{\alpha}} 
	\tilde{\varphi}(x) \, dx 
	=   \int_{\Omega} v_2(x) \varphi(x) \, dx  \quad\forall \varphi \in C_{0}^{\infty} (\Omega).
	\end{align*}
	Thus, 
	\begin{align*}
	0 = \int_{\Omega} \big(v_1(x) - v_2(x) \big) \varphi(x) \, dx \qquad\forall \varphi \in C^{\infty}_{0} (\Omega).
	\end{align*}
	Therefore, $v_1 = v_2$ almost everywhere. The proof is complete.
\end{proof}

  A few remarks are given below to help understand the above definition. 
  
    \begin{remark}
       (a) The introduction of $\tilde{\varphi}$ in the definitions makes the weak fractional derivatives intrinsic 
       in the sense that ${^{\pm}}{D}{}^{\alpha} \tilde{\varphi}$ is independent of the choice of ${^{\pm}}{D}{}^{\alpha}$, because 
       ${ }{D}{^{\alpha}_{x}}\tilde{\varphi}={_{a}}{D}{^{\alpha}_{x}}\tilde{\varphi}= {^{\mathcal{F}}}{D}{^{\alpha}}\tilde{\varphi}$ and  
       ${_{x}}{D}{^{\alpha} }\tilde{\varphi}={_{x}}{D}{^{\alpha}_{b}}\tilde{\varphi}= {^{\mathcal{F}}}{D}{^{\alpha}}\tilde{\varphi}$.
       
       (b) The constant $(-1)^{[\alpha]}$ helps guarantee consistency with the integer order case. 
       
      (c) Integration by parts is built into the definitions.
       
       
       (d) The reason to require $u \in L^{1}(\Omega)$ is because ${^{\pm}}{D}{^{\alpha}} \tilde{\varphi} \in L^{\infty}(\R)$  
       is not compactly supported. When $\alpha \in \N$, this condition can be relaxed to $L^{1}_{loc}(\Omega)$. In fact, the restriction $u \in L^{1}(\Omega)$ can be relaxed to the weighted $L^1$ space 
       $u\in L^1(\Omega, \rho)$ with the weight $\rho=L'$ or $\rho=R'$.
       
       (e)  As expected, weak fractional derivatives are {\em domain-dependent}. 
       	However, unlike 
       the classical fractional derivatives, whose domain dependence is explicitly shown in the limits of the integrals
       involved, the domain dependence of weak fractional derivatives is implicitly introduced by using 
       domain-dependent test functions $\varphi\in C^\infty_0(\Omega)$.  
       
       (f)  The above definitions can be easily extended to non-interval domains or subdomains of $\Omega$.  Indeed, given  a bounded set $E\subset \R$,  the 
       only changes which need  to be made in the definitions are to replace $\Omega$ by $E$ and $\varphi\in C^\infty_0(\Omega)$  by  $\varphi\in C^\infty_0((a^*, b^*))$  where $(a^*,b^*) =\cap\{(c,d):\, E\subset (c,d)\}$,
       the smallest interval which contains $E$.
    
       (g) Extensions of the above definitions to distributions will be given in Section \ref{sec-6}.
      \end{remark}

 \begin{proposition}\label{Weak=RL}
	Let $u$ be Riemann-Liouville differentiable such that ${^{\pm}}{D}{^{\alpha}}u \in L^{1}_{loc}(\Omega)$. Then ${^{\pm}}{\mathcal{D}}{^{\alpha}} u = {^{\pm}}{D}{^{\alpha}} u$ almost everywhere.
\end{proposition}

\begin{proof}
	Follows as a direct consequence of the Riemann-Liouville integration by parts formula given in Theorem \ref{IBP}.
\end{proof}

\begin{proposition}
    Let $n - 1 < \alpha <n$. The $\alpha$ order weak fractional derivative converges to the $n^{th}$ order weak derivative almost everywhere as $\alpha \rightarrow n$. 
\end{proposition}

\begin{proof}
    Consider the case when $n=1$; the others follow similarly. In order to prove that ${^{\pm}}{\mathcal{D}}{^{\alpha}} u \rightarrow \mathcal{D}u$ almost everywhere as $\alpha \rightarrow 1$, we see that
    \begin{align*}
        0 
        &= \int_{\Omega} u \varphi'\,dx + \lim_{\alpha \rightarrow 1} (-1)^{[\alpha]} \int_{\Omega} u {^{\mp}}{D}{^{\alpha}} \varphi\,dx \\ 
        &= \lim_{\alpha \rightarrow 1}  \int_{\Omega} {^{\pm}}{\mathcal{D}}{^{\alpha}} u \varphi\,dx - \int_{\Omega} \mathcal{D}u \cdot \varphi\,dx \\
        &= \lim_{\alpha \rightarrow 1} \int_{\Omega} (  {^{\pm}}{\mathcal{D}}{^{\alpha}} u  - \mathcal{D}u) \varphi\,dx,
    \end{align*}
    which follows by the consistency of classical fractional derivatives on test functions $\varphi \in C^{\infty}_{0}(\Omega)$.
\end{proof}

    \subsection{Relationships with Other Derivative Notions}\label{sec-4.2}
    Although the notion of a weak fractional derivative is analogous to the integer order weak derivative and hence is deserving of the name in this sense, we provide simple examples to illustrate the following points.
    \begin{itemize}
        \item [(a)] The notion of a weak fractional derivative is a unifying concept of fractional differentiation with respect to the derivatives defined in Section \ref{sec-2}.
        \item[(b)]  Weak fractional derivatives can exist for functions whose classical fractional
        derivatives do not exist. 
        \item[(c)] Functions that do not have first order weak derivatives may have weak fractional derivatives. 
    \end{itemize}
    
    First, we give an example to demonstrate that the weak fractional derivative is a unifying concept (item (a) above). To illustrate this point, we show that for $\Omega = \R$, $0 < \alpha <1$, and $c \in \R \setminus\{0\}$, the Caputo derivative exists, but the Riemann-Liouville derivative does not. Thus in the classical sense, the choice of fractional derivative definition becomes essential. However, we will show that the weak fractional derivative exists and all ambiguity is avoided because the requirement of satisfying the integration by parts formula in the definition 
    automatically selects the ``correct" derivative notion.
    
    Let $u(x) \equiv c$. A trivial calculation shows that the Caputo derivative,
    \begin{align*}
        {^{C}}{D}{^{\alpha}_{x}} u(x) = \dfrac{1}{\Gamma(1-\alpha)} \int_{-\infty}^{x} \dfrac{u'(y)}{ (x-y)^{-\alpha}}\,dy  
        =  0.
    \end{align*}
    However, in the Riemann-Liouville case,
    \begin{align*}
        {}{D}{^{\alpha}_{x}} u(x) &= \dfrac{1}{\Gamma(1-\alpha)} \dfrac{d}{dx}\int_{-\infty}^{x} \dfrac{c}{(x-y)^{\alpha}} \,dy = \dfrac{c}{\Gamma(1-\alpha)}\dfrac{d}{dx} \left( \dfrac{(x-y)^{1-\alpha}}{1-\alpha} \bigg|_{y=x}^{y=-\infty} \right),
        \end{align*}
    which does not exist as a function because the singular integral diverges. Thus the 
    Riemann-Liouville derivative of a constant function does not exist on $\R$. This then begs 
    the questions of whether the weak fractional derivative exists and whether the fractional 
    weak derivative is able to properly avoid this ambiguity and select the correct notion of derivative. We now compute the weak fractional derivative of $u$ below. For any 
    $\varphi \in C^{\infty}_{0}(\R)$,
    \begin{align*}
        \int_{\R} c {^{+}}{D}{^{\alpha}} \varphi(x)\,dx &= \int_{\R} c \dfrac{d}{dx} {}{I}{^{1-\alpha}_{x}} \varphi(x) \,dx
        = c \Bigl[ {}{I}{^{1-\alpha}_{x}} \varphi(x) \Bigr] \bigg|_{-\infty}^{\infty} = 0,
    \end{align*}
    where the final equality follows by evaluating \eqref{2.20b} at $\infty$. Therefore,   
    the weak derivative exists and is equal to zero, which coincides with the Caputo derivative. Here we see that by forcing the integration by parts formula to hold, the definition automatically  
    selects the appropriate fractional derivative.

    What if $\Omega = (a,b)$ is finite? In this case, we know that $${^{C}_{a}}{D}{^{\alpha}_{x}} c = 0 \neq \dfrac{c (x-a)^{-\alpha}}{\Gamma(1-\alpha)} = {_{a}}{D}{^{\alpha}_{x}} c.$$ In fact, by Proposition \ref{Weak=RL}, we know that ${^{-}}{\mathcal{D}}{^{\alpha}} c = {_{a}}{D}{^{\alpha}_{x}} c$. However, for 
    the sake of illustrating this selection property of the weak fractional derivative, 
    we look at why it does not select the Caputo derivative in this case. 
    A simple calculation yields that 
    \begin{align*}
        \int_{a}^{b} c {^{+}}{D}{^{\alpha}}\varphi(x)\, dx &= c\int_{a}^{b}  \dfrac{d}{dx} {_{x}}{I}{^{\alpha}_{b}} \varphi(x)\,dx = c \Bigl[ {_{x}}{I}{^{\alpha}_{b}} \varphi(x) \Bigr] \bigg|_{a}^{b} = c\, {_{a}}{I}{^{\alpha}_{b}}\varphi(b)
    \end{align*}
     holds for all $\varphi \in C^{\infty}_{0}((a,b))$, which clearly is not satisfied by 
     the Caputo fractional derivative. Hence 
     $ {^{C}_{a}}{{D}}{_{x}^{\alpha}} c\neq {^{-}}{\mathcal{D}}{^{\alpha}} c$.  
     However, as we know by Proposition \ref{Weak=RL} and a direct computation that 
    \begin{align*}
        \int_{a}^{b} c {^{+}}{D}{^{\alpha}}\varphi\,dx  = \int_{a}^{b} \varphi  {_{a}}{D}{^{\alpha}
        _{x}} c\,dx  \qquad \forall \varphi \in C^{\infty}_{0}((a,b)).
    \end{align*}
    Hence, ${^{-}}{\mathcal{D}}{^{\alpha}} c = {_{a}}{D}{^{\alpha}_x} c$. Again, we see that 
    the built-in feature of an integration by parts formula effectively selects an appropriate fractional derivative.
    
    Next, we illustrate that the notion of weak fractional derivatives is truly a generalization 
    of the notion of classical fractional derivatives by showing that there are functions whose
    weak derivatives exist, but classical fractional (Rieamann-Liouville) derivatives do not. 
    Moreover, we give a characterization of functions that are weakly differentiable, which parallels the characterization for first order weakly differentiable functions.
   In lieu of concrete examples, we demonstrate that there is a procedural way to produce 
   functions that are not Riemann-Liouville differentiable, but are weakly differentiable. 
   Notice that for $u \in L^{1}(\Omega)$ and $\varphi \in C^{\infty}_{0}(\Omega)$, there holds
    \begin{align*}
        \int_{\Omega} u {^{\mp}}{D}{^{\alpha}} \varphi\,dx &= \int_{\Omega} u {^{\mp}}{I}{^{1-\alpha}} \varphi'\,dx = \int_{\Omega} {^{\pm}}{I}{^{1-\alpha}} u \varphi '\,dx.
    \end{align*}
    In order to perform an integration by parts on the right side, we need that ${^{\pm}}{I}{^{1-\alpha}}u\in W^{1,1}(\Omega)$ (or at least absolutelely continuous). In that case, the function $u$ then has a weak fractional derivative. On the other hand, we want the function 
    $u$ not to be Riemann-Liouville differentiable, which requires that ${^{\pm}}{I}{^{1-\alpha}} u \not \in C^{1}(\Omega)$. Since ${^{\pm}}{I}{^{1-\alpha}}u \in W^{1,1}(\Omega)$ does not imply ${^{\pm}}{I}{^{1-\alpha}}u \in C^{1}(\Omega) $, then we want to find $u \in L^{1}(\Omega)$ so that ${^{\pm}}{I}{^{1-\alpha}} u = f$ 
    for a given function $f \in W^{1,1}(\Omega)$, but $f \not\in C^{1}(\Omega)$. There are many such $f$ functions, the best known example perhaps is $f(x) = |x|$. 
 
    It follows from Lemma \ref{lemma3.1} that we obtain the desired 
    examples by taking $u  = {^{\pm}}{D}{^{1-\alpha}} f$ 
    for any $f \in \bigl\{ W^{1,1}(\Omega); f \not\in C^{1}(\Omega) \mbox{ and }
    {^{\pm}}{D}{^{1-\alpha}}f \mbox{ exists} \bigr\}$. It follows by the characterization of functions in $W^{1,1}(\Omega)$ being absolutely continuous that
   \textit{$u$ is weakly fractional differentiable with ${^{\pm}}{\mathcal{D}}{^{\alpha}} u \in L^{1}(\Omega)$ if and only ${^{\pm}}{I}{^{1-\alpha}} u$ is absolutely continuous.}

     \begin{remark} The above procedure can be relaxed to characterize all weakly fractional differentiable functions by requiring $f$ to be only first order weakly differentiable; rather than $f \in W^{1,1}(\Omega)$.  However, the above procedure does produce a rich (and nearly complete) validation of item (b) above.
   \end{remark}

    Finally, we compare the weak fractional derivative to the integer order weak derivative; in particular, we demonstrate that the notion of weak fractional derivative is indeed consistent with, and extends, the notion of integer order weak derivatives by identifying a class of functions so that their weak fractional derivatives exist, but their integer order weak derivatives do not. 

    To that end, consider $\Omega = (-1,1)$ and $\lambda , \mu \in \R$ so that $\lambda \neq \mu$, 
    then define
    \begin{align*}
        u(x) : = \begin{cases}
            \lambda &\text{if } -1<x <0, \\ 
            \mu &\text{if } 0<x <1;
        \end{cases}
    \end{align*}
    a genuine step function. Let $\mathcal{D}$ denote the first order weak derivative operator. 
    Obviously, $\mathcal{D}u$ does not exist (\cite{Brezis}) because $u \not\in C((-1,1))$; such a function has only a distributional derivative, a concept we will not be concerned with at this time. However, 
 a direct calculation shows that  
    \begin{align*}
        \int_{-1}^{1} u  {^{\mp}}{D}{^{\alpha}} \varphi \,dx = \int_{-1}^{1}  \varphi {^{\pm}}{ \mathcal{D}}{^{\alpha}}u \,dx \qquad \forall \varphi\in C^\infty_0((-1,1))
    \end{align*}
    holds, where 
    \begin{align*}
        {^{-}}{ \mathcal{D} }{^{\alpha}} u(x) := \begin{cases}
         \dfrac{1}{\Gamma(1-\alpha)} \dfrac{\lambda}{(x+1)^{\alpha}}&\text{if } x \in (-1,0], \\ 
         \dfrac{1}{\Gamma(1-\alpha)} \left( \dfrac{\lambda}{(x+1)^{\alpha}} - \dfrac{\lambda}{x^{\alpha}} + \dfrac{\mu}{x^{\alpha}} \right) &\text{if } x \in (0,1).
        \end{cases} 
    \end{align*}
    A similar formula also holds for ${^{+}}{ \mathcal{D} }{^{\alpha}}$. Note that the weak derivative is locally integrable. In fact, since $0 < \alpha <1$, it is globally integrable; an observation that is foundational to density properties to be shown in Section \ref{sec-5}. Thus, we have shown that all step functions are weakly fractional differentiable, but are not weakly differentiable to any integer order. In fact, it can be shown that the same conclusion also holds for all piecewise smooth, but globally discontinuous functions. Simple exams are given in \cite{Brezis, Evans}.
    
    \subsection{Approximation and Characterization of Weak Fractional Derivatives}\label{sec-4.3}
    In this section we present a characterization for weak fractional derivatives so that they can be approached/understood from a different, but equivalent point of view. Like in the integer order case, 
    we prove that weakly fractional differentiable functions can be approximated by 
    smooth functions. In this section, we assume $0<\alpha< 1$ unless it is stated otherwise. 
    
    \subsubsection{\bf The Finite Interval Case}\label{sec-4.3.1}
     
     We first consider the case when $\Omega:=(a,b) \subset \R$ is a finite interval. Let 
    $\eps > 0$,  define the $\eps$- interior of $\Ome$ as $\Omega_{\eps} : = \{x \in \Omega \, : \mbox{dist}(x,\partial \Omega) >\eps\}.$

        \begin{lemma}\label{WDMollifier}
           Suppose ${^{\pm}}{\mathcal{D}}{^{\alpha}} u \in L_{loc}^{1} (\Omega)$ exists. Then
            \begin{align}\label{WeakMollifier}
                {^{\pm}}{\mathcal{D}}{^{\alpha}} \tilde{u}^{\eps} = \eta_{\eps} * {^{\pm}}{\mathcal{D}}{^{\alpha}}u \qquad \mbox{a.e. in } \Omega_{\eps}
            \end{align}
 where $\eta_{\eps}$ denotes the standard mollifier and $ \tilde{u}^{\eps} $ stands for the 
 mollification of $\tilde{u}$.
        \end{lemma}
        
        \begin{proof}
        Let $\varphi \in C^{\infty}_{0}(\Omega)$ such that $supp(\varphi)\subset \Omega_{\eps}$. Recall that 
         $\tilde{u}$ denotes  the zero extension of $u$ and $\tilde{u}^{\eps}$ is defined by
            \begin{align*}
                \tilde{u}^{\eps} : = \eta_\eps * \tilde{u} \in C^{\infty}_{0}(\R).
            \end{align*}
        Then by Lemma \ref{FourierMollifier} and Corollary \ref{RLMollifier}, we have for any 
        $\varphi \in C^{\infty}_{0}(\Omega_\eps)$
            \begin{align*}
                \int_{\Omega} {^{\pm}}{\mathcal{D}}{^{\alpha}} \tilde{u}^{\eps}(x) \tilde{\varphi}(x) \, dx  &= (-1)^{[\alpha]} \int_{\Omega} \tilde{u}^{\eps}(x) \,{^{\mp}}{D}{^{\alpha}} \tilde{\varphi}(x) \, dx\\
                &= (-1)^{[\alpha]} \int_{\Omega}  (\eta_{\eps} * \tilde{u})(x) \,{^{\mp}}{D}{^{\alpha}} \tilde{\varphi}(x) \, dx \\ 
                &= (-1)^{[\alpha]} \int_{\Omega} \int_{\Omega} \eta_{\eps}(x - y) \tilde{u}(y)\,  {^{\mp}}{D}{^{\alpha}} \tilde{\varphi}(x) \, dy dx \\ 
                &= (-1)^{[\alpha]} \int_{\Omega}\int_{\Omega} \tilde{u}(y) \eta_{\eps}(y -x)\, {^{\mp}}{D}{^{\alpha}} \tilde{\varphi}(x) \, dxdy \\ 
                &= (-1)^{[\alpha]} \int_{\Omega} \tilde{u}(y) \Bigl(\eta_{\eps} * {^{\mp}}{D}{^{\alpha}} \tilde{\varphi}\Bigr)(y) \, dy \\ 
                &= (-1)^{[\alpha]} \int_{\Omega} u(y) {^{\mp}}{D}{^{\alpha}} \tilde{\varphi}^{\eps}(y) \, dy \\ 
                &= \int_{\Omega} {^{\pm}}{\mathcal{D}}{^{\alpha}}u(y) \tilde{\varphi}^{\eps}(y) \, dy \\ 
                &= \int_{\Omega} \int_{\Omega} {^{\pm}}{\mathcal{D}}{^{\alpha}} u(y) \eta_{\eps}(y-x) \tilde{\varphi}(x) \, dxdy \\ 
                &= \int_{\Omega} \left( \eta_{\eps} * {^{\pm}}{\mathcal{D}}{^{\alpha}} u \right)(x) \tilde{\varphi}(x) \, dx .
            \end{align*}
            Thus, \eqref{WeakMollifier} holds. The proof is complete. 
        \end{proof}
        
        The next theorem gives a characterization of   fractional order weak derivatives. 
        
        \begin{theorem}
            Let $u \in L^{1}(\Omega)$. Then $v = {^{\pm}}{\mathcal{D}}{^{\alpha}} u$ in  $L^{1}_{loc} (\Omega)$ if and only if there exists a sequence $\left\{u_j \right\}_{j=1}^{\infty} \subset C^{\infty} (\Omega)$ such that $u_j \rightarrow u$ in $L^{1}(\Omega)$ and ${^{\pm}}{\mathcal{D}}{^{\alpha}} u_j \rightarrow v$ in $L^{1}_{loc}(\Omega)$ as $j \rightarrow \infty$. 
        \end{theorem}

        \begin{proof}
            Let $u \in L^{1}(\Omega)$ and $u^{\eps}$ denote its mollification. 

            {\em Step 1:}  Suppose that $v = {^{\pm}}{\mathcal{D}}{^{\alpha}} u \in L^{1}_{loc}(\Omega)$.  Let $\tilde{u}^{\eps}$ denote the  mollification  of $\tilde{u}$. By the properties of mollification, $\tilde{u}^{\eps} \rightarrow u$ in $L^{1}(\Omega)$ as $\eps\to 0$. By the previous lemma, we have  ${^{\pm}}{\mathcal{D}}{^{\alpha}} \tilde{u}^{\eps} = \eta_{\eps} * {^{\pm}}{\mathcal{D}}{^{\alpha}}u 
             \rightarrow {^{\pm}}{\mathcal{D}}{^{\alpha}} u$ in $L^{1}_{loc}(\Omega)$ as 
            $\eps\to 0$. Hence, $\{\tilde{u}^{\eps} \}$ is a desired sequence. 
            
            \smallskip
            {\em Step 2:}  Suppose that $\left\{u_{j} \right\}_{j=1}^{\infty} \subset C^{\infty}(\Omega)$ and $u_j \rightarrow u$ in $L^{1}(\Omega)$ and ${^{\pm}}{\mathcal{D}}{^{\alpha}} u_{j} \rightarrow v$ in $L^{1}_{loc} ( \Omega)$. Then for any $\varphi \in C^{\infty}_{0}(\Omega)$
            \begin{align*}
                \left|\int_{\Omega} (u -u_j)(x) {^{\mp}}{D}{^{\alpha}} \varphi(x) \,dx \right| &\leq M\|u - u_j \|_{L^{1}(\Omega)}\to 0 \qquad \mbox{as } j\to \infty,
            \end{align*}
            and  
            \begin{align*}
                \left|\int_{\Omega} \left({^{\pm}}{\mathcal{D}}{^{\alpha}} u_j - v\right)(x) \varphi(x) \,dx \right| &= \left| \int_{K} \left({^{\pm}}{\mathcal{D}}{^{\alpha}} u_j - v\right)(x) \varphi(x) \,dx \right|\\
                &\leq M \left\|{^{\pm}}{\mathcal{D}}{^{\alpha}} u_j - v \right\|_{L^{1}(K)} \rightarrow 0
                \quad\mbox{as } j\to \infty,
            \end{align*}
            because $K:= \supp(\varphi)$ is compact.  It follows from the definition of weak fractional derivatives that 
            \begin{align*}
                (-1)^{[\alpha]} \int_{\Omega} u(x) {^{\mp}}{D}{^{\alpha}} \varphi (x) \, dx
                &= (-1)^{[\alpha]} \lim_{j \rightarrow \infty} \int_{\Omega} u_j(x) {^{\mp}}{D}{^{\alpha}} \varphi (x) \, dx \\ 
                &= \lim_{j \rightarrow \infty} \int_{\Omega} {^{\pm}}{\mathcal{D}}{^{\alpha}} u_{j}(x) \varphi(x) \, dx 
                = \int_{\Omega} v(x) \varphi(x) \, dx.
            \end{align*}
            By the uniqueness of the weak fractional derivative, we conclude that $v = {^{\pm}}{\mathcal{D}}{^{\alpha}} u$ almost everywhere.  The proof is complete. 
        \end{proof}
        
       \begin{corollary}
            Let $u \in L^{p}(\Omega)$ for $1 \leq p < \infty$. Then $v = {^{\pm}}{\mathcal{D}}{^{\alpha}} u$ in  $L^{q}_{loc} (\Omega)$ for $1\leq q <\infty$ if and only if there exists a sequence $\left\{u_j \right\}_{j=1}^{\infty} \subset C^{\infty} (\Omega)$ such that $u_j \rightarrow u$ in $L^{p}(\Omega)$ and ${^{\pm}}{\mathcal{D}}{^{\alpha}} u_j \rightarrow v$ in $L^{q}_{loc}(\Omega)$ as $j \rightarrow \infty$.
           
        \end{corollary}
      
        \begin{remark}
            The conclusion of the above corollary still holds if $L^q_{loc}(\Omega)$ is replaced by $L^q(\Omega)$
            in the statement.
        \end{remark}

    \subsubsection{\bf The Infinite Domain Case}\label{sec-4.3.2}
    We now consider the case $\Omega = \R$. It turns out this case is significantly different
    from the finite interval case. In particular, it requires the construction of a compactly supported 
    approximation sequence for each fractionally differentiable function, which turns out is quite complicated. 
    
    First, we establish the following  analogue of Lemma \ref{WDMollifier}.
    
        \begin{lemma}\label{WDMollifierR}
        Suppose ${^{\pm}}{\mathcal{D}}{^{\alpha}}u \in L^1_{loc}(\R)$ exists, then 
             \begin{align} \label{WeakMollifier2}
                {^{\pm}}{\mathcal{D}}{^{\alpha}} u^\eps = \eta_{\eps} * {^{\pm}}{\mathcal{D}}{^{\alpha}} u \qquad\mbox{a.e. in } \R. 
            \end{align}
 
        \end{lemma}
        
        \begin{proof}
        	By Lemma \ref{FourierMollifier} and Corollary \ref{RLMollifier}, we have for any 
        	$\varphi \in C^{\infty}_{0}(\R)$
            \begin{align*}
                \int_{\R} {^{\pm}}{\mathcal{D}}{^{\alpha}} u^{\eps} (x) \varphi(x) \,dx &= (-1)^{[\alpha]} \int_{\R} \left(\eta_{\eps}* u \right)(x) {^{\mp}}{D}{^{\alpha}}\varphi(x)\,dx \\ 
                &= (-1)^{[\alpha]} \int_{\R} \int_{\R} \eta_{\eps} (x-y) u (y) {^{\mp}}{D}{^{\alpha}}\varphi(x) \,dydx\\
                &= (-1)^{[\alpha]} \int_{\R} \int_{\R} u(y) \eta(y-x) {^{\mp}}{D}{^{\alpha}}\varphi(x)\,dxdy \\ 
                &= (-1)^{[\alpha]} \int_{\R}u(y) \left(\eta_{\eps} *{^{\mp}}{D}{^{\alpha}}\varphi\right)(y) \,dy\\
                &= (-1)^{[\alpha]} \int_{\R} u(y) {^{\mp}}{D}{^{\alpha}}\varphi^{\eps}(y) \,dy \\ 
                &= \int_{\R} {^{\pm}}{\mathcal{D}}{^{\alpha}}u(y)\varphi^{\eps}(y)\,dy \\
                &= \int_{\R} {^{\pm}}{\mathcal{D}}{^{\alpha}}u(y) \left(\eta_{\eps} * \varphi\right)(y)\,dy \\ 
                &= \int_{\R}\int_{\R} {^{\pm}}{\mathcal{D}}{^{\alpha}} u(y) \eta_{\eps} ( y -x) \varphi(x)\,dxdy \\ 
                &= \int_{\R} \int_{\R} \eta_{\eps}(x-y){^{\pm}}{\mathcal{D}}{^{\alpha}} u(y) \varphi(x)\,dydx \\ 
                &= \int_{\R} \left(\eta_{\eps} * {^{\pm}}{\mathcal{D}}{^{\alpha}}u \right)(x) \varphi(x)\,dx.
            \end{align*}
            Hence, \eqref{WeakMollifier2} holds. The proof is complete. 
        \end{proof}
        
        The next theorem gives a characterization of weak fractional derivatives on $\R$. 
        
         \begin{theorem}\label{characterization}
            Suppose $u \in L^{1}(\R)$. Then ${^{\pm}}{\mathcal{D}}{^{\alpha}} u = v \in L^{1}_{loc}(\R)$ exists if and only if there 
            exists a sequence 
              $\left\{u_j \right\}_{j=1}^{\infty} \subset C^{\infty}_{0}(\R)$ such that $u_j \rightarrow u$ in $L^{1}(\R)$ and ${^{\pm}}{\mathcal{D}}{^{\alpha}} u_j \rightarrow v$ in $L^{1}_{loc}(\R)$.
        \end{theorem}

        \begin{proof}
            {\em Step 1:}  Suppose that there exists $v \in L^{1}_{loc}(\R)$ and $\left\{u_j \right\}_{j=1}^{\infty} \subset C^{\infty}_{0}(\R)$
             such that $u _ j \rightarrow u$ in $L^{1}(\R)$ and ${^{\pm}}{\mathcal{D}}{^{\alpha}}u_j     \rightarrow v$ in $L^{1}_{loc}(\R)$. We want to show $v = {^{\pm}}{\mathcal{D}}{^{\alpha}} u$ almost everywhere. For any $\varphi\in C^\infty_0(\R)$  
            \begin{align*}
                \left| \int_{\R} (u-u_j )(x) {^{\mp}}{D}{^{\alpha}}\varphi(x) \,dx \right| &\leq \int_{\R} \left| (u - u_j)(x) \right| \left| {^{\mp}}{D}{^{\alpha}}\varphi (x) \right|\,dx\\
                &\leq \|u - u_j \|_{L^{1}(\R)} \left\|{^{\mp}}{D}{^{\alpha}} \varphi \right\|_{L^{\infty}(\R)} \to 0   
            \end{align*} 
       for $j\to \infty$ and for $K : = \supp(\varphi)$
            \begin{align*}
                \left| \int_{\R} \left(v - {^{\pm}}{\mathcal{D}}{^{\alpha}} u_j\right)(x) \varphi(x)\,dx  \right| &\leq \int_{\R} \left|\left(v - {^{\pm}}{\mathcal{D}}{^{\alpha}} u_j\right)(x)\right| |\varphi(x)|\,dx\\
                &= \int_{K} \left| \left(v - {^{\pm}}{\mathcal{D}}{^{\alpha}} u_j \right) (x) \right| \left|\varphi(x) \right| \,dx\\
                &\leq \left\| v - {^{\pm}}{\mathcal{D}}{^{\alpha}}u_{j}\right\|_{L^{1}(K)} \| \varphi \|_{L^{\infty}(\R)} \to 0   
            \end{align*}
         for $j\to \infty$. 
         It follows from the above inequalities and the definition of  weak fractional derivatives that 
            \begin{align*}
                (-1)^{[\alpha]} \int_{\R} u(x) {^{\mp}}{D}{^{\alpha}} \varphi (x)\,dx &= \lim_{j \rightarrow \infty} (-1)^{[\alpha]} \int_{\R} u_j(x) {^{\mp}}{D}{^{\alpha}}\varphi(x)\,dx \\ 
                &=\lim_{j \rightarrow \infty} \int_{\R} {^{\pm}}{\mathcal{D}}{^{\alpha}} u_j (x)  \varphi(x) \,dx 
                = \int_{\R} v(x) \varphi(x)\,dx. 
            \end{align*}
            By the uniqueness of the  weak fractional derivative, we conclude that $v = {^{\pm}}{\mathcal{D}}{^{\alpha}} u$ almost everywhere.
            
            \smallskip
            {\em Step 2:} Suppose that $u \in L^{1}(\R)$ and $v := {^{\pm}}{\mathcal{D}}{^{\alpha}} u 
            \in L^{1}_{loc}(\R)$. We want to show that there exists $\left\{ u_j 
             \right\}_{j=1}^{\infty} \subset C^{\infty}_{0}(\R)$ such that $u_j \rightarrow u$ in $L^{1}(\R)$
             and ${^{\pm}}{\mathcal{D}}{^{\alpha}} u_j \rightarrow v$ in $L^{1}_{loc}(\R)$. 
            To the end, let $\psi \in C^{\infty}(\R)$ satisfy $\psi(t) =1$ if $t \leq 0$ and $\psi (t) = 0$ if 
            $t \geq 1$. For $j =1 ,2,3,...$ let $\psi_{j} \in C^{\infty}_{0}(\R)$ be defined by 
            $\psi_{j}(x) : = \psi(|x| - j)$. Let $u_j : =\eta_{\frac{1}{j}} * (\psi_{j}u)$. Then $u_j\in C^\infty_0(\R)$ and $u_j \rightarrow u$ in $L^{1}(\R)$ as $j \rightarrow \infty$.
            We also claim that ${^{\pm}}{\mathcal{D}}{^{\alpha}} u_j \rightarrow v$ in $L^{1}_{loc}(\R)$  as
            $j\to \infty$ and prove this conclusion below in the subsequent corollary.
       \end{proof}
   
   \medskip
   \begin{corollary}\label{corollary4.8}
   	Suppose $u \in L^{p}(\R)$ for $1 \leq p < \infty$. Then $v := {^{\pm}}{\mathcal{D}}{^{\alpha}} u \in L^{q}_{loc}(\R)$
   	for $1\leq q <\infty$ 
   	if and only if there exists $\left\{u_j \right\}_{j=1}^{\infty} \subset C^{\infty}_{0}(\R)$ such that $u_j \rightarrow u$ in $L^{p}(\R)$ and ${^{\pm}}{\mathcal{D}}{^{\alpha}} u_j \rightarrow v$ in $L^{q}_{loc}(\R)$.
   \end{corollary}

          \begin{proof}
            {\em Step 1:}  Same as {\em Step 1} of the proof of Theorem \ref{characterization}. 
            
            {\em Step 2:}  Suppose that $u \in L^{p}(\R)$ for $1 \leq p < \infty$ and $v := {^{\pm}}{\mathcal{D}}{^{\alpha}} u 
            \in L^{q}_{loc}(\R)$. Let $\{u_j\}_{j=1}^{\infty}$ be the same as in the proof of Theorem \ref{characterization} and $\eps>0$. We now want to show that ${^{\pm}}{\mathcal{D}}{^{\alpha}} u_j 
            \rightarrow v$ in $L^{q}_{loc}(\R)$ as $j\to \infty$.

            By the assumption, we have $v={^{\pm}}{\mathcal{D}}{^{\alpha}} u \in L^{q}_{loc}(\R)$. 
            For any fixed compact subset $K \subset \R$, choose $a,b \in \R$ such that $K \subset (a , b)$ finite. 
            By the construction of $u_j$, we have ${^{\pm}}{\mathcal{D}}{^{\alpha}} u_j
            =\eta_{\frac{1}{j}} * {^{\pm}}{\mathcal{D}}{^{\alpha}}(\psi_{j} u)$. 
            Let $K_j : = \text{supp}(\psi_j)$ and for every $\varphi\in C^{\infty}_{0}(\R)$ with 
            $\text{supp}(\varphi) \subset K_j$ we have  
            \begin{align*}
                \int_{K_j} {^{\pm}}{\mathcal{D}}{^{\alpha}}(\psi_j u)(x) \varphi(x)\,dx &= \int_{\R} {^{\pm}}{\mathcal{D}}{^{\alpha}}(\psi_j u)(x) \varphi(x)\,dx \\ 
                &=\int_{\R} (\psi_j u)(x) {^{\mp}}{D}{^{\alpha}} \varphi(x)\,dx 
                = \int_{K_j} (\psi_ju)(x) {^{\mp}}{D}{^{\alpha}} \varphi(x)\,dx. 
            \end{align*}
            Hence, ${^{\pm}}{\mathcal{D}}{^{\alpha}} (\psi_j u)$ can be regarded as the weak 
            fractional derivative of $\psi_j u$ over the domain $K_j$. It is due to this fact that we could 
            use the product rule with remainder for fractional weak derivatives (to be proved  in Theorem \ref{WeakProductRule})
            to get  
            \begin{align*}
                {^{\pm}}{\mathcal{D}}{^{\alpha}} (\psi_j u ) (x) = \psi_j (x) {^{\pm}}{\mathcal{D}}{^{\alpha}} u(x) + \sum_{k=1}^{m} C_{k} {^{\pm}}{I}{^{k-\alpha}} u(x) D^{k}\psi_{j}(x) + {^{\pm}}{R}{^{\alpha}_{m}}(u,\psi_j)(x). 
            \end{align*}
            Therefore, 
            \begin{align*}
                & \Bigl\| {^{\pm}}{\mathcal{D}}{^{\alpha}} u - {^{\pm}}{\mathcal{D}}{^{\alpha}} u_j\Bigr\|_{L^{q}(K)} 
                =\Bigl\| {^{\pm}}{\mathcal{D}}{^{\alpha}} u - \eta_{\frac{1}{j}} * {^{\pm}}{\mathcal{D}}{^{\alpha}} (\psi_{j} ,u)\Bigr\|_{L^{q}(K)}\\ 
                &\qquad =\Bigl\|{^{\pm}}{\mathcal{D}}{^{\alpha}} u - \eta_{\frac{1}{j}} * \Bigl( \psi_j {^{\pm}}{\mathcal{D}}{^{\alpha}} u + \sum_{k=1}^{m} C_{k} {^{\pm}}{I}{^{\alpha}} u D^{k} \psi_{j} + {^{\pm}}{R}{^{\alpha}_{m}} (u,\psi_j)\Bigr) \Bigr\|_{L^{q}(K)}\\
                &\qquad \leq \Bigl\|{^{\pm}}{\mathcal{D}}{^{\alpha}} u - \eta_{\frac{1}{j}} * \psi_{j} {^{\pm}}{\mathcal{D}}{^{\alpha}} u \Bigr\|_{L^{q}(K)} + \Bigl\| \eta_{\frac{1}{j}} * \sum_{k=1}^{m} C_{k} {^{\pm}}{I}{^{\alpha}} u D^{k} \psi_{j} \Bigr\|_{L^{q}(K)} \\
                &\hskip 1.0in + \Bigl\| {^{\pm}}{R}{^{\alpha}_{m}} (u,\psi_j) \Bigr\|_{L^{q}(K)}\\
                &\qquad  \leq \Bigl\|{^{\pm}}{\mathcal{D}}{^{\alpha}} u - \eta_{\frac{1}{j}} * \psi_{j} {^{\pm}}{\mathcal{D}}{^{\alpha}} u \Bigr\|_{L^{q}(K)} + \sum_{k=1}^{m} \Bigl\| \eta_{\frac{1}{j}} * C_{k} {^{\pm}}{I}{^{\alpha}} u D^{k} \psi_{j} \Bigr\|_{L^{q}(K)} \\
                &\hskip 1.0in + \left\| {^{\pm}}{R}{^{\alpha}_{m}} (u,\psi_j) \right\|_{L^{q}(K)}.
            \end{align*}
            Then it suffices to show that each of the above three terms vanishes as $j\to \infty$. 
            
            Since ${^{\pm}}{\mathcal{D}}{^{\alpha}} u \in L^{q}(K)$, by the same arguments used 
            to show that $u_j \rightarrow u$ in $L^{p}(\R)$, we have that $\eta_{\frac{1}{j}} * \psi_j {^{\pm}}{\mathcal{D}}{^{\alpha}} u \rightarrow {^{\pm}}{\mathcal{D}}{^{\alpha}} u$ in $L^{q}(K)$. Hence, there exists $J_{1} \in \N$ such that for every $j\geq J_{1}$, we have that 
            \begin{align*}
                \left\| {^{\pm}}{\mathcal{D}}{^{\alpha}} u - \eta_{\frac{1}{j}} * \psi_{j} {^{\pm}}{\mathcal{D}}{^{\alpha}} u \right\|_{L^{q}(K)} < \dfrac{\eps}{2}. 
            \end{align*}

            Next, by construction, for set $K$, there exists $J_{2}:= J_{2}(K) \in \N$ so that for every $j\geq J_{2}$, $D^{k} \psi_j(x) = 0$ for every $x \in (a,b)$. Therefore, for every $j \geq J_{2},$
            \begin{align*}
                \left\|\eta_{\frac{1}{j}} * C_{k} {^{\pm}}{I}{^{k-\alpha}} u D^{k}\psi_j \right\|_{L^{q}(K)} 
                &\leq \left\| C_{k} {^{\pm}}{I}{^{k-\alpha}} u D^{k}\psi_j \right\|_{L^{q}(K)} \\
                &\leq \left\|C_{k} {^{\pm}}{I}{^{k-\alpha}} u D^{k} \psi_j\right\|_{L^{q}((a,b))} = 0.
            \end{align*}
            Here we have used  the fact that for each $k$ and $j \geq J_{2}$, ${^{\pm}}{I}{^{k-\alpha}} u$ is finite on $(a,b)$. This of course is true since $u \in L^{p}(\R)$. In fact, we need only that ${^{\pm}}{I}{^{k-\alpha}} u$ is finite on $(a,b)$ for $j = J_{2}$ since for all $j \geq J_{2}$, $D^{k} \psi_{j} \equiv 0$ in $(a,b)$. 

            Finally, we need to show that the remainder term vanishes as $j \rightarrow \infty$. For illustrative 
             purposes, we only show this for the left remainder term because the right remainder term follows similarly. 
             Let $a_j \in \R$ be the left endpoint of $K_{j}$ and $b_j$ be the right endpoint of $K_{j}$. By 
             construction, $\supp\bigl(\psi_{j}^{(m+1)}\bigr) \subseteq [a_j , a_j +1] \cup [b_j -1 , b_j]$. Choose $j$ 
             sufficiently large so that $\supp\bigl(\psi_{j}^{(m+1)}\bigr) \cap K = \emptyset$. It follows for the case $p \neq 1$ with $p'$ being the H\"older conjugate of $p$, there exists a sufficiently large $j$ so that,
            \begin{align*}
                &\left\|\eta_{\frac{1}{j}} * {^{-}}{R}{^{\alpha}_{m}} (u,\psi_j) \right\|_{L^{q}(K)}^{q} 
                \leq \left\| {^{-}_{}}{R}{^{\alpha}_{m}} (u,\psi_j) \right\|_{L^{q}((a,b))}^{q} \\ 
                &\quad = \int_{a}^{b} \Bigl|C(m,\alpha) \int_{a_j}^{x} \int_{y}^{x} \dfrac{u(y)}{(x-y)^{1 + \alpha}} \psi_{j}^{(m+1)}(z) (x-z)^{m} \,dzdy\Bigr|^{q}\,dx\\
                &\quad = \int_{a}^{b}\Bigl|C(m,\alpha)\int_{a_{j}}^{x} \int_{y}^{a_{j}+1} \dfrac{u(y)}{(x-y)^{1+\alpha}} \psi_{j}^{(m+1)} (z) (x-z)^{m}\,dzdy \Bigr|^{q}\,dx \\
                &\quad \leq \int_{a}^{b} \Bigl|C(m,\alpha) \int_{a_j}^{a_{j+1}} \int_{y}^{a_{j}+1} \dfrac{u(y)}{(x-y)^{1+\alpha}} \psi_{j}^{(m+1)} (z) (x-z)^{m}\,dzdy \Bigr|^{q}\,dx\\
                &\qquad\quad  + \int_{a}^{b} \Bigl|C(m,\alpha) \int_{a_{j+1}}^{x} \int_{y}^{a_{j}+1} \dfrac{u(y)}{(x-y)^{1+\alpha}} \psi_{j}^{(m+1)} (z) (x-z)^{m}\,dzdy \Bigr|^{q}\,dx\\
                &\quad =\int_{a}^{b} \Bigl|C(m,\alpha) \int_{a_j}^{a_{j+1}} \int_{y}^{a_{j}+1} \dfrac{u(y)}{(x-y)^{1+\alpha}} \psi_{j}^{(m+1)} (z) (x-z)^{m}\,dzdy \Bigr|^{q}\,dx\\
                &\quad \leq \int_{a}^{b} \Bigl|C(m,\alpha,\psi) \int_{a_j}^{a_{j+1}} \int_{y}^{a_{j}+1} \dfrac{u(y)}{(x-y)^{1+ \alpha}} (x-z)^{m} \,dzdy\Bigr|^{q}\,dx\\
                &\quad \leq \int_{a}^{b} \Bigl|C(m,\alpha,\psi) \int_{a_j}^{a_{j}+1} \int_{a_j}^{a_{j}+1} \dfrac{u(y)}{(x-y)^{1+ \alpha}} \,dzdy\Bigr|^{q}\,dx\\
                &\quad = \int_{a}^{b} \Bigl|C(m,\alpha,\psi) \int_{a_j}^{a_{j}+1} \dfrac{u(y)}{(x-y)^{1+\alpha}} \,dy\Bigr|^{q}\,dx \\ 
                &\quad \leq C(m , \alpha, \psi)^{q} \int_{a}^{b}  \left(\int_{a_j}^{a_{j}+1} |u(y)|^{p} \,dy \right)^{q/p} \left(\int_{a_{j}}^{a_{j}+1} \dfrac{dy}{(x-y)^{p'(1+\alpha)}}\,dy\right)^{\frac{q}{p'}} dx\\
                &\quad \leq C(m, \alpha , \psi)^{q}\|u\|_{L^{p}(\R)}^{q} \int_{a}^{b} \left( \dfrac{1}{(x-a_j)^{p'(1+\alpha) -1} }- \dfrac{1}{(x-a_{j}+1)^{p'(1 + \alpha) -1}}\right)^{\frac{q}{p'}} dx\\
                &\quad \leq C(m, \alpha , \psi)^{q} \|u\|_{L^{p}(\R)}^{q} \left(\int_{a}^{b} \dfrac{dx}{(x-a_j)^{q(1+ \alpha)-q/p'}} + \int_{a}^{b} \dfrac{dx}{(x-a_j -1)^{q(1+ \alpha )-q/p'}} \right)\\
                &\quad \leq C(m,\alpha,\psi)^{q} \|u\|_{L^{p}(\R)}^{q} \left(\int_{a}^{b} \dfrac{dx}{(a-a_j)^{q(1+ \alpha)-q/p'}}+ \int_{a}^{b} \dfrac{dx}{(a-a_j - 1)^{q(1+ \alpha )-q/p'}}\right) \\
                &\quad = C(m,\alpha,\psi)^{q} \|u\|_{L^{p}(\R)}^{q} \left( \dfrac{b-a}{(a-a_j)^{q(\alpha +1/p)}} +  \dfrac{b-a}{(a-a_j - 1)^{q(\alpha +1/p)}}\right)
                 <\dfrac{\eps}{2}
            \end{align*}
            In the case that $p=1$, we can recycle the above calculation and input the estimate 
            \begin{align*}
                \int_{a}^{b} \sup_{y \in (a_j, a_{j}+1)} |(x-y)^{-1-\alpha}|\,dx &= \int_{a}^{b} (x-a_j -1)^{-1 - \alpha}\,dx \\ 
                &\leq C(\alpha) \left( (b-a_j -1)^{-\alpha} -(a - a_j -1)^{-\alpha} \right).
            \end{align*}
            Taking $j \rightarrow \infty$, we have that the remainder vanishes, as desired. This concludes that given ${^{\pm}}{\mathcal{D}}{^{\alpha}} u \in L^{q}_{loc}(\R)$, then there exists $\left\{u_j \right\}_{j=1}^{\infty} \subset C^{\infty}_{0}(\R)$ so that $u_j \rightarrow u$ in $L^{p}(\R)$ and ${^{\pm}}{\mathcal{D}}{^{\alpha}} u_j \rightarrow {^{\pm}}{\mathcal{D}}{^{\alpha}} u$  in $L^{q}_{loc}(\R)$ as $j\to \infty$.
        \end{proof}

    \begin{remark}
    	The conclusion of the above corollary still holds if $L^q_{loc}(\R)$ is replaced by $L^q(\R)$ in the statement. 
    \end{remark}

    \subsection{Basic Properties of Weak Fractional Derivatives}\label{sec-4.4}  
    Similar to the classical calculus theory, we wish for weak derivatives to satisfy certain properties and rules 
    of calculus. As in the classical fractional calculus, many of the rules in the weak fractional calculus theory differ 
    from their integer counterparts, which is expected. Below we collect a few elementary properties for
    weak fractional derivatives.
    
    \begin{proposition} \label{properties}
    	Let  $\alpha , \beta >0$, $\lambda ,\mu \in \R$, and $u,v$ be weakly differentiable to the appropriate order. Then the following properties hold.  
    \begin{itemize}
    \item[{\rm(i)}] Linearity: ${^{\pm}}{\mathcal{D}}{^{\alpha}}(\lambda u + \mu v) = \lambda {^{\pm}}{\mathcal{D}}{^{\alpha}} u + \mu {^{\pm}}{\mathcal{D}}{^{\alpha}} v$.
    \item[{\rm (ii)}] Inclusivity: Let $0 < \alpha < \beta < 1$, suppose that $u$ is $\beta$ order weakly differentiable. Then $u$ is $\alpha$ order weakly differentiable.    
     \item[{\rm (iii)}] Semigroup:  suppose $ 0<\alpha, \beta, \alpha+\beta<1$ and 
    ${^{\pm}}{\mathcal{D}}{^{\alpha}} u, {^{\pm}}{\mathcal{D}}{^{\beta}} u, {^{\pm}}{\mathcal{D}}{^{\alpha+\beta}} u\in L^1(\Omega)$, then ${^{\pm}}{\mathcal{D}}{^{\alpha}} {^{\pm}}{\mathcal{D}}{^{\beta}} u = {^{\pm}}{\mathcal{D}}{^{\alpha+\beta}} u$. Moreover, if $\alpha>1$, then  ${^{\pm}}{\mathcal{D}}{^{\alpha}} u
    ={^{\pm}}{\mathcal{D}}{^{[\alpha]+\sigma}} u=  {\mathcal{D}}{^{[\alpha]}} ({^{\pm}}{\mathcal{D}}{^{\sigma}} u)$ with $\sigma:=\alpha- [\alpha]$.
    \item[{\rm (iv)}] Consistency: if $u$ is first-order weakly differentiable, then the $\alpha\, (<1)$ order weak derivative coincides with the first-order weak derivative in the limit as $\alpha\to 1$.
    \end{itemize}
    \end{proposition}
    
    \begin{proof}
    (i)  It is easy to see that
	\begin{align*}
		\int_{\Omega} (\lambda u + \mu v) {^{\mp}}{D}{^{\alpha}} \tilde{\varphi}\,dx &= \int_{\Omega} \lambda u {^{\mp}}{D}{^{\alpha}} \tilde{\varphi}\,dx + \int_{\Omega} \mu v {^{\mp}}{D}{^{\alpha}} \tilde{\varphi}\,dx  \\ 
		&= \lambda \int_{\Omega} u {^{\mp}}{D}{^{\alpha}} \tilde{\varphi}\,dx + \mu \int_{\Omega} v {^{\mp}}{D}{^{\alpha}} \tilde{\varphi}\,dx \\ 
		&= \lambda (-1)^{[\alpha]}\int_{\Omega} {^{\pm}}{\mathcal{D}}{^{\alpha}} u \varphi\,dx + \mu (-1)^{[\alpha]} \int_{\Omega} {^{\pm}}{\mathcal{D}}{^{\alpha}} v \varphi\,dx   \\ 
		&= (-1)^{[\alpha]}\int_{\Omega} (\lambda {^{\pm}}{\mathcal{D}}{^{\alpha}} u + \mu {^{\pm}}{\mathcal{D}}{^{\alpha}} v) \varphi\,dx 
		\end{align*}
		for every $\varphi \in C^{\infty}_{0}(\Omega)$. Hence, assertion (i) holds. 
	
	(ii) We shall postpone this proof until after the Fundamental Theorem of Weak Fractional Calculus (FTwFC, cf. Theorem \ref{WeakFTFC}) is established.
	
	(iii) If $0<\alpha,\beta, \alpha+\beta<1$, by the definition we have 
		\begin{align}\label{ee1}
		\int_{\Omega} {^{\pm}}{\mathcal{D}}{^{\alpha+\beta}} u\, \varphi\,dx
		&=  \int_{\Omega} u {^{\mp}}{D}{^{\alpha + \beta}} \varphi\,dx \qquad\forall \varphi\in C^\infty_0(\Omega), \\
		\int_{\Omega} {^{\pm}}{\mathcal{D}}{^{\beta}} {^{\pm}}{\mathcal{D}}{^{\alpha}} u\, \varphi \,dx 
		&=  \int_{\Omega} {^{\pm}}{\mathcal{D}}{^{\alpha}} u {^{\mp}}{D}{^{\beta}} \varphi\,dx  \qquad\forall \varphi\in C^\infty_0(\Omega).\label{ee2} 
		\end{align}
		Let $\{u_j\}_{j=1}^{\infty} \subset C^\infty (\Omega)$ such that $u_j\to u$ in $L^1(\Omega)$ and ${^{\pm}}{D}{^{\alpha}}u_j \to {^{\pm}}{D}{^{\alpha}}u$ in $L^1(\Omega)$, then using the integration by parts formula for Riemann-Liouville fractional order derivatives, we obtain 
		\begin{align}\label{ee3}
		\int_{\Omega} {^{\pm}}{\mathcal{D}}{^{\alpha}} u {^{\mp}}{D}{^{\beta}} \varphi\,dx
		&= \lim_{j \rightarrow \infty} \int_{\Omega} {^{\pm}}{D}{^{\alpha}}u_j {^{\mp}}{D}{^{\beta}} \varphi\,dx\\
		&= \lim_{j \rightarrow \infty}  \int_{\Omega}  u_j {^{\mp}}{D}{^{\alpha}}{^{\mp}}{D}{^{\beta}} \varphi\,dx \nonumber\\
		&= \lim_{j \rightarrow \infty}  \int_{\Omega}  u_j {^{\mp}}{D}{^{\alpha+\beta}}  \varphi\,dx
		=   \int_{\Omega}  u {^{\mp}}{D}{^{\alpha+\beta}}  \varphi\,dx. \nonumber
 		\end{align}
 		Combining \eqref{ee1}--\eqref{ee3} we get 
 		\[
 		\int_{\Omega} {^{\pm}}{\mathcal{D}}{^{\alpha+\beta}} u \varphi\,dx
 		= \int_{\Omega} {^{\pm}}{\mathcal{D}}{^{\beta}} {^{\pm}}{\mathcal{D}}{^{\alpha}} u \varphi\,dx
 		\qquad\forall \varphi\in C^\infty_0(\Omega).
 		\]
 		Thus, ${^{\pm}}{\mathcal{D}}{^{\alpha+\beta}} u=  {^{\pm}}{\mathcal{D}}{^{\beta}} {^{\pm}}{\mathcal{D}}{^{\alpha}} u$ almost everywhere in $\Omega$. 
 		
 		If $\alpha>1$, set $m=[\alpha]$ and $\sigma=\alpha-m$. By the definition we get for any $\varphi\in C^\infty_0(\Omega)$,
 		\begin{align*}
 		\int_{\Omega} {\mathcal{D}}{^{[\alpha]}} ({^{\pm}}{\mathcal{D}}{^{\sigma}} u) \varphi\,dx
 		&=(-1)^{[\alpha]} \int_{\Omega}   {^{\pm}}{\mathcal{D}}{^{\sigma}} u \, {\mathcal{D}}{^{[\alpha]}} \varphi\,dx
 		=(-1)^{[\alpha]} \int_{\Omega}     u  \, {^{\mp}}{\mathcal{D}}{^{\sigma}}  {\mathcal{D}}{^{[\alpha]}} \varphi\,dx \\
 		&=(-1)^{[\alpha]} \int_{\Omega}     u  \, {^{\mp}}{\mathcal{D}}{^{\sigma +[\alpha]}} \varphi\,dx 
 		=(-1)^{[\alpha]} \int_{\Omega}  u  \, {^{\mp}}{\mathcal{D}}{^\alpha} \varphi\,dx.
 		\end{align*}
 		Thus, ${^{\pm}}{\mathcal{D}}{^{\alpha}} u
 		 =  {\mathcal{D}}{^{[\alpha]}} ({^{\pm}}{\mathcal{D}}{^{\sigma}} u)$ almost everywhere in $\Omega$ and
        the assertion (iii) is proved. 
 
	  (iv)  It follows by the consistency of the classical fractional derivatives that for every $\varphi \in C^{\infty}(\Omega)$, 
		\begin{align*}
			\lim_{\alpha \rightarrow 1} \int_{\Omega} {^{\pm}}{\mathcal{D}}{^{\alpha}} u\, \varphi\,dx &:= \lim_{\alpha \rightarrow 1} (-1)^{[\alpha]} \int_{\Omega} u {^{\mp}}{D}{^{\alpha}} \tilde{\varphi} \,dx
			 = - \int_{\Omega} u D\tilde{\varphi} \,dx
			=:  \int_{\Omega} \mathcal{D} u\, \varphi \,dx
		\end{align*}

    \end{proof}
    
    \begin{remark}
    	 We note that for $\alpha>1$,  generally, ${^{\pm}}{\mathcal{D}}{^{\alpha}} u
    	 \neq  {^{\pm}}{\mathcal{D}}{^{\sigma}}  {\mathcal{D}}{^{[\alpha]}}  u$, consequently, 
    	 ${^{\pm}}{\mathcal{D}}{^{\sigma}}  {\mathcal{D}}{^{[\alpha]}}  u  \neq {\mathcal{D}}{^{[\alpha]}} {^{\pm}}{\mathcal{D}}{^{\sigma}}   u$, 
    	  in general.
    \end{remark}

    We conclude this section by stating a general integration by parts formula in the case $\Ome=\R$. To extend such
    a formula to the finite domain case must be delayed to Section \ref{sec-5} after the notion of function traces
    is introduced. 
    
    \begin{proposition}
    	Let $\alpha>0$, $1\leq p_k\leq \infty$ and $q_k=\frac{p_k}{p_k-1}$ for $k=1,2$. Suppose that $u \in L^{p_1}(\R)$, $v\in L^{p_2}(\R)$,  ${^{\pm}}{\mathcal{D}}{^{\alpha}} u \in L^{q_2}(\R)$, and 
    		$ {^{\mp}}{\mathcal{D}}{^{\alpha}} v \in L^{q_1}(\R)$. Then there holds 
    	 \begin{align}\label{integration_by_parts} 
    	\int_{\R} {^{\pm}}{ \mathcal{\mathcal{D} }}{^{\alpha}} u\, v\,dx  
    	= (-1)^{[\alpha]} \int_{\R} u\, {^{\mp}}{\mathcal{D}}{^{\alpha}} v\,dx . 
    	\end{align}
    \end{proposition}
    
    \begin{proof}
    	By Corollary \ref{corollary4.8} we know that there exists a sequence $\{v_j\}_{j =1}^{\infty} \subset C^\infty_0(\R)$ 
    	such that $v_j\to v$ in $L^{p_2}(\R)$ and ${^{\mp}}{ \mathcal{\mathcal{D} }}{^{\alpha}} v_j \to 
    	{^{\mp}}{ \mathcal{\mathcal{D} }}{^{\alpha}} v $ in $L^{q_1}(\R)$ as $j\to \infty$. 
    	By the definition of $ {^{\pm}}{ \mathcal{\mathcal{D} }}{^{\alpha}} u$ we have 
    	 \begin{align*}
    	\int_{\R}  {^{\pm}}{ \mathcal{\mathcal{D} }}{^{\alpha}} u\, v_j\,dx 
    	= (-1)^{[\alpha]} \int_{\R} u\, {^{\mp}}{\mathcal{D} }{^{\alpha}} v_j\,dx . 
    	\end{align*}
    	Setting $j\to \infty$ immediately infers \eqref{integration_by_parts}. The proof is complete. 
    \end{proof} 
    

    \subsection{Product and Chain Rules for Weak Fractional Derivatives}\label{sec-4.5}
    In this subsection we present some product rules and chain rules for  weak fractional derivatives, which are
    similar to those rules for classical fractional derivatives given in Section \ref{sec-2.5}.

\begin{theorem}\label{WeakProductRule}
    Let $(a,b) \subset \R$ and $0<\alpha <1$.  Suppose that $\psi \in C^{m+1}([a,b])$ for $m \geq 1$ and ${^{\pm}}{\mathcal{D}}{^{\alpha}} u\in L^1_{loc}((a,b))$ exists. Then  ${^{\pm}}{\mathcal{D}}{^{\alpha}} (u \psi)$ 
    exists and is given by 
    \begin{align*}
        {^{\pm}}{\mathcal{D}}{^{\alpha}} (u \psi)(x) = {^{\pm}}{\mathcal{D}}{^{\alpha}} u(x) \cdot \psi(x) &+ \sum_{k=1}^{m} \dfrac{\Gamma(1 + \alpha)}{\Gamma(1+ k)\Gamma(1 - k + \alpha)} {^{\pm}}{I}{^{k-\alpha}} u(x) D^{k} \psi(x) \\
        &+ {^{\pm}}{R}{^{\alpha}_{m}}(u, \psi)(x) \qquad \mbox{a.e. in } (a,b),
    \end{align*}
    where 
    \begin{align*}
        {^{+}}{R}{^{\alpha}_{m}} (u,\psi)(x) = \dfrac{(-1)^{m+1}}{m! \Gamma(-\alpha)} \int_{x}^{b} \dfrac{u(y)}{(y -x )^{1 + \alpha}} \, dy \int_{x}^{y} \psi^{(m+1)} (z) (z - x)^{m} \, dz,\\
        {^{-}}{R}{^{\alpha}_{m}}(u,\psi)(x) = \dfrac{(-1)^{m+1}}{m! \Gamma(-\alpha)} \int_{a}^{x} \dfrac{u(y)}{(x-y)^{1 + \alpha}} \, dy \int_{y}^{x} \psi^{(m+1)}(z) (x -z)^{m} \, dz.
    \end{align*}
\end{theorem}

\begin{proof}
    Let $\left\{u_j \right\}_{j=1}^{\infty} \subset C^{\infty}((a,b))$ so that $u_j \rightarrow u$ in $ L^{1}((a,b))$ 
    and ${^{\pm}}{\mathcal{D}}{^{\alpha}} u_j \rightarrow {^{\pm}}{\mathcal{D}}{^{\alpha}} u$ in $
     L^{1}_{loc}((a,b))$. Consider the product $u_j \psi$, which belongs to $C((a,b))$,  and $\varphi \in C^{\infty}_{0}((a,b))$ with $\supp(\varphi) : = (c,d) \subset (a,b)$.  
      Since $u_j \rightarrow u$ in  $L^{1}(a,b))$, ${^{\pm}}{I}{^{\sigma}} u_j \rightarrow 
      {^{\pm}}{I}{^{\sigma}} u$ in $L^{1}((a,b))$. Using this fact and Theorem \ref{ProductRule}, 
      we obtain  
    \begin{align*}
        \int_{\Omega} u \psi {^{\mp}}{D}{^{\alpha}}\varphi\,dx
        &= \lim_{j \rightarrow \infty} \int_{\Omega} u_j \psi {^{\mp}}{D}^{{\alpha}} \varphi \,dx \\
        &= \lim_{j\rightarrow \infty} \int_{\Omega} {^{\pm}}{D}{^{\alpha}} (u_j \psi) \cdot \varphi \,dx
        = \lim_{j\rightarrow \infty} \int_{\Omega'} {^{\pm}}{D}{^{\alpha}} (u_j \psi) \cdot \varphi\,dx \\
        &= \lim_{j \rightarrow \infty} \int_{\Omega'} \left({^{\pm}}{D}{^{\alpha}} u_j \cdot \psi + \sum_{k=1}^{m} C_{k} {^{\pm}}{I}{^{k-\alpha}} u_j D^{k} \psi + {^{\pm}}{R}{^{\alpha}_{m}}(u_j,\psi) \right) \varphi\,dx \\ 
        &= \lim_{j \rightarrow \infty} \int_{\Omega'} \left({^{\pm}}{\mathcal{D}}{^{\alpha}} u_j \cdot \psi + \sum_{k=1}^{m} C_{k} {^{\pm}}{I}{^{k-\alpha}} u_j D^{k} \psi + {^{\pm}}{R}{^{\alpha}_{m}}(u_j,\psi) \right) \varphi\,dx\\
        &= \int_{\Omega'} \left({^{\pm}}{\mathcal{D}}{^{\alpha}} u \cdot \psi + \sum_{k=1}^{m} C_{k} {^{\pm}}{I}{^{k-\alpha}} u D^{k} \psi + {^{\pm}}{R}{^{\alpha}_{m}}(u,\psi) \right) \varphi\,dx\\
        &= \int_{\Omega} \left({^{\pm}}{\mathcal{D}}{^{\alpha}} u \cdot \psi + \sum_{k=1}^{m} C_{k} {^{\pm}}{I}{^{k-\alpha}} u D^{k} \psi + {^{\pm}}{R}{^{\alpha}_{m}(u,\psi)} \right) \varphi\,dx,
    \end{align*}
    which implies that 
    \begin{align*}
        {^{\pm}}{\mathcal{D}}{^{\alpha}} (u\psi) = {^{\pm}}{\mathcal{D}}{^{\alpha}} u \cdot \psi + \sum_{k=1}^{m} \dfrac{\Gamma(1+ \alpha)}{\Gamma(k+1) \Gamma(1 - k + \alpha)} {^{\pm}}{I}{^{k-\alpha}} u D^{k} \psi + {^{\pm}}{R}{_{m}^{\alpha}}(u,\psi)
    \end{align*}
    almost everywhere in $(a,b)$ with 
    \begin{align*}
        {^{+}}{R}{^{\alpha}_{m}} (u,\psi) (x) &= \dfrac{(-1)^{m}}{m! \Gamma(-\alpha)} \int_{x}^{b} \dfrac{u(y)}{(y -x )^{1 + \alpha}} \, dy \int_{x}^{y} \psi^{(m+1)} (z) (z - x)^{m} \, dz, \\
        {^{-}}{R}{^{\alpha}_{m}}(u,\psi)(x) & = \dfrac{(-1)^{m+1}}{m! \Gamma(-\alpha)} \int_{a}^{x} \dfrac{u(y)}{(x-y)^{1 + \alpha}} \, dy \int_{y}^{x} \psi^{(m+1)}(z) ( x-z)^{m} \, dz.
    \end{align*}
    The proof is complete.
\end{proof}

\begin{remark}
We also can prove another version of the product rules that do not include remainder terms. That version of the 
product rules will instead be written as infinite sums. However, just like in the classical fractional calculus theory, 
both functions are required to be analytic. Because we do not wish to make such an assumption in our applications 
of the  weak fractional derivative product rule, we omit this version of the product rules. 
\end{remark}
 
Based on the above product rules with $m=0$, we can easily obtain this following chain rules for weak fractional 
derivatives. We omit the proof because it is similar to the proof of Theorem \ref{thm_CR}. 

\begin{theorem}\label{thm_CR_weak}
	Let $(a,b)\subset \R$. Suppose that $\varphi \in C^1(\R)$ such that $\varphi(0)=0$ and $f\in C((a,b))$. Then there hold
	\begin{align}\label{chain_rule_weak}
	{^{\pm}}{\mathcal{D}}{^{\alpha}}  \varphi(f)(x)  =   \frac{\varphi(f)(x)}{f(x)} {^{\pm}}{\mathcal{D}}{^{\alpha}} f(x)
	+ {^{\pm}}{R}{_{0}^{\alpha}}\Bigl(f, \frac{\varphi(f)}{f} \Bigr)(x) \qquad \mbox{a.e. in } (a,b),
	\end{align}
	where ${^{\pm}}{R}{_{0}^{\alpha}}(f, g)$  are defined by \eqref{LeftProductRuleRemainder_0} and \eqref{RightProductRuleRemainder_0}. 
	
\end{theorem}

\subsection{Fundamental Theorem of Weak Fractional Calculus}\label{sec-4.6}

In this section, we aim to extend the FTcFC (cf. Theorem \ref{FTFC}) to weakly differentiable functions. Similar to the FTcFC for the classical fractional (Riemann-Liouville)  derivatives, the finite and infinite domain cases are significantly different, hence must be treated separately. 

\subsubsection{\bf The Finite Interval Case}\label{sec-4.6.1}
To establish the Fundamental Theorem of Weak Fractional Calculus (FTwFC) on a finite domain, 
we first need to show the following crucial lift lemma.
   
   \begin{lemma}\label{FTFCConstant}
    Let $\Omega\subset \R$ and $0<\alpha <1$. Suppose that  $u \in L^{p}(\Omega)$ and  ${^{\pm}}{\mathcal{D}}{^{\alpha}}u \in L^{p}(\Omega)$ for some $1 \leq p < \infty$. 
    Then ${^{\pm}}{I}{^{1-\alpha}}u \in  W^{1,1}(\Omega)$. 
    \end{lemma}
   
   \begin{proof}
       Choose $\{u_j\}_{j=1}^{\infty} \subset C^{\infty}(\Omega)$ so that $u_j \rightarrow u$ in $L^{p}(\Omega)$ and ${^{\pm}}{\mathcal{D}}{^{\alpha}}u_j \rightarrow {^{\pm}}{\mathcal{D}}{^{\alpha}}u$ in $L^{p}(\Omega)$. Since ${^{\pm}}{D}{^{\alpha}}u_j \in L^{1}(\Omega)$, 
       then ${^{\pm}}{I}{^{1-\alpha}} u_j \in W^{1,1}(\Omega) $.  
       By the stability property of ${^{\pm}}{I}{^{1-\alpha}}$ we have
       \begin{align*}
           \|{^{\pm}}{I}{^{1-\alpha}} u_m - {^{\pm}}{I}{^{1-\alpha}} u_n\|_{ W^{1,1}(\Omega)}         
           &=\|{^{\pm}}{I}{^{1-\alpha}} u_m - {^{\pm}}{I}{^{1-\alpha}} u_n\|_{L^{1}(\Omega)} \\
           &\qquad  + \|{^{\pm}}{D}{^{\alpha}} u_m - {^{\pm}}{D}{^{\alpha}} u_n \|_{L^{1}(\Omega) } \nonumber \\ 
           &\leq C \|u_m - u_n\|_{L^{1} (\Omega)} + \| {^{\pm}}{\mathcal{D}}{^{\alpha}}u_m - {^{\pm}}{\mathcal{D}}{^{\alpha}} u_n \|_{L^{1}(\Omega)}\\
           &\to 0\quad\mbox{as } n,m\to \infty.
       \end{align*}
       Hence,  $\{{^{\pm}}{I}{^{1-\alpha}}u_j\}_{j=1}^{\infty}$ is a Cauchy sequence in $W^{1,1}(\Omega)$. Since $W^{1,1}(\Omega)$ is a Banach space, there exists $v \in W^{1,1}(\Omega)$ so that ${^{\pm}}{I}{^{1-\alpha}}u_j \rightarrow v$ in $W^{1,1}(\Omega)$ . 
       
       It remains to show that $v = {^{\pm}}{I}{^{1-\alpha}}u$. On noting that  
       \begin{align*}
        \|v - {^{\pm}}{I}{^{1-\alpha}}u\|_{L^1(\Omega)}
        &  \leq \|v - {^{\pm}}{I}{^{1-\alpha}}u_j \|_{L^{1}(\Omega)} + \|{^{\pm}}{I}{^{1-\alpha}} u_j - {^{\pm}}{I}{^{1-\alpha}} u \|_{L^{1}(\Omega)}  \\
        &\leq \|v - {^{\pm}}{I}{^{1-\alpha}}u_j \|_{L^{1}(\Omega)} + C \| u_j -  u \|_{L^{1}(\Omega)} 
        \to 0 \quad \mbox{as } j\to \infty.
       \end{align*}
       Hence,  $v = {^{\pm}}{I}{^{1-\alpha}} u$ almost everywhere in $\Omega$. The proof is complete.
   \end{proof}

	\begin{remark}
	Since $AC(\overline{\Omega})$ is isomorphic to $W^{1,1}(\Omega)$ in the 1D case, 
	the above lemma also implies that ${^{\pm}}{I}{^{1-\alpha}}u \in AC(\overline{\Omega})$. The above lemma shows that 
	if $u \in {^{\pm}}{W}{^{\alpha,1}}(\Omega)$ (see the space definition in Section \ref{sec-5}), 
	then the operator ${^{\pm}}{I}{^{1-\alpha}}$ lifts $u$ from 
	${^{\pm}}{W}{^{\alpha,1}}(\Ome)$ into ${W}{^{1,1}}(\Ome)$. This result reinforces the characterization of weakly fractional differentiable functions as stated in Section \ref{sec-4.2}. In particular, 
	one can roughly think about weakly fractional differentiable functions as those whose classical fractional derivatives exist almost everywhere. This is (almost) exactly the same characterization 
	for first order weakly differentiable functions (in 1D). Precisely, absolute continuity characterizes weakly differentiable functions and the absolute continuity of 
	${^{\pm}}{I}{^{1-\alpha}}u$ characterizes weakly fractional differentiable 
	functions. 
	\end{remark}

   \begin{theorem}\label{FTWFC}
       Let $\Omega\subset \R$ and $0 < \alpha <1$. Suppose that $u \in L^{p}(\Omega)$ and  ${^{\pm}}{\mathcal{D}}{^{\alpha}}u \in L^{p}(\Omega)$ for some $1\leq p < \infty$. Then
       there holds
       \begin{align}\label{WeakFTFC}
           u = c^{1-\alpha}_{\pm} \kappa^{\alpha}_{\pm}  + {^{\pm}}{I}{^{\alpha}}{^{\pm}}{\mathcal{D}}{^{\alpha}} u \qquad
           \mbox{a.e. in } \Omega.
       \end{align}
  
   \end{theorem}

   \begin{proof}
      Let $\{u_j \}_{j=1}^{\infty} \subset C^{\infty}(\Omega)$ so that $u_j \rightarrow u$ 
      in $L^{p}(\Omega)$ and ${^{\pm}}{\mathcal{D}}{^{\alpha}}u_j \rightarrow {^{\pm}}{\mathcal{D}}{^{\alpha}}u$ in $L^{p}(\Omega)$; in particular, 
      $u_j$ and its derivative converge in $L^{1}(\Omega)$. By Lemma \ref{FTFCConstant}, ${^{\pm}}{I}{^{1-\alpha}}u_j \rightarrow {^{\pm}}{I}{^{1-\alpha}} u$ in $W^{1,1}(\Ome)\cong AC(\overline{\Omega})$.
      Moreover, by the FTCFC we get
      \[
      u_j (x) = c^{1-\alpha}_{j,\pm} \kappa^{\alpha}_{\pm}(x)  + {^{\pm}}{I}{^{\alpha}}{^{\pm}}{D}{^{\alpha}} u_j (x).
      \]
Thus,
      \begin{align*}
     & \|u-c^{1-\alpha}_{\pm} \kappa^{\alpha}_{\pm} - {^{\pm}}{I}{^{\alpha}}  {^{\pm}}{D}{^{\alpha}} u\|_{L^1(\Omega)} \\
         &\quad =\| (u - u_j) - (c^{1-\alpha}_{\pm} - c^{1-\alpha}_{j,\pm}) \kappa^{\alpha}_{\pm} - {^{\pm}}{I}{^{\alpha}} ({^{\pm}}{\mathcal{D}}{^{\alpha}}u - {^{\pm}}{D}{^{\alpha}} u_j)\|_{L^{1}(\Omega)}\\
         &\quad\leq  \|u - u_j\|_{L^{1}(\Omega)} + |c^{1-\alpha}_{\pm} - c^{1-\alpha}_{j,\pm}| \,\|\kappa^{\alpha}_{\pm}\|_{L^{1}(\Omega)} + \|{^{\pm}}{I}{^{\alpha}} ({^{\pm}}{\mathcal{D}}{^{\alpha}} u - {^{\pm}}{D}{^{\alpha}} u_j ) \|_{L^{1}(\Omega)}\\
         &\quad\leq  \|u - u_j\|_{L^{1}(\Omega)} + |c^{1-\alpha}_{\pm} - c^{1-\alpha}_{j,\pm}| \,\|\kappa^{\alpha}_{\pm}\|_{L^{1}(\Omega)} + C\|{^{\pm}}{\mathcal{D}}{^{\alpha}} u - {^{\pm}}{D}{^{\alpha}}u_j\|_{L^{1}(\Omega)} \\
         &\quad \to 0 \qquad\mbox{as } j\to \infty.
      \end{align*}
     Therefore, 
     \[
     u -c^{1-\alpha}_{\pm} \kappa^{\alpha}_{\pm} - {^{\pm}}{I}{^{\alpha}}  {^{\pm}}{D}{^{\alpha}} u =0 \qquad \mbox{a.e. in }  \Omega.
     \]
  The proof is complete. 

   \end{proof}

\begin{remark}
	(a) We refer  toTheorem \ref{FTWFC} as the Fundamental Theorem of Weak Fractional Calculus (FTwFC) in this paper.  
	
	(b) The above FTwFC is an essential tool for studying weakly fractional differentiable functions, 
	in particular, it will play a crucial role in proving compact and Sobolev embeddings and a fractional Poincar\'e inequality in Sections \ref{sec-5.6} and \ref{sec-5.7}. 
\end{remark}

    To conclude this subsection,  we would like to circle back to an unproven 
    inclusion result for weak fractional derivatives which was alluded to in 
    Proposition \ref{properties} part $(ii)$. This then presents the first application of the FTwFC. 
   
   \begin{proposition}
       Let $\Omega \subset \R$ and $0 < \alpha < \beta <1$. Suppose that ${^{\pm}}{\mathcal{D}}{^{\beta}}u$ exists in $L^{1}(\Omega)$. Then ${^{\pm}}{\mathcal{D}}{^{\alpha}}u$ exists in $L^{1}(\Omega)$.
   \end{proposition}
   
   \begin{proof}
        It follows by Theorem \ref{FTWFC} that 
        \begin{align*}
            u = c^{1-\beta}_{\pm} \kappa^{\beta}_{\pm} + {^{\pm}}{I}{^{\beta}} {^{\pm}}{\mathcal{D}}{^{\beta}} u \qquad\mbox{a.e. in } \Omega.
        \end{align*}
         Then there holds   
        \begin{align*}
            \int_{\Omega} u\, {^{\mp}}{D}{^{\alpha}} \varphi\,dx 
            &= \int_{\Omega} \bigl(c^{1-\beta}_{\pm} \kappa^{\beta}_{\pm} + {^{\pm}}{I}{^{\beta}} {^{\pm}}{\mathcal{D}}{^{\beta}} u \bigr)\,{^{\mp}}{D}{^{\alpha}} \varphi\,dx \\
            &= \int_{\Omega} {^{\pm}}{D}{^{\alpha}} \bigl(c^{1-\beta}_{\pm} \kappa^{\beta}_{\pm} + {^{\pm}}{I}{^{\beta}} {^{\pm}}{\mathcal{D}}{^{\beta}} u \bigr)\,\varphi\,dx \\ 
            &= \int_{\Omega} \bigl(c^{1-\beta}_{\pm} \kappa^{\beta -\alpha}_{\pm} + {^{\pm}}{I}{^{\beta -\alpha}} {^{\pm}}{\mathcal{D}}{^{\beta}} u \bigr)\, \varphi\,dx. 
        \end{align*}
        Since a direct calculation shows that $v:=c^{1-\beta}_{\pm} \kappa^{\beta -\alpha}_{\pm} + {^{\pm}}{I}{^{\beta -\alpha}} {^{\pm}}{\mathcal{D}}{^{\beta}} u\in L^1(\Omega)$,
        then the above identity implies that ${^{\pm}}{\mathcal{D}}{^{\alpha}}u$ exists and 
        ${^{\pm}}{\mathcal{D}}{^{\alpha}}u=v$ almost everywhere in $\Omega$. The proof is complete.
   \end{proof}
   
\subsubsection{\bf The Infinite Interval Case}\label{sec-4.6.2}
Unlike the finite domain, the absence of any boundary in the infinite interval case 
$\Ome=\R$ allows for a cleaner statement of the FTwFC and a simpler proof.

\begin{theorem}\label{FTWFCR}
    Let $0 < \alpha < 1$. Suppose that $u, v \in L^{1}(\R)$. If
    \begin{align}\label{WeakFTFCR}
        u = {^{\pm}}{I}{^{\alpha}} v \quad \mbox{a.e. in } \R,
    \end{align}
    then ${^{\pm}}{\mathcal{D}}{^{\alpha}}u = v$ almost everywhere. The converse is true under the additional assumption $u(x) \rightarrow 0$ almost everywhere as $|x| \to \infty$.
\end{theorem}

\begin{proof}
    The assertion and the accompanying equation \eqref{WeakFTFCR} follow from
    an application of the characterization theorem (cf. Theorem \ref{characterization})
    for weak fractional derivatives and Theorem \ref{FTFCa}.
\end{proof}

As was illustrated in the finite domain case and the infinite domain case for classical fractional derivatives, we can use the relation \eqref{WeakFTFCR}
 to show a basic inclusion result for weak fractional derivatives.

\begin{proposition}
    Let $0 < \alpha < \beta < 1$. Suppose that $u, {^{\pm}}{\mathcal{D}}{^{\beta}} u \in L^{1}(\R)$. Then ${^{\pm}}{\mathcal{D}}{^{\alpha}}u$ exists in $L^{1}(\R)$. 
\end{proposition}

\begin{proof}
    Apply the characterization theorem for weakly fractional differentiable functions on $\R$ (cf. Theorem \ref{characterization}) and the FTwFC (cf. Theorem \ref{FTWFCR}), then pass limits.
\end{proof}

\section{Fractional Sobolev Spaces}\label{sec-5}
After having established a notion of fractional order weak derivatives analogous to that of integer order 
 weak derivatives, 
it is natural to define Sobolev spaces in the same manner as in the integer order case. To the best 
of our knowledge, such an approach is missing in the existing theories of fractional Sobolev spaces.
The goal of this section is to fill such a void in the 1D case and to present a new fractional 
Sobolev theory based on such an approach.

\subsection{Existing Definitions and Essential Properties}\label{sec-5.1}
 Several definitions of fractional order Sobolev spaces have been given in the literature. Below we will only 
 quote two such definitions which we will be concerned with in this section. 

\begin{definition}
    Let $\Omega \subseteq \R$,   $s>0$, and   $1 \leq p \leq  \infty$. Set $m:=[s]$ and $\sigma:=s-m$. 
    Define the fractional Sobolev space ${\tW}^{s,p}(\Omega)$ by
    \[
    {\tW}^{s,p}(\Omega) : = \left\{ u \in W^{m,p}(\Omega) ; \dfrac{\left|\mathcal{D}^m u(x) - \mathcal{D}^m u(y) \right|}{|x- y|^{\frac{1}{p} + \sigma}} \in L^{p}(\Omega \times \Omega)\right\}, 
    \]
    which is endowed with the norm 
    \begin{align*}
    \|u\|_{{\tW}^{s,p}(\Omega)} : =\begin{cases} \displaystyle{ 
    \Bigl( \|u\|_{W^{m,p}(\Omega)}^{p} + [\mathcal{D}^m u]_{{\tW}^{\sigma,p}(\Omega) }  \Bigr)^{\frac{1}{p}} } &\qquad\mbox{if } 1\leq p< \infty, \\
    \displaystyle{  \|u\|_{W^{m,\infty}(\Omega)}  + [\mathcal{D}^m u]_{{\tW}^{\sigma,\infty}(\Omega) } }    &\qquad\mbox{if } p= \infty,
     \end{cases}
    \end{align*}
    where
    \begin{equation*}
    [u]_{{\tW}^{\sigma,p}(\Omega) }: =  \begin{cases} \displaystyle{
    \Bigl( \int_{\Omega} \int_{\Omega} \dfrac{|u(x) - u(y)|^{p}}{|x-y|^{1 + \sigma p}} \Bigr)^{\frac{1}{p}}\, dx dy }
    &\qquad\mbox{if } 1\leq p< \infty,\\
    \displaystyle{ \sup_{(x,y)\in \Omega\times \Omega} \dfrac{|u(x) - u(y)|}{|x-y|^{\sigma } } }
    &\qquad\mbox{if } p=\infty.  
    \end{cases} 
    \end{equation*}
    When $p=2$, we set ${\tH}^{s}(\Omega) := {\tW}^{s,2}(\Omega)$. 
\end{definition}

We may elect to use the convention $\hat{u}: = \mathcal{F}[u]$ to denote the Fourier transform of a given function. In many instances, we will use the ``hat" convention to inform the reader that a given object is related in some way to the Fourier transform.

\begin{definition}
    Let   $s>0$ and  $1 \leq p \leq  \infty$. 
     Define the fractional Sobolev space  $\widehat{W}^{s,p}(\R)$ by
    \[
    \widehat{W}^{s,p}(\R) : = \left\{ u \in L^{p}(\R) :  [u]_{\widehat{W}^{s,p}(\R)} < \infty \right\},  \qquad 1\leq p\leq  \infty,
    \]
   where 
   \[
   [u]_{\widehat{W}^{s,p}(\R)} : = \int_{\R} (1+ |\xi|^{s p}) |\hat{u} (\xi)|^{p}\,d\xi ,
   \qquad 1\leq p\leq \infty.
   \]
   When $p=2$, we set $\widehat{H}^{s}(\R) := \widehat{W}^{s,2}(\R)$.
\end{definition}

\begin{remark}
(a) It is well known (\cite{Adams}, \cite{Nezza}) that  ${\tW}^{s,p}(\Omega)$ and $\widehat{W}^{s,p}(\R)$ are Banach spaces,
	and ${\tH}^{s}(\Omega)$ and  $\widehat{H}^{s}(\R)$ are Hilbert spaces.  
	
(b) It is also well known (\cite{Adams},\cite{Nezza}) that ${{\tH}^{s}(\R)}$ and $\widehat{H}^{s}(\R)$ are equivalent spaces. In particular,
\begin{align}\label{SeminormRelation}
    [u]_{{\tH}^{s}(\R)} \cong \int_{\R} |\xi|^{2s} |\hat{ u}(\xi)|^{2}\,d\xi.
\end{align}
However, ${{\tW}^{s,p}(\R)}$ and $\widehat{W}^{s,p}(\R)$ are not equivalent spaces for $p \neq 2$. 
    
\end{remark}

\subsection{New Definitions of Fractional Sobolev Spaces}\label{sec-5.2}
Although the definitions above have some kind of differentiability built in, neither of them are analogous 
to the definitions used in the integer order case which are constructed using weak derivatives. It is the main point 
of our study to introduce fractional order Sobolev spaces that are constructed using   weak fractional  derivatives. 

    \begin{definition}
       For $\alpha>0$, let $m :=[\alpha]$. For $1 \leq p \leq \infty$,  the left/right fractional Sobolev spaces ${^{\pm}}{W}{^{ \alpha , p}} (\Omega)$ are defined by  
        \begin{align} \label{FSS}
             {^{\pm}}{W}{^{ \alpha , p}} (\Omega) = \left\{ u \in W^{m,p}(\Omega): 
             {^{\pm}}{\mathcal{D}}{^{\alpha}}   u \in L^{p}(\Omega) \right\},
        \end{align}
        which are endowed respectively with the norms 
        \begin{align} \label{FSS_norm}
          \|u\|_{{^{\pm}}{W}{^{\alpha , p}}(\Omega)}:= \begin{cases}
          \Bigl(\left\|u\right\|_{W^{m,p}(\Omega)}^{p} + \left\|{^{\pm}}{\mathcal{D}}{^{\alpha}}  u \right\|_{L^{p}(\Omega)}^{p} \Bigr)^{\frac{1}{p}} &\qquad \text{if } 1 \leq p < \infty,\\
          \|u\|_{W^{m,\infty}(\Omega)} 
          + \left\|{^{\pm}}{\mathcal{D}}{^{\alpha}}  u \right\|_{L^{\infty}(\Omega)} &\qquad \text{if } p = \infty.
          \end{cases} 
        \end{align}

    \end{definition}

    \begin{definition}
        For $  \alpha >0$ and $1 \leq p \leq \infty$, the symmetric fractional Sobolev space is defined by 
        \begin{align}\label{SFSS}
           { {W}^{\alpha,p}(\Omega)}:= {^{-}}{W}{^{\alpha,p}}(\Omega) \cap {^{+}}{W}{^{\alpha,p}}(\Omega),
        \end{align}
        which is endowed with the norm
        \begin{align}\label{SFSS_norm}
            \|u\|_{{ {W}^{\alpha,p}(\Omega)} } &:= \begin{cases} \displaystyle{
            \Bigl( \|u\|_{{^{-}}{W}{^{\alpha,p}}(\Omega)}^{p} + \|u\|_{{^{+}}{W}{^{\alpha,p}}(\Omega)}^{p} \Bigr)^{\frac{1}{p}} }  &\qquad \text{if } 1\leq p < \infty,\\
            \displaystyle{ \|u\|_{{^{-}}{W}{^{\alpha,\infty}}(\Omega)} + \|u \|_{{^{+}}{W}{^{\alpha,\infty}}(\Omega)} } &\qquad \text{if } p = \infty.
             \end{cases}
        \end{align}
    \end{definition}

\begin{remark}
	For $\alpha>0$, let $m :=[\alpha]$ and $\sigma:=\alpha- m$. Using the semigroup property of  weak fractional derivatives, it is easy to see that 
	\begin{align} \label{FSS_equiv}
	{^{\pm}}{W}{^{ \alpha , p}} (\Omega) = \left\{ u \in W^{m,p}(\Omega): 
	\mathcal{D}^m ({^{\pm}}{\mathcal{D}}{^{\sigma}} u) \in L^{p}(\Omega) \right\}
	\end{align}
	and
	\begin{align} \label{FSS_norm_equiv}
	\|u\|_{{^{\pm}}{W}{^{\alpha , p}}(\Omega)}:= \begin{cases}
	\Bigl(\left\|u\right\|_{W^{m,p}(\Omega)}^{p} + \left\|\mathcal{D}^m ({^{\pm}}{\mathcal{D}}{^{\sigma}}   u) \right\|_{L^{p}(\Omega)}^{p} \Bigr)^{\frac{1}{p}} &\quad \text{if } 1 \leq p < \infty,\\
	\|u\|_{W^{m,\infty}(\Omega)} 
	+ \left\|\mathcal{D}^m ({^{\pm}}{\mathcal{D}}{^{\sigma}}  u) \right\|_{L^{\infty}(\Omega)} &\quad \text{if } p = \infty.
	\end{cases}
	\end{align}
\end{remark}

\subsection{Elementary Properties of Fractional Sobolev Spaces}\label{sec-5.3}
  We first verify that $\left\| \cdot \right\|_{{^{\pm}}{W}{^{\alpha , p}}(\Omega)}$  are indeed norms 
  and the spaces ${^{\pm}}{W}{^{\alpha , p}}(\Omega)$ are Banach spaces. 
    
     \begin{proposition}
        Let $ \alpha >0$, $1 \leq p \leq \infty$, and $\Omega \subseteq \R$. Then $\left\| \cdot \right\|_{{^{\pm}}{W}{^{\alpha , p}}(\Omega)}$ are  norms  on ${^{\pm}}{W}{^{\alpha , p}}(\Omega)$, 
        which are in turn Banach spaces with these norms. Consequently, $W{^{\alpha , p}(\Omega)}$ is also a Banach space. 
    \end{proposition}

    \begin{proof}
    	Since the verification that $\left\| \cdot \right\|_{{^{\pm}}{W}{^{\alpha , p}}(\Omega)}$ are norms is straightforward, we omit its proof to save space. 

      To show that ${^{\pm}}{W}{^{\alpha,p}}$ are  Banach spaces, it suffices to show that they are  complete.
       We shall only give a proof for ${^{-}}{W}{^{\alpha,p}}$ with $0<\alpha<1$ and $1 \leq p < \infty$ 
       since the cases for ${^{+}}{W}{^{\alpha,p}}$, $\alpha >1$, and $p = \infty$ follow similarly. 
         	
         	Let $\eps >0$ and $\left\{u_n\right\}_{n=1}^{\infty}$ be a Cauchy sequence in  ${^{-}}{W}{^{\alpha,p}}(\Omega)$. By definition, there exists a positive  integer $N \in \N$ so that for 
         	every $n,m \geq N$, 
         	\begin{align*}
         	\left\|u_n - u_m \right\|_{{^{-}}{W}{^{\alpha , p}}(\Omega)} < \eps.
         	\end{align*}
         	This implies that 
         	\begin{align*}
         	\left\|u_n - u_m \right\|_{L^{p}(\Omega)}^{p} + \left\| {^{-}}{\mathcal{D}}{^{\alpha}}  u_n - {^{-}}{\mathcal{D}}{^{\alpha}} u_m \right\|_{L^{p}(\Omega)}^{p} < \eps^p.
         	\end{align*}
         	Thus $\left\|u_n - u_m \right\|_{L^{p}(\Omega)}^{p} < \eps$. Therefore, $\left\{u_n \right\}_{n=1}^{\infty}$ is a Cauchy sequence in $L^{p}(\Omega)$. Since $L^{p}$ is complete, there exists $u \in L^{p}(\Omega)$ so that $u_n \rightarrow u$ in $L^{p}(\Omega)$ as $n \rightarrow \infty$. Similarly, $\left\|{^{-}}{\mathcal{D}}{^{\alpha}} u_n - {^{-}}{\mathcal{D}}{^{\alpha}} u_m \right\|_{L^{p}(\Omega)}^{p} < \eps$ implies that $\left\{{^{-}}{\mathcal{D}}{^{\alpha}} u_n \right\}_{n=1}^{\infty}$ is a Cauchy sequence in $L^{p}(\Omega)$. Then there exists $v \in L^{p}(\Omega)$ such that ${^{-}}{\mathcal{D}}{^{\alpha}} u_n \rightarrow v$ in $L^{p}(\Omega)$ as $n \rightarrow \infty$.

         	It remains to show that ${^{-}}{\mathcal{D}}{^{\alpha}}u = v$. To the end, by the definition of weak fractional  derivatives we have 
         	\begin{align*}
         	\int_{\Omega} u \, {^{+}}{D}{^{\alpha}} \varphi \,dx
         	= \lim_{n \rightarrow \infty} \int_{\Omega} u_n  {^{+}}{D}{^{\alpha}}  \varphi \,dx
         	= \lim_{n\rightarrow \infty} \int_{\Omega} {^{-}}{\mathcal{D}}{^{\alpha}} u_n  \varphi \,dx	
         	= \int_{\Omega} v  \varphi\,dx. 
         	\end{align*}
         	By the uniqueness of weak fractional derivatives, ${^{-}}{\mathcal{D}}{^{\alpha}} u =v$ almost everywhere. This concludes that $u \in {^{-}}{W}{^{\alpha,p}}(\Omega)$ and hence ${^{-}}{W}{^{\alpha,p}}(\Omega)$ is complete. 
         \end{proof}
         
         \begin{remark}
         The above proof  coupled with the inner products 
        \[ 
        \langle u,v \rangle_\pm : =(u,v) + \bigl( {^{\pm}}{\mathcal{D}}{^{\alpha}} u , {^{\pm}}{\mathcal{D}}{^{\alpha}}v \bigr)
         = \int_{\Omega} uv\,dx + \int_{\Omega} {^{\pm}}{\mathcal{D}}{^{\alpha}} u {^{\pm}}{\mathcal{D}}{^{\alpha}}v\,dx  
         \]
         is sufficient to verify that ${^{\pm}}{W}{^{\alpha ,2}}(\Omega)$ is a Hilbert space and consequently, $ {W}^{\alpha,2}(\Omega)$ is also a Hilbert space. In this case, we adopt the standard notations ${^{\pm}}{H}{^{\alpha}}(\Omega) : = {^{\pm}}{W}{^{\alpha,2}}(\Omega)$   
         and $ {H}^{\alpha}(\Omega): =  {W}^{\alpha,2}(\Omega)$.
    \end{remark}

    \begin{proposition} ${^{\pm}}{W}{^{\alpha,p}}(\Omega)$ is reflexive for $1 < p < \infty$ and separable for $1 \leq p < \infty$. Consequently, the same assertions hold for $ {W}^{\alpha,p}(\Omega)$.
    \end{proposition}
    
    \begin{proof}
        Let $E:= L^{p}(\Omega) \times L^{p}(\Omega)$ and define the norm 
        \[
        \|(u,v)\|_{E} :=  \left( \|u\|_{L^{p}(\Omega)}^{p} + \|v\|_{L^{p}(\Omega)}^{p}\right)^{\frac{1}{p}}
        \]
        on $E$. Clearly, $E$ is reflexive for $1 < p < \infty$. Define the mapping $\Phi: {^{\pm}}{W}{^{\alpha,p}}(\Omega) \rightarrow E$ be $\Phi u = ( u , {^{\pm}}{\mathcal{D}}{^{\alpha}} u )_{E}$. It follows that for $u,v \in {^{\pm}}{W}{^{\alpha,p}}(\Omega)$, $\|\Phi u - \Phi v\|_{L^{p}(\Omega)} = \|u - v\|_{{^{\pm}}{W}{^{\alpha,p}}(\Omega)}$ implies that $\Phi$ is an isometry from ${^{\pm}}{W}{^{\alpha,p}}(\Omega)$ to $E$. Since ${^{\pm}}{W}{^{\alpha,p}}(\Omega)$ are Banach spaces, $\Phi({^{\pm}}{W}{^{\alpha,p}}(\Omega))$ are respectively closed subspaces of $E$. 
        Therefore, $\Phi ({^{\pm}}{W}{^{\alpha,p}}(\Omega))$ is reflexive. Hence, ${^{\pm}}{W}{^{\alpha,p}}(\Omega)$ is reflexive. 
        
        Note that $E$ is separable for $1 \leq p < \infty$. Then since $\Phi ({^{\pm}}{W}{^{\alpha,p}}(\Omega)) \subset E$, we have that $\Phi ({^{\pm}}{W}{^{\alpha,p}}(\Omega))$ is separable. Consequently, ${^{\pm}}{W}{^{\alpha,p}}(\Omega)$ is separable.  The proof is complete. 
    \end{proof}

    \subsection{Approximation and Characterization of Fractional Sobolev Spaces}\label{sec-5.4}
    
    In the integer order case, a very popular alternative way to define Sobolev spaces is to use the completion 
    spaces of smooth functions under
    chosen Sobolev norms. The goal of this subsection is to establish an analogous result for fractional Sobolev spaces
    introduced in Section \ref{sec-5.2}. To the end,  we first need to introduce spaces that we refer to as   \textit{one-side supported spaces}.
    
    \begin{definition}
    	For $(a,b) \subseteq \R$, we set 
    	\begin{align*}
    	{^{-}}{C}{^{\infty}_{0}}((a,b)) &: = \{ \varphi \in C^{\infty}((a,b)) \, | \,\exists\, c \in (a,b) \mbox{ such that } \varphi(x) \equiv 0\,\, \forall x > c\},\\
    	{^{+}}{C}{^{\infty}_{0}}((a,b)) &: = \{ \varphi \in C^{\infty}((a,b)) \, | \,\exists\, c \in (a,b)  \mbox{ such that }  \varphi(x) \equiv 0  \,\,\forall x < c\}.
    	\end{align*}
    \end{definition}
    
    \begin{remark}
        Here we use the notation ${^{-}}{C}{^{\infty}_{0}}((a,b))$ to represent functions whose support is a not actually a compact subset of $(a,b)$. In particular, if $u \in {^{-}}{C}{^{\infty}_{0}}((a,b))$, then $\mbox{supp}(u) = [a,c]$, which of course is not a compact subset of $(a,b)$. However, we believe that it is most nature for the reader by using traditional support notation $C^{\infty}_{0}$ coupled with the directionality ${^{-}}{C}{^{\infty}_{0}}$ in order to remind the reader throughout that these functions have \textit{direction-dependent} zero value. The use of ${^{-}}{C}{^{\infty}_{0}}$ and ${^{+}}{C}{^{\infty}_{0}}$ (particularly the direction indication) are chosen so that these spaces will pair with the appropriate direction-dependent spaces ${^{-}}{W}{^{\alpha,p}}$ and ${^{+}}{W}{^{\alpha,p}}$ respectively. The need for these and the aforementioned space coupling will become evident in Section \ref{sec-5.6}.
    \end{remark}

    Now we are ready to introduce completion spaces using the norms defined in Section \ref{sec-5.2}.

    \begin{definition}\label{completionspaces}
    	Let $ \alpha >0$ and $1 \leq p \leq \infty$.  We define 
    	\begin{itemize}
    		\item[(i)] ${^{\pm}}{\overline{W}}{^{\alpha,p}}(\Omega)$ to be the completion of $C^{\infty}(\Omega)$ in the norm $\|\cdot\|_{{^{\pm}}{W}{^{\alpha,p}}(\Omega)}$,
    		\item[(ii)] ${^{\pm}}{\overline{W}}{^{\alpha,p}_{0}}(\Omega)$ to be  the completion of ${^{\pm}}{C}{^{\infty}_{0}}(\Omega)$ in the norm $\|\cdot \|_{{^{\pm}}{W}{^{\alpha,p}}(\Omega)}$,
    		\item[(iii)] $ {\overline{ {W}}^{\alpha,p}(\Omega)}$ to be the completion of $C^{\infty}(\Omega)$ in the  norm $\|\cdot\|_{{ {W}^{\alpha,p}(\Omega)}}$,
    		\item[(iv)] $\overline{{W}}^{\alpha,p}_{0}(\Omega)$ to be  the completion of $C^{\infty}_{0}(\Omega)$ in the  norm $\|\cdot\|_{ {W}^{\alpha,p}(\Omega)}$.
    	\end{itemize}
    \end{definition} 
    
    \smallskip
    It is a goal of this section to prove that each of the above defined \textit{completion spaces} is equivalent to another fractional Sobolev space; some have yet to be defined (i.e. zero trace spaces see Definition \ref{completionspaces}).
    
   . 
       
        \subsubsection{\bf The Finite Domain Case: $\Omega=(a,b)$}\label{sec-5.4.1}
        The goal of this subsection is to establish the equivalence  
         ${^{\pm}}{\overline{W}}{^{\alpha,p}}(\Omega) = {^{\pm}}{W}{^{\alpha, p}} (\Omega)$. 
         This is analogous to Meyers and Serrin's celebrated  ``$H = W$" result (cf. \cite{Adams, Evans, Meyers}). 
         It turns out that the proof is more complicated than that of the integer order case due to more complicated 
         product rule for fractional derivatives.

    \begin{lemma}
        Let $ \alpha >0$ and $1 \leq p <\infty$. Suppose  $\psi \in C^{\infty}_{0}(\Omega)$ and $u \in {^{\pm}}{W}{^{\alpha,p}}(\Omega)$. Then $u \psi \in {^{\pm}}{W}{^{\alpha,p}}(\Omega)$.
    \end{lemma}

    \begin{proof}
    	We only give a proof for $0<\alpha<1$ because the case $\alpha>1$ follows immediately by setting
    	$m:=[\alpha]$ and $\sigma:=\alpha-m$ and using the Meyers and Serrin's celebrated result.
    	
        Since $\psi \in C^{\infty}_{0}(\Omega)$, there exists a compact set $K:=\supp(\psi) \subset \Omega$ such that  $\psi \in C^{\infty}(K)$. Then there exists $0 \leq M <\infty$ so that $M_0 = \max_{\Omega}|\psi|$ and 
        $\|\psi\|_{L^{\infty}(\Omega)} = M_{0} < \infty.$      
        Since $u \in L^{p}(\Omega)$, then trivially we have $u\psi \in L^{p}(\Omega)$.

        It remains to show that ${^{\pm}}{\mathcal{D}}{^{\alpha}} (u\psi) \in L^{p}(\Omega)$. To that end, by Theorem \ref{WeakProductRule} for arbitrarily large $m \in\N$, we get
        \begin{align*}
            &\left\|{^{\pm}}{\mathcal{D}}{^{\alpha}} (u \psi) \right\|_{L^{p}(\Omega)} \leq \left\|{^{\pm}}{\mathcal{D}}{^{\alpha}} u \cdot \psi \right\|_{L^{p}(\Omega)} 
            + \left\|\sum_{k=1}^{m} C_{k} {^{\pm}}{I}{^{k-\alpha}} u D^{k} \psi  + {^{\pm}}{R}{^{\alpha}_{m}}(u,\psi) \right\|_{L^{p}(\Omega)} \\
            &\,\, \leq M_0 \left\| {^{\pm}}{\mathcal{D}}{^{\alpha}} u \right\|_{L^{p}(\Omega)} + M_1 \sum_{k=1}^{m}\left| C_{k}\right| \left\|{^{\pm}}{I}{^{k - \alpha}} u\right\|_{L^{p}(\Omega)}  + \left\| {^{\pm}}{R}{^{\alpha}_{m}} (u,\psi) \right\|_{L^{p}(\Omega)}\\
            &\,\, \leq M_0 \left\| {^{\pm}}{\mathcal{D}}{^{\alpha}} u \right\|_{L^{p}(\Omega)} + M_1 \sum_{k=1}^{m} \dfrac{\left| C_{k}\right| \cdot |\Omega|^{k-\alpha} }{ ( k- \alpha) \Gamma(k - \alpha)} \left\|u\right\|_{L^{p}(\Omega)}  + \left\| {^{\pm}}{R}{^{\alpha}_{m} }(u,\psi) \right\|_{L^{p}(\Omega)},
        \end{align*}
        where $M_1 : = \sup \left| D^{k} \psi(x) \right|$ taken over $1 \leq k \leq m$ and $x \in \Omega$. Clearly, $M_1 < \infty$ since $\psi \in C^{\infty}_{0} (\Omega)$. Because  $u, {^{\pm}}{\mathcal{D}}{^{\alpha}} u\in L^{p}(\Omega)$  and 
        \begin{align*}
            \dfrac{\left| C_{k}\right| \cdot |\Omega|^{k-\alpha} }{ ( k- \alpha) \Gamma(k - \alpha)} &= \dfrac{\Gamma(1 + \alpha) |\Omega|^{k-\alpha}  }{(k - \alpha) \Gamma(k+1) |\Gamma( 1- k + \alpha)|} <\infty,
        \end{align*}
        the first two terms on the right-hand side of the above inequality are finite. 

        It is left to show that the remainder term is also finite in $L^{p}(\Omega)$. To be precise, we consider the case for ${^{-}}{R}{^{\alpha}_{m}}(u,\psi)$. By its definition we get  
        \begin{align*}
            \left|{^{-}}{R}{^{\alpha}_{m}}(u,\psi)(x)\right| &\leq  \dfrac{1}{m! \left|\Gamma(- \alpha)\right|}\int_{a}^{x} \int_{y}^{x} \dfrac{|u(y)|}{(x-y)^{1+\alpha}} \cdot \left| \psi^{(m+1)} (z)\right| (x-z)^{m} \, dz dy\\
            &\leq \dfrac{M_2}{m! |\Gamma(- \alpha)|} \int_{a}^{x} \int_{y}^{x} \dfrac{|u(y)|}{(x-y)^{1+\alpha}} (x-z)^{m} \, dz dy\\
            &= \dfrac{M_2}{(m+1)! |\Gamma(- \alpha)|} \int_{a}^{x} \dfrac{\left|u(y) \right|}{(x-y)^{\alpha -m}} \, dy \\ 
            &= \dfrac{M_2}{(m+1)! |\Gamma(- \alpha)|} {^{-}}{I}{^{ m - \alpha+1}}|u|(x)
        \end{align*}
        where $M_2 : = \sup_{x \in \Omega} \left| \psi^{(m+1)} (x) \right|$. Since $\Gamma( - \alpha) \neq 0$, hence the 
        coefficient is finite. The same estimate holds for ${^{+}}{R}{^{\alpha}_{m}}(u,\psi)$ as well. Thus,
        \begin{align*}
            \left\|{^{\pm}}{R}{^{\alpha}_{m}}(u,\psi) \right\|_{L^{p}(\Omega)} 
            &\leq \left\| \dfrac{M_3}{(m+1)! |\Gamma(- \alpha)|} {^{\pm}}{I}{^{ m - \alpha+1}}|u| \right\|_{L^{p}(\Omega)} \\ 
            &\leq \dfrac{M_2|\Omega|^{m - \alpha +1}}{(m+1)! (m- \alpha +1) |\Gamma(- \alpha)|  \Gamma(m - \alpha +1)} \left\| u \right\|_{L^{p}(\Omega)} < \infty .
        \end{align*}
        This concludes that ${^{\pm}}{\mathcal{D}}{^{\alpha}} (u \psi) \in L^{p}(\Omega)$, consequently, $u \psi \in {^{\pm}}{W}{^{\alpha,p}}(\Omega)$. 
    \end{proof}

We are now ready to state and prove the following fractional counterpart of Meyers and Serrin's  
``$H = W$" result.

    \begin{theorem}\label{H=W}
        Let $ \alpha >0$ and $1\leq p <\infty$. Then ${^{\pm}}{\overline{W}}{^{\alpha,p}}(\Omega) = {^{\pm}}{W}{^{\alpha,p}}(\Omega)$.
    \end{theorem}

    \begin{proof}
        Because ${^{\pm}}{W}{^{\alpha,p}}(\Omega)$ are Banach spaces, then by the definition we have  ${^{\pm}}{\overline{W}}{^{\alpha,p}}(\Omega) \subseteq {^{\pm}}{W}{^{\alpha,p}}(\Omega)$. To show 
        the reverse inclusion ${^{\pm}}{\overline{W}}{^{\alpha,p}}(\Omega) \supseteq {^{\pm}}{W}{^{\alpha,p}}(\Omega)$, it suffices to show that $C^{\infty}(\Omega)$ is dense in ${^{\pm}}{W}{^{\alpha , p}}(\Omega)$.  This will be done in the same fashion as in the 
        integer order case given in \cite{Meyers} (also see \cite{Adams, Evans}). Below we shall only give a proof for 
        the case $0<\alpha< 1$ because the case $\alpha >1$ follows similarly.
        
        For $k = 1,2...$ let 
        \begin{align*}
            \Omega_{k} = \left\{x \in \Omega : |x| < k \text{ and } \text{dist}(x , \partial \Omega) > \frac{1}{k} \right\}.
        \end{align*}
        For convenience, let $\Omega_{-1} = \Omega _0 = \emptyset$. Then 
        \begin{align*}
            \Theta = \left\{ \Omega'_{k} : \Omega'_{k} = \Omega_{k+1} \setminus \overline{\Omega}_{k-1} \right\}
        \end{align*}
        is an open cover of $\Omega$. Let $\{\psi_k\}_{k =1}^{\infty}$ be a $C^{\infty}$-partition of unity of $\Omega$ subordinate to $\Theta$ so that  $\supp\left( \psi_{k} \right) \subset \Omega'_{k}$. Then $\psi_{k} \in C^{\infty}_{0} \left(\Omega'_{k} \right)$. 

        If $ 0 < \eps < \frac{1}{(k+1)(k+2)}$, let $\eta_{\eps}$ be a $C^{\infty}$ mollifier satisfying
        \begin{align*}
            \supp \left( \eta_{\eps}\right) \subset \left\{x : |x| < \frac{1}{(k+1)(k+2)}\right\}.
        \end{align*}
        Evidently,  $\eta_{\eps} * \left( \psi_{k} u \right)$ has support in $\Omega_{k+2} \setminus \overline{\Omega}_{k-2} \subset \subset \Omega$. Since $\psi_{k} u \in {^{\pm}}{W}{^{\alpha,p}}(\Omega)$ we can choose $0 < \eps_{k} < \frac{1}{(k+1)(k+2)}$ such that 
        \begin{align*}
            \left\|\eta_{\eps_{k}} * (\psi_ku) - \psi_{k} u \right\|_{{^{\pm}}{W}{^{\alpha,p}}(\Omega)} < \dfrac{\eps}{2^{k}}.
        \end{align*}
        Let $v = \sum_{k=1}^{\infty} \eta_{\eps_{k}} * (\psi_{k} u)$. On any $U \subset \subset \Omega$ only finitely
         many terms in the sum can fail to vanish. Thus, $v \in C^{\infty}(\Omega)$. For $x \in \Omega_{k}$ we have 
        \begin{align*}
            u(x) = \sum_{j =1}^{k+2} (\psi_{j} u)(x), \qquad v(x) = \sum_{j=1}^{k+2} \left(\eta_{\eps_{j}} * (\psi_{j}u)\right)(x). 
        \end{align*}
        Therefore, 
        \begin{align*}
            \|u - v\|_{{^{\pm}}{W}{^{\alpha, p}}\left(\Omega_k\right)} &= \biggl\|\sum_{j=1}^{k+2} \eta_{\eps_j}*(\psi_{j}u) - \psi_{j} u \biggr\|_{{^{\pm}}{W}{^{\alpha , p}}(\Omega_{k})} \\
            &\leq \sum_{j=1}^{k+2} \left\|\eta_{\eps_j}* (\psi_{j}u) - \psi_{j} u \right\|_{{^{\pm}}{W}{^{\alpha , p}}(\Omega)} < \eps < \frac{1}{(k+1)(k+2)}.
        \end{align*}
        Setting $k \rightarrow \infty$ and applying the Monotone Convergence theorem yields the desired result 
        $\left\|u - v\right\|_{{^{\pm}}{W}{^{\alpha,p}}(\Omega)} < \eps$. The proof is complete.
    \end{proof}
    
    One crucial difference between integer order Sobolev spaces $W^{k,p}(\Omega)$ and fractional order Sobolev spaces $^{{\pm}}{W}{^{\alpha,p}}(\Omega)$ (for $0<\alpha<1$) is that
    piecewise constant functions are not dense in the former, but are dense in the latter  
    (see the next theorem below). 
    Such a difference helps characterize a major difference between the fractional order weak derivatives and 
    integer order weak derivatives.
    
    %
    %

    \begin{theorem}
        Let $\Omega=(a,b)$,  $ \alpha >0 $ and $1 \leq p <\infty$ so that $\alpha p  <1$. Then piecewise constant functions are dense in $^{{\pm}}{W}{^{\alpha,p}}(\Omega)$.
    \end{theorem}

    \begin{proof}
        Let $\eps > 0$ and $u \in {^{\pm}}{W}{^{\alpha,p}}((a,b))$ for $0<\alpha <1$.  The case when $\alpha >1$ follows as a direct consequence of the Riemann-Liouville derivative definition and the calculations below. Since $C^{\infty}((a,b))$ is dense in ${^{\pm}}{W}{^{\alpha,p}}((a,b))$, then there exists $v \in C^{\infty}((a,b))$ such that $\|u-v\|_{{^{\pm}}{W}{^{\alpha,p}}((a,b))} < \frac{\eps}{2}$. Moreover, choose a piecewise constant 
        function $w$ such that $\sup_{x \in (a,b)} |v(x) - w(x)| < \frac{\eps}{2}\max\{|b-a|^{1-\alpha p} , |b-a|\} =:M$.
        Then 
        \begin{align*}
            \|u - w\|_{L^{p}((a,b))}^{p} &\leq \|u-v\|_{L^{p}((a,b))}^{p} + \|v - w\|_{L^{p}((a,b))}^{p} \\ 
            &< \dfrac{\eps}{2} + \int_{a}^{b} |v-w|^{p}\,dx 
            <\dfrac{\eps}{2} + \left( \dfrac{\eps}{M}\right)^{p} |b-a| 
            \leq  \eps .
        \end{align*}
        Similarly,  on noting that ${^{\pm}}{\mathcal{D}}{^{\alpha}}w$ exists and belongs to $L^p((a,b))$, 
        we have
        \begin{align*}
            \left\|{^{\pm}}{\mathcal{D}}{^{\alpha}} u - {^{\pm}}{\mathcal{D}}{^{\alpha}}w \right\|_{L^{p}((a,b))}^{p} &\leq \left\|{^{\pm}}{\mathcal{D}}{^{\alpha}} u - {^{\pm}}{\mathcal{D}}{^{\alpha}}v \right\|_{L^{p}((a,b))}^{p} + \left\|{^{\pm}}{\mathcal{D}}{^{\alpha}} v - {^{\pm}}{\mathcal{D}}{^{\alpha}}w \right\|_{L^{p}((a,b))}^{p}\\
            &= \left\|{^{\pm}}{\mathcal{D}}{^{\alpha}} u - {^{\pm}}{\mathcal{D}}{^{\alpha}}v \right\|_{L^{p}((a,b))}^{p} + \left\|{^{\pm}}{D}{^{\alpha}} v - {^{\pm}}{D}{^{\alpha}}w \right\|_{L^{p}((a,b))}^{p}\\
            &< \dfrac{\eps}{2} + \left\|{^{\pm}}{D}{^{\alpha}} v - {^{\pm}}{D}{^{\alpha}}w \right\|_{L^{p}((a,b))}^{p},
        \end{align*}
        and the last term can be bounded as follows 
        \begin{align*}
            \left\|{^{\pm}}{D}{^{\alpha}} v - {^{\pm}}{D}{^{\alpha}}w \right\|_{L^{p}((a,b))}^{p}&= \int_{a}^{b} \left| \dfrac{d}{dx} \int_{a}^{x} \dfrac{v(y) - w(y)}{(x-y)^{\alpha}}\,dy \right|^{p}\,dx\\
            &< \left(\dfrac{\eps}{2 M}\right)^{p} \int_{a}^{b} \left| \dfrac{d}{dx}\int_{a}^{x}\dfrac{dy}{(x-y)^{\alpha}}\right|^{p}\,dx \\ 
            &= \left(\dfrac{\eps}{2 M}\right)^{p} \int_{a}^{b} \dfrac{dx}{(x-a)^{\alpha p}} 
            <\dfrac{\eps}{2}.
        \end{align*}
        This proves the assertion. 
    \end{proof}

    \subsubsection{\bf Infinite Domain Case: $\Omega=\R$}\label{sec-5.4.2}
    The approximation of functions in the fractional Sobolev functions on $\R$ is much easier than the case when $\Omega=(a,b)$. In this case, all of the legwork has already been done in the characterization theorem for weak derivatives.
    
%

    \begin{theorem}
        Let $\alpha >0$ and $1 \leq p <\infty$. Then $C^{\infty}_{0}(\R)$ is dense in ${^{\pm}}{W}{^{\alpha,p}}(\R)$. Hence, ${^{\pm}}{\overline{W}}{^{\alpha,p}}(\R) =
        {^{\pm}}{\overline{W}}{^{\alpha,p}_{0}}(\R)= {^{\pm}}{W}{^{\alpha,p}}(\R)$.
    \end{theorem}
    
    \begin{proof}
        We only give a proof for the case $0<\alpha<1$ since the case $\alpha>1$ follows similarly. Let $u \in {^{\pm}}{W}{^{\alpha,p}}(\R)$. Recall that there exists a sequence $\left\{u_j \right\}_{j = 1}^{\infty} \subset C^{\infty}_{0}(\R)$ such that $u_j \rightarrow u$ in $L^{p}(\R)$ and ${^{\pm}}{\mathcal{D}}{^{\alpha}} u_j \rightarrow {^{\pm}}{\mathcal{D}}{^{\alpha}} u$ in $L^{p}(\R)$ because ${^{\pm}}{\mathcal{D}}{^{\alpha}} u \in L^{p}(\R)$. The proof is complete. 
    \end{proof}

    \subsection{Extension Operators}\label{sec-5.5}
    In this subsection we address the issue of extending Sobolev functions from a finite domain $\Omega=(a,b)$ to 
    the real line $\R$. As we shall see below, constructing such an extension operator in ${^{\pm}}{W}{^{\alpha,p}}(\Omega)$ requires a different process and added conditions relative to the integer order case. Recall that spaces ${^{\pm}}{W}{^{\alpha,p}}(\Omega)$ differ greatly from 
    integer Sobolev spaces due to the following characteristics: {\em 
    (i) ${^{\pm}}{W}{^{\alpha,p}}$ is direction-dependent  and domain-dependent;
    (ii) fractionally differentiable functions inherit singular kernel functions;
    (iii) continuity is not a necessary condition for fractional differentiability;
    (iv) compact support is a desirable property to dampen the singular effect of the kernel functions and pollution.}
    Moreover, we also note that due to the pollution effect 
    of fractional derivatives,  zero function values may result in nonzero contribution to fractional derivatives, 
    as having seen  many times previously, controlling the pollution contributions is also the key 
    in the subsequent analysis.  
   
 \subsubsection{\bf Extensions of Compactly Supported Functions}\label{sec-5.5.1}
    We first consider the easy case of compactly supported functions. In this case, we show that 
    the trivial extension will do the job. 
    
    \begin{lemma}\label{TrivialExtension}
        Let $\Omega=(a,b)$, $\alpha>0$, and $1 \leq p < \infty$. If $u \in {^{\pm}}{W}{^{\alpha,p}}(\Omega)$ and $K:=\supp(u) \subset\subset\Omega$, then the trivial extension $\tilde{u}$ belongs to ${^{\pm}}{W}{^{\alpha,p}}(\R)$ and there exists 
        $C = C(\alpha, p,K)>0$ such that
        \begin{align*}
            \|\tilde{u}\|_{{^{\pm}}{W}{^{\alpha,p}}(\R)} \leq C \left\|u\right\|_{{^{\pm}}{W}{^{\alpha, p}}(\Omega)}.
        \end{align*}
    \end{lemma}
 
    \begin{proof}
        Let $\{u_j\}_{j = 1}^{\infty}  \subset C^{\infty}_{0}(\Omega)$ be an approximating sequence of $u$ and define
        \begin{align*}
            \tilde{u}_j(x):= \begin{cases}
             u_j(x) &\text{if } x\in \Omega, \\ 
             0 &\text{if } x \in \R \setminus \Omega .
             \end{cases}
         \end{align*}
         Clearly, $\|\tilde{u}_j\|_{L^{p}(\R)} = \|u_j\|_{L^{p}(\Omega)} <\infty$. Let $\varphi \in C^{\infty}_{0}(\R)$ and by Theorem \ref{IBP},
        \begin{align*}
            \int_{\R} \tilde{u}_j {^{\mp}}{D}{^{\alpha}}\varphi  &= \int_{\R} {^{\pm}}{D}{^{\alpha}} \tilde{u}_j \varphi .
        \end{align*}
        For clarity, let $\supp(u_j) \subset K\subseteq (c,d) \subset\subset (a,b)$ and we look at the left derivative.
        \begin{align*}
            \left\|{^{-}}{\mathcal{D}}{^{\alpha}} \tilde{u}_j \right\|_{L^{p}(\R)}^{p} &= \left\| {^{-}}{D}{^{\alpha}} \tilde{u}_j \right\|_{L^{p}(\R)}^{p}\\
            &= \left\| {^{-}}{D}{^{\alpha}} u_j \right\|_{L^{p}((a,b))}^{p} + \left\|L(u_j)\right\|_{L^{p}((b,\infty))}^{p} \\ 
            &= \left\|{^{-}}{\mathcal{D}}{^{\alpha}}u_j \right\|_{L^{p}((a,b))}^{p} + \int_{b}^{\infty} \biggl| \int_{a}^{b} \dfrac{u_j(y)}{(x-y)^{1+\alpha}} \,dy \biggr|^{p}\,dx \\ 
            &= \left\|{^{-}}{\mathcal{D}}{^{\alpha}}u_j \right\|_{L^{p}((a,b))}^{p} + \int_{b}^{\infty} \biggl| \int_{c}^{d} \dfrac{u_j(y)}{(x-y)^{1+\alpha}} \,dy \biggr|^{p}\,dx \\ 
            &\leq \left\|{^{-}}{\mathcal{D}}{^{\alpha}}u_j \right\|_{L^{p}((a,b))}^{p} + \|u_j\|_{L^{p}((a,b))}^{p} \int_{b}^{\infty} \biggl( \int_{c}^{d} \dfrac{dy}{(x-y)^{q(1+\alpha)}} \biggr)^{\frac{p}{q}}\,dx \\ 
            &\leq \left\|{^{-}}{\mathcal{D}}{^{\alpha}}u_j \right\|_{L^{p}((a,b))}^{p} + \|u_j\|_{L^{p}((a,b))}^{p} \int_{b}^{\infty}  \dfrac{(b-a)^{\frac{p}{q}}}{(x -d)^{p(1+ \alpha)}}\,dx \\ 
            &=\left\|{^{-}}{\mathcal{D}}{^{\alpha}}u_j \right\|_{L^{p}((a,b))}^{p} + \|u_j\|_{L^{p}((a,b))}^{p}\dfrac{(b-a)^{\frac{p}{q}}}{(d-b)^{p(1+\alpha) - 1}}.
        \end{align*}
        Then there exists $C= C(\alpha,p,K)$ such that 
        \begin{align*}
            \left\|\tilde{u}_j\right\|_{{^{\pm}}{W}{^{\alpha,p}}(\R)} \leq C \left\|u_j\right\|_{{^{\pm}}{W}{^{\alpha,p}}(\Omega)}.
        \end{align*}
        Now, we need to show that the appropriate limits are realized. By construction, $\left\|u_j \right\|_{{^{\pm}}{W}{^{\alpha,p}}(\Omega)} \rightarrow \|u\|_{{^{\pm}}{W}{^{\alpha,p}}(\Omega)}.$ Therefore, $\lim_{j\rightarrow \infty} \left\|\tilde{u}_j\right\|_{{^{\pm}}{W}{^{\alpha,p}}(\R)} < \infty$. Let $\eps > 0$ and choose sufficiently large $m$ and $n$ so that
        \begin{align*}
            \left\|\tilde{u}_m - \tilde{u}_n\right\|_{{^{\pm}}{W}{^{\alpha,p}}(\R)} \leq C \left\|u_n - u_m \right\|_{{^{\pm}}{W}{^{\alpha,p}}(\Omega)} < \eps.
        \end{align*}
        Hence, $\left\{\tilde{u}_j\right\}_{j=1}^{\infty}$ is a Cauchy sequence in ${^{\pm}}{W}{^{\alpha,p}}(\R)$. Since ${^{\pm}}{W}{^{\alpha,p}}(\R)$ is complete, there exists $v\in {^{\pm}}{W}{^{\alpha,p}}(\R)$ such that $u_j \rightarrow v$ in ${^{\pm}}{W}{^{\alpha,p}}(\R)$. We claim finally that $v = \tilde{u}$ almost everywhere. For sufficiently large $j$, we have that
        \begin{align*}
            \left\|\tilde{u} - v \right\|_{L^{p}(\R)}&\leq \left\|\tilde{u} - \tilde{u}_j\right\|_{L^{p}(\R)}+\left\|\tilde{u}_j - v\right\|_{L^{p}(\R)}\\
            &= \left\|\tilde{u} - \tilde{u}_j\right\|_{L^{p}(\Omega)} + \left\|\tilde{u}_j - v\right\|_{L^{p}(\R)} \\ 
            &= \left\|u - u_j \right\|_{L^{p}(\Omega)} + \left\|\tilde{u}_j - v\right\|_{L^{p}(\R)} 
            < \eps. 
        \end{align*}
        Therefore, $v= \tilde{u}$ almost everywhere in $\R$. 
        This concludes that the trivial extension $\tilde{u}$ satisfies the desired properties on compactly supported functions.
    \end{proof}
    
    \begin{corollary}
         The same result holds for any $u \in  {W}^{\alpha,p}(\Omega)$ with the same construction. 
    \end{corollary}
    

    \subsubsection{\bf Interior Extensions}\label{sec-5.5.2}
    For any function $u \in {^{\pm}}{W}{^{\alpha,p}}(\Omega)$ and $\Omega' \subset\subset \Omega$, 
    we first rearrange  $u$ in  $\Omega\setminus \Omega'$ so that the rearranged function $u^*$ has compact support in $\Omega$ and coincide with $u$ in $\Omega'$.  With the help of such a rearrangement and the
     extension  result of the previous subsection we then can extend any function $u$ in 
     $ {^{\pm}}{W}{^{\alpha,p}}(\Omega)$ to a function in ${^{\pm}}{W}{^{\alpha,p}}(\R)$ with some 
     preferred properties. We refer to such an extension as an {\em interior extension} of $u$. 
    
    \begin{lemma}\label{InteriorExtesionLemma}
    	Let $\Omega=(a,b)$ and $\alpha>0$. 
        For each $\Omega' \subset \subset  \Omega$, there exists a compact set $K\subset \Omega$ 
        and a constant $C=C(\alpha, K) >0$, such that for every $u \in {^{\pm}}{W}{^{\alpha,p}}(\Omega)$, 
        there exists $u^* \in {^{\pm}}{W}{^{\alpha,p}}(\Omega)$ satisfying 
        \begin{enumerate}
            \item[{\rm (i)}] $u^* = u$   a.e. in $\Omega'$,
            \item[{\rm (ii)}] $\supp(u^*) \subseteq K$,
            \item[{\rm (iii)}] $\|u^*\|_{{^{\pm}}{W}{^{\alpha,p}}(\Omega)} 
            \leq C \|u\|_{{^{\pm}}{W}{^{\alpha,p}}(\Omega)}.$
        \end{enumerate}
    \end{lemma}
    
    \begin{proof}
    	Again, we only give a proof for the case $0<\alpha<1$ because the case $\alpha>1$ can be proved similarly.
        Choose $\Omega' \subset \subset \Omega$. Let $\{B_{i}\}_{i=1}^{N}$ be a cover of $\overline{\Omega'}$ 
        with a subordinate partition of unity $\{ \psi_{i}\}_{i=1}^{N} \subset C^{\infty}(\Omega)$ in the sense that $\supp(\psi_i)\subset B_i$ for $i=1,2,\cdots, N$. Define 
        $u^*:\Omega \rightarrow \R$ by $u^* := u \psi$ with $\psi:= \sum_{i=1}^{N} \psi_{i}$. 
        Note that $u^* = u$ almost 
        everywhere on $\Omega'$ and $\supp(u^*) \subseteq K:=\cup \overline{B}_i$. We need to show 
        that $u^* \in {^{\pm}}{W}{^{\alpha,p}}(\Omega)$. Of course, 
        \begin{align*}
            \|u^*\|_{L^{p}(\Omega)} &= \bigl\| u \psi  \bigr\|_{L^{p}(\Omega)} 
            = \bigl\| u \psi \bigr\|_{L^p(K)} \leq \| u \|_{L^{p}(K)} \leq \|u\|_{L^{p}(\Omega)}.
        \end{align*}
        Next, applying Theorem \ref{WeakProductRule}, we have 
        \begin{align*}
            {^{\pm}}{\mathcal{D}}{^{\alpha}} u^* &= {^{\pm}}{\mathcal{D}}{^{\alpha}} u \cdot 
           \psi + \sum_{k=1}^{m} C(k,\alpha) {^{\pm}}{I}{^{k-\alpha}} u  D^{k} \psi  + {^{\pm}}{R}{^{\alpha}_{m}} (u,\psi).
        \end{align*}
        It follows by direct calculations that 
        \begin{align*}
            \|{^{\pm}}{\mathcal{D}}{^{\alpha}} u^*\|_{L^{p}(\Omega)} \leq C \|u\|_{{^{\pm}}{W}{^{\alpha,p}}(\Omega)}.
        \end{align*}
        Hence, $u^* \in {^{\pm}}{W}{^{\alpha,p}}(\Omega)$ and there exists $C >0$ 
        such that assertion (iii) holds. The proof is complete. 
    \end{proof}

Now, we are ready to state the following interior extension theorem. 

    \begin{theorem}\label{InteriorExtension}
    	Let $\Omega=(a,b)$ and $\alpha>0$. 
        For each $\Omega' \subset\subset \Omega$,   there exist  a compact set $K\subset \Omega$  and  a constant $C = C(\alpha ,p, K) >0$ so that for every $u \in {^{\pm}}{W}{^{\alpha,p}}(\Omega)$,  
        there exists a mapping $E: {^{\pm}}{W}{^{\alpha,p}}(\Omega) \rightarrow {^{\pm}}{W}{^{\alpha,p}}(\R)$ so that 
        \begin{enumerate}
            \item[{\rm (i)}]  ${Eu} = u$  a.e. in $\Omega'$,
            \item[{\rm (ii)}] $\supp(E{u})\subseteq K,$
            \item[{\rm (iii)}] $\|E{u}\|_{{^{\pm}}{W}{^{\alpha,p}}(\R)} \leq C \|u\|_{{^{\pm}}{W}{^{\alpha,p}}(\Omega)}.$
        \end{enumerate}
        We call $Eu$ an (interior) extension of $u$ to $\R$. 
        
    \end{theorem}
 
    \begin{proof}
        For $u\in {^{\pm}}{W}{^{\alpha,p}}(\Omega) $, let $u^* \in {^{\pm}}{W}{^{\alpha,p}}(\Omega) $ denote the rearrangement of $u$  as defined in Lemma \ref{InteriorExtesionLemma},  let  $K\subset\subset \Omega$ and $C(\alpha, K)$ be the same as well. Since $u^*$ has a  compact support in $\Omega$,  
         we can invoke Lemma \ref{TrivialExtension}  to conclude that  $E{u} :=
        \widetilde{u^*}$ satisfies the desired properties (i)--(iii) with $C=C(\alpha,p,K) C(\alpha, K)$. The proof is complete. 
    \end{proof}

    %
    
  \begin{remark}
  	We emphasize that  the extension operator $E$ defined above depends on the choice of subdomain $\Omega'$, 
  	on the other hand, it does not depend on the left or right direction,  consequently,  $E$ also provides a valid 
  	{\em interior extension operator} from the symmetric space $W^{\alpha ,p}(\Omega)$ to the symmetric space $ {W}^{\alpha,p}(\R)$. 
  \end{remark}
    
    \subsubsection{\bf Exterior Extensions}\label{sec-5.5.3}
     The goal of this subsection is to construct a more traditional (exterior) extension so that the extended function 
     coincides with the original function in the entire domain $\Omega$ where the latter is defined. As we alluded  earlier, if we do not want to pay in part of the domain, we need to pay with a restriction on 
     the function to be extended.
     
    \begin{lemma}\label{ExtensionOutsideLemma}
        Let $\Omega=(a,b)$, $0 < \alpha <1$ and $1 \leq p < \infty$. Assume that $\alpha p<1$ and $\mu \in \R$ so that $\mu > p(1-\alpha p)^{-1}$ (hence, $\mu>p$). Then for every bounded domain 
        $\Omega ' \supset \supset \Omega$, 
        there exists a constant $C = C(\alpha,p,\Omega') >0$ such that for every $u \in {^{\pm}}{W}{^{\alpha,p}}(\Omega) \cap L^{\mu}(\Omega)$, there exists $u^{\pm} \in {^{\pm}}{W}{^{\alpha,p}}(\Omega')$ such that 
        \begin{enumerate}
            \item[{\rm (i)}] $u^{\pm} = u$ a.e. in $\Omega$,
            \item[{\rm (ii)}] $\supp(u^{\pm})\subset\subset \Omega'$,
            \item[{\rm (iii)}] $\|u^{\pm}\|_{{^{\pm}}{W}{^{\alpha,p}}(\Omega')} \leq C \|u\|_{{^{\pm}}{W}{^{\alpha,p}}(\Omega)}$. 
        \end{enumerate}
    \end{lemma}
    
    \begin{proof}
        Let $u \in {^{\pm}}{W}{^{\alpha,p}}(\Omega) \cap L^{\mu}(\Omega)$. For ease of
        presentation and without loss of the generality,  we only consider the left weak fractional  derivative  with $\Omega = (0,1)$.
        
         Let $\Omega' = (-1 ,2)$, $\{B_{i}\}_{i=1}^{N} \subset \Omega'$ be a cover of 
        $\overline{\Omega}$ and $\{\psi_i\}_{i=1}^{N}$ be a subordinate partition of unity so that 
        $\supp(\varphi_i) \subset B_i$ for $i=1,2,\cdots, N$. 
         Define $u^{-} = \overline{u}^{-} \psi$ in $\Omega'$, where 
        \begin{align*}
           \psi:=\sum_{i=1}^{N}\psi_{i}, \qquad 
            \overline{u}^{-}(x):= \begin{cases}
                0 &\text{if }  x \in (-1,0), \\
                u &\text{if } x \in (0,1) ,\\ 
                u(x-1) &\text{if } x \in (1,2).
            \end{cases}
        \end{align*}
       Notice that $\overline{u}^{-}$ is a periodic extension of $u$ to the right on interval $(1,2)$.  
        
        Trivially, $\|u^{-}\|_{L^{p}(\Omega')} \leq 2 \|u\|_{L^p(\Omega)}$. It remains to prove that $u^{-}$ is weakly differentiable in $L^{p}(\Omega')$. To the end, let $\{u_j\}_{j= 1}^{\infty} \subset C^{\infty}(\Omega)$ such that $u_j  \rightarrow u$ in ${^{-}}{W}{^{\alpha,p}}(\Omega) \cap L^{\mu}(\Omega)$ as
        $j\to \infty$. Let $u_j^-:=\overline{u}^{-}_{j} \psi$ and $\overline{u}^{-}_{j} $ is the extension of $u_j$ to $\Omega'$
        constructed in the same way  as $\overline{u}^{-}$ is done above.  
           %
         Since $u_j\to u$ in $L^\mu(\Omega)$, by the construction, we  have $\overline{u}^{-}_j \to \overline{u}^{-}$ and  
         $u_j^- \to  u^-r$ in $L^{\mu}(\Omega')$. 
         Hence, $\{u_j^-\}_
         {j=1}^{\infty}$ is bounded in $L^{\mu}(\Omega')$. $\{{^{-}}{\mathcal{D}}{^{\alpha}}u_j\}_{j=1}^{\infty}$ is bounded in $L^{p}(\Omega)$ because  ${^{-}}{\mathcal{D}}{^{\alpha}}u_j\to {^{-}}{\mathcal{D}}{^{\alpha}}u$ in $L^{p}(\Omega)$.  Let $M >0$ be such a bound for both sequences. 
         
         Now, using the fact that ${^{-}}{\mathcal{D}}{^{\alpha}} u_j^- = {^{-}}{D}{^{\alpha}} u_j^-$
         and the definition of $u_j^-$ we have 
        \begin{align*}
            \|{^{-}}{\mathcal{D}}{^{\alpha}} u_j^-\|_{L^{p}(\Omega')}^{p} &=
            \|{^{-}}{D}{^{\alpha}} u_j^-\|_{L^{p}(\Omega')}^{p} 
            = \int_{-1}^{2}\left|\dfrac{d}{dx} \int_{-1}^{x} \dfrac{u_j^-(y)}{(x-y)^{\alpha}}\right|^{p}\,dx \\ 
            &\leq \int_{0}^{1}\left|\dfrac{d}{dx} \int_{0}^{x} \dfrac{u_j(y)}{(x-y)^{\alpha}}\,dy \right|^{p}\,dx 
            + \int_{1}^{2} \left|\int_{0}^{1} \dfrac{u_j(y)}{(x-y)^{1+\alpha}}\,dy \right|^{p}\,dx \\
            &\qquad +  \int_{1}^{2}\left|\dfrac{d}{dx} \int_{1}^{x} \dfrac{\overline{u}_j\psi}{(x-y)^{\alpha}}\,dy \right|^{p}\,dx \\
            &\leq \|{^{-}}{D}{^{\alpha}} u_j \|_{L^{p}((0,1))}^{p} + \int_{1}^{2} \left|\int_{0}^{1} \dfrac{u_j(y)}{(x-y)^{1+\alpha}}\,dy \right|^{p}\,dx  \\
            &\qquad + \int_{1}^{2}\left|\dfrac{d}{dx} \int_{1}^{x} \dfrac{\overline{u}_j\psi}{(x-y)^{\alpha}}\,dy \right|^{p}\,dx.
        \end{align*}
        
        Next, we bound  the last two terms above separately.
        To bound the second to the last (middle) term,  let  $\nu$ be the H\"older conjugate of $\mu$, then we have 
        \begin{align*}
            \int_{1}^{2} \biggl|\int_{0}^{1} \dfrac{u_j(y)}{(x-y)^{1+\alpha}}\,dy &\biggr|^{p}\,dx \leq \|u_j\|_{L^{\mu}((0,1))}^{p} \int_{1}^{2}\left(\int_{0}^{1}\dfrac{dy}{(x-y)^{\nu(1+\alpha)}}\right)^{\frac{p}{\nu}}\,dx\\
            &= M^{p} \| u_j \|_{L^{\mu}((0,1))}^{p} \int_{1}^{2} \Bigl( x^{1- \nu(1+\alpha)} - (x-1)^{1 - \nu(1+\alpha)} \Bigr)^{\frac{p}{\nu}}\,dx \\
            &\leq M^{p}\| u_j \|_{L^{\mu}((0,1))}^{p} \int_{1}^{2} x^{\frac{p}{\nu} - p(1+\alpha)} + (x-1)^{\frac{p}{\nu} -p(1+\alpha)}\,dx.
        \end{align*}
        In order for this term to be finite, $p\nu^{-1}- p(1+\alpha)>-1 $ must holds,  which 
        implies that $\mu > p(1 - \alpha p)^{-1}$, which is assumed in the statement of the theorem. 
        
        Lastly,  to bound the final term, using the product rule we get 
        \begin{align*}
            \int_{1}^{2} \left|\dfrac{d}{dx} \int_{1}^{x} \dfrac{\overline{u}_{j} (y) \psi(y)}{(x-y)^{\alpha}}\,dy \right|^{p}\,dx 
            &\leq \|{^{-}}{D}{^{\alpha}} u_j \|_{L^{p}((0,1))}^{p} + \Bigl\|\sum_{k=1}^{m} C_{k} {^{-}}{I}{^{k-\alpha}}  u_j\, D^k   \psi \Bigr\|_{L^{p}((1,2))}^{p} \\
            &\qquad + \|{^{-}}{R}{^{\alpha}_{m}}(u_j,  \psi)\|_{L^{p}((1,2))}^{p} \\
            &\leq C \Bigl(\|{^{-}}{\mathcal{D}}{^{\alpha}} u_j\|_{L^{p}((0,1))}^{p} + \|u_j\|_{L^p((0,1))}^{p} \Bigr).
        \end{align*}
        It follows for given $\eps > 0$ and sufficiently large $m,n$, 
        \begin{align}
            \|{^{\pm}}{\mathcal{D}}{^{\alpha}} u_m^{\pm} - {^{\pm}}{\mathcal{D}}{^{\alpha}} u_n^{\pm} \|_{L^{p}(\Omega')}^{p} \leq C\left( \|u_m - u_n \|_{{^{\pm}}{W}{^{\alpha,p}}(\Omega)}^{p} + \| u_m - u_n \|_{L^{\mu}(\Omega)}^{p}\right) < \eps.
        \end{align}
        Therefore, there exists $v \in L^{p}(\Omega')$ so that ${^{\pm}}{\mathcal{D}}{^{\alpha}}u_j^{\pm} \rightarrow v$ in $L^{p}(\Omega')$. It is easy to see then that $v = {^{\pm}}{\mathcal{D}}{^{\alpha}} u^{\pm}$ using the definition of the weak derivative. Hence ${^{\pm}}{\mathcal{D}}{^{\alpha}}u^{\pm} \in L^{p}(\Omega')$. This completes the proof. 
    \end{proof}

\begin{remark}
(a) We note that there is no redundancy in assumption 
$u \in {^{\pm}}{W}{^{\alpha,p}}(\Omega)$ $\cap L^\mu(\Omega)$ for $\mu > p(1-\alpha p)^{-1}$ 
because it will be proved in Section \ref{sec-5.6} that ${^{\pm}}{W}{^{\alpha,p}}(\Omega)$ 	
is not embedded into $L^{\mu}(\Omega)$ in general. 

(b)  It can be proved that the condition $\alpha p < 1$ and $u \in L^{\mu}(\Omega)$ for some $p < \mu \leq \infty$ 
	are necessary (given the current  calculations). 
	In order for the kernel to remain bounded, we must impose the condition 
	$-1 < p{\nu}^{-1} -p(1+\alpha) < 0$ which implies that  $(1-\alpha p)  > p\mu^{-1}$. 
	Thus,  it follows from $(1- \alpha p) > 0$ that $\alpha p < 1$.
	This shows that $\alpha p <1$ is a necessary condition for the integrability of the kernel function 
	using an estimate as shown above. Moreover, if $\mu = p$, then $\nu = p(p-1)^{-1}$ and the inequality 
	$-1 < p{\nu}^{-1} -p(1+\alpha)$ implies that $\alpha p<0$, which is a contradiction. Hence, 
	we must take $\mu > p$. In particular, $\mu = \infty$ is allowed though not necessary. We need only assume 
	that $u \in L^{\mu}(\Omega)$ with the condition $\mu > p/(1-\alpha p)$.
	
(c) The same result can be proven for $u \in{W}^{\alpha,p}(\Omega)\cap L^{\mu}(\Omega)$. In this case, $\overline{u} : = \overline{u}^{\pm}$ is taken to be the periodic extension over all of $\Omega'$.
\end{remark}
 
    \begin{theorem}\label{ExteriorExtension}
         Let $\Omega=(a,b)$, $0 < \alpha <1$ and $1 \leq p < \infty$. Assume that $\alpha p<1$ and $\mu \in \R$ so that $\mu > p(1-\alpha p)^{-1}$ (hence, $\mu>p$). Then for every bounded domain 
        $\Omega ' \supset \supset \Omega$,      
        there exists mappings $E_{\pm} :{^{\pm}}{W}{^{\alpha,p}}(\Omega)\cap L^{\mu}(\Omega) \rightarrow {^{\pm}}{W}{^{\alpha,p}}(\R)$ and $C = C(\alpha,p,\Omega') > 0$ such that for any $u \in {^{\pm}}{W}{^{\alpha,p}}(\Omega) \cap L^{\mu}(\Omega)$, 
        \begin{enumerate}
            \item[{\rm (i)}] $E_{\pm}u = u$ a.e in $\Omega$, 
            \item[{\rm (ii)}] $\supp(E_{\pm}u)\subset\subset \Omega'$,
            \item[{\rm (iii)}] $\|E_{\pm}u\|_{{^{\pm}}{W}{^{\alpha,p}}(\R)} \leq C
            \bigl(\|u\|_{{^{\pm}}{W}{^{\alpha,p}}(\Omega)} + \|u\|_{L^{\mu}(\Omega)} \bigr).$
        \end{enumerate}
    \end{theorem}

    \begin{proof}
    	For any $u \in {^{\pm}}{W}{^{\alpha,p}}(\Omega) \cap L^{\mu}(\Omega)$, let 
    	$u^{\pm} \in {^{\pm}}{W}{^{\alpha,p}}(\Omega')$ be the function defined in 
        Lemma \ref{ExtensionOutsideLemma} and set $E_{\pm}u = \widetilde{u^{\pm}}$, the trivial extension of $u^{\pm}$.
        It follows immediately  from  Lemma \ref{TrivialExtension}  that $E_{\pm}$ satisfies the desired properties.
        The proof is complete. 
    \end{proof}

\begin{corollary}
The conclusion of Theorem \ref{ExteriorExtension} also holds for functions in ${W}^{\alpha,p}(\Omega) \cap L^{\mu}(\Omega)$.
\end{corollary}

    %
   
    
    \subsection{One-side Boundary Traces and Compact Embedding}\label{sec-5.6}
    Similar to the integer order case, since functions in Sobolev spaces ${^{\pm}}{W}{^{\alpha,p}}((a,b))$
    are integrable functions, a natural question is under what condition(s) those functions can be 
    assigned pointwise values, especially, at two boundary points $x=a,b$. Such a question arises naturally when 
    studying initial and initial-boundary value problems for fractional order differential equations. 
    It turns out that the situation is more delicate in the fractional order case because the existence of the kernel functions creates a hick-up in this pursuit. We shall establish a one-side embedding result for each of spaces 
    ${^{\pm}}{W}{^{\alpha,p}}((a,b))$, which then allows us to assign one-side traces for those functions.  
    First, we establish the following classical characterization of Sobolev functions.
    
    \begin{proposition}\label{ContinuousRepresentative}
        Let $(a,b)\subset \R$, $0< \alpha <1$, $1 \leq p \leq \infty$ so that $\alpha p > 1$.
        \begin{itemize}
            \item[(i)] If $u \in {^{-}}{W}{^{\alpha,p}}((a,b))$, then for any $c \in (a,b)$, there exists $\bar{u} \in C([c,b])$ so that $u = \bar{u}$ almost everywhere in $[c,b]$. 
            \item[(ii)] If $u \in {^{+}}{W}{^{\alpha,p}}((a,b))$, then for any $c \in (a,b)$, there exists $\bar{u} \in C([a,c])$ so that $u = \bar{u}$ almost everywhere in $[a,c]$.
            \item[(iii)] If $u \in W^{\alpha,p}((a,b))$, then there exists $\bar{u} \in C([a,b])$ so that $u = \bar{u}$ almost everywhere in $[a,b]$. 
        \end{itemize}
    \end{proposition}
    
    \begin{proof}
        We only give a proof for $(i)$ because $(ii)$ follows similarly and $(iii)$ is proved by combining $(i)$ and $(ii)$. Let $u \in {^{-}}{W}{^{\alpha,p}}((a,b))$ and set $u^* = {^{-}}{I}{^{\alpha}} {^{-}}{\mathcal{D}}{^{\alpha}} u$. Then for any $\varphi \in C^{\infty}_{0}((a,b))$, it follows by Theorem \ref{LpMappings}, Theorem \ref{IBP}, and the weak derivative definition that 
        \begin{align*}
            \int_{a}^{b} u^* {^{+}}{D}{^{\alpha}} \varphi\, dx &= \int_{a}^{b} \varphi  {^{-}}{D}{^{\alpha}} u^* \, dx
            = \int_{a}^{b} \varphi  {^{-}}{D}{^{\alpha}} {^{-}}{I}{^{\alpha}} {^{-}}{\mathcal{D}}{^{\alpha}} u  \, dx \\
            &= \int_{a}^{b}  \varphi {^{-}}{\mathcal{D}}{^{\alpha}} u \, dx
            = \int_{a}^{b} u {^{+}}{D}{^{\alpha}} \varphi \, dx.
        \end{align*}
        Consequently,
        \begin{align*}
            0 = \int_{a}^{b} (u - u^*) {^{+}}{D}{^{\alpha}} \varphi\, dx = \int_{a}^{b} {^{-}}{I}{^{1-\alpha}}(u - u^*) \varphi'\, dx.
        \end{align*}
        Thus, ${^{-}}{I}{^{1-\alpha}}u - {^{-}}{I}{^{1-\alpha}} u^* = C$ a.e. in $(a,b)$. Following from Lemma \ref{lemma3.1}, $u = u^* + {^{-}}{D}{^{1-\alpha}} C$ almost everywhere. Choose $\bar{u} = u^* + {^{-}}{D}{^{1-\alpha}} C$, we have that $\bar{u} \in C([c,b])$ for every $c \in(a,b)$ and $u = \bar{u}$ almost everywhere. 
    \end{proof}
    
    \begin{remark}
        If a function $u$ belongs to ${^{\pm}}{W}{^{\alpha,p}}$, then any function $v = u$ almost everywhere must also belong to ${^{\pm}}{W}{^{\alpha,p}}$. Therefore, we do not differentiate between any two functions that may only differ from one another on a measure zero set. Proposition \ref{ContinuousRepresentative} asserts that every function $u \in {^{-}}{W}{^{\alpha,p}}((a,b))$ admits a continuous representative on $[c,b]$. When it is helpful, (i.e. giving meaning to $u(x)$ for some $x \in [c,b]$) we replace $u$ with its continuous representative $\bar{u}$. In order to avoid confusion and eliminate unnecessary notation, we will still use $u$ to denote the continuous representative.
    \end{remark}

     \begin{theorem}\label{CompactEmbedding}
    	Let $(a,b) \subset \R$, $0 < \alpha <1$ and $1 < p < \infty$. Suppose that $\alpha p >1$. 
    	\begin{itemize}
    		\item[{\rm (i)}] If $u \in {^{-}}{W}{^{\alpha,p}}((a,b))$, then for any $c \in (a,b)$, 
    		the injection ${^{-}}{W}{^{\alpha,p}}((a,b)) \hookrightarrow C^{\alpha -\frac{1}{p}}([c,b])$
    		is compact. 
     
    		\item[{\rm (ii)}] If $u \in {^{+}}{W}{^{\alpha,p}}((a,b))$, then for any $c\in (a,b)$, the injection ${^{+}}{W}{^{\alpha,p}}((a,b)) \hookrightarrow C^{\alpha -\frac{1}{p}}([a,c])$
    		is compact. 
    		
    		\item[{\rm (iii)}] If $u \in W^{\alpha,p}((a,b))$, then the injection $W^{\alpha,p}((a,b)) \hookrightarrow C^{\alpha -\frac{1}{p}}([a,b])$ is compact. 
    	\end{itemize}
        
    \end{theorem}
    
    \begin{proof}
         We only give a proof for (i) because the other two cases follow similarly. 
       
       Let ${B}_1^-$ be the unit ball in ${^{-}}{W}{^{\alpha,p}}((a,b))$ and take $u \in {B}_1^-$. Let $c \in (a,b)$. For any two distinct points $x,y \in [c,b]$ (assume $x>y$), by the FTwFC, (cf. Theorem \ref{WeakFTFC}) we get 
        \begin{align}\label{a1}
             \bigl|u (x) - u (y)\bigr| &= \biggl|c_{-}^{1-\alpha} [(x-a)^{\alpha -1} - (y-a)^{\alpha -1}] \\
              &\qquad
              + \dfrac{1}{\Gamma(\alpha)} \int_{a}^{x} \dfrac{{^{-}}{\mathcal{D}}{^{\alpha}}u(z)}{(x-z)^{1-\alpha}}\,dz 
               - \dfrac{1}{\Gamma(\alpha)} \int_{a}^{y} \dfrac{{^{-}}{\mathcal{D}}{^{\alpha}} u(z)}{(y-z)^{1-\alpha}} \,dz \biggr| \nonumber\\
            &= \biggl|c_{-}^{1-\alpha} [(x-a)^{\alpha -1} - (y-a)^{\alpha -1}] + C_{\alpha} \int_{y}^{x} \dfrac{{^{-}}{\mathcal{D}}{^{\alpha}}u(z)}{(x-z)^{1-\alpha}}\,dz \nonumber \\
            &\qquad  + C_{\alpha} \int_{a}^{y} \dfrac{{^{-}}{\mathcal{D}}{^{\alpha}}u(z)}{(x-z)^{1-\alpha}} - \dfrac{{^{-}}{\mathcal{D}}{^{\alpha}} u(z)}{(y-z)^{1-\alpha}}\,dz \biggr| \nonumber \\
            &\leq \left|c_{-}^{1-\alpha} [(x-a)^{\alpha -1} - (y-a)^{\alpha -1}]\right| + C_{\alpha} \left|\int_{y}^{x} \dfrac{{^{-}}{\mathcal{D}}{^{\alpha}}u(z)}{(x-z)^{1-\alpha}}\,dz \right| \nonumber \\
            &\quad  + C_{\alpha} \left| \int_{a}^{y} \dfrac{{^{-}}{\mathcal{D}}{^{\alpha}} u(z) [ (y-z)^{1-\alpha} - (x-z)^{1-\alpha}]}{[(y-z)(x-z)]^{1 - \alpha}}\,dz\right|. \nonumber
            \end{align}
            
           Below we bound each of the three terms on the right-hand side.
           Upon noticing that $|c_{-}^{1-\alpha}| \leq C_{\Omega,\alpha ,p} \|u\|_{{^{-}}{W}{^{\alpha,p}}(\Omega)}$, 
            \begin{align}\label{a2}
            \left|c_{-}^{1-\alpha} [ (x-a)^{\alpha -1} - (y-a)^{\alpha -1} ]\right| 
            &\leq C_{\Omega , \alpha , p} \|u\|_{{^{-}}{W}{^{\alpha,p}}(\Omega)} |\xi - a|^{\alpha -1} |x-y|\\
            &\leq C_{\Omega,\alpha,p} \|u\|_{{^{-}}{W}{^{\alpha,p}}(\Omega)} |x-y|^{\alpha - \frac{1}{p}}, \nonumber
            \end{align}
            \begin{align}\label{a3}
            C_{\alpha} \left|\int_{y}^{x} \dfrac{{^{-}}{\mathcal{D}}{^{\alpha}}u(z)}{(x-z)^{1-\alpha}}\,dz \right|
            &\leq C_{\alpha} \left\| {^{-}}{\mathcal{D}}{^{\alpha}} u \right\|_{L^{p}((a,b))} \biggl| \int_{y}^{x} (x-z)^{-q(1-\alpha)}\,dz \biggl|^{\frac{1}{q}} \\
            &\leq C_{\alpha , p} \left\|{^{-}}{\mathcal{D}}{^{\alpha}} u \right\|_{L^{p}((a,b))} |x-y|^{\frac{1}{q} - (1-\alpha)} \nonumber \\
            &= C_{\alpha , p} \left\|{^{-}}{\mathcal{D}}{^{\alpha}} u \right\|_{L^{p}((a,b))} |x-y|^{\alpha-\frac{1}{p}}, \nonumber
            \end{align}
            \begin{align}\label{a4}
            & C_{\alpha}  \int_{a}^{y} \dfrac{\left|{^{-}}{\mathcal{D}}{^{\alpha}} u(z)\right| \left|(y-z)^{1-\alpha} - (x-z)^{1-\alpha}\right|}{\left|(y-z)(x-z)\right|^{1 - \alpha}}\,dz\\
            &\hskip 1.2in \leq  C_{\alpha,p} \left\| {^{-}}{\mathcal{D}}{^{\alpha}} u \right\|_{L^{p}((a,b))} |x-y|^{\frac{1}{q} - (1-\alpha)}  \nonumber\\
            &\hskip 1.2in = C_{\alpha,p} \left\| {^{-}}{\mathcal{D}}{^{\alpha}} u \right\|_{L^{p}((a,b))} |x-y|^{\alpha- \frac{1}{p}}. \nonumber
            \end{align}
            %
        
        Substituting \eqref{a2}--\eqref{a4} into \eqref{a1} yields 
        \begin{align}\label{a5}
        |u(x) - u(y)| \leq C|x-y|^{\alpha -\frac{1}{p}} \qquad \forall x, y\in [c,b],  
        \end{align}
        where $C$ is a positive constant independent of $x$ and $y$. 
        Because $\alpha -\frac{1}{p}>0$, then ${B}^{-}_{1}$ is uniformly equicontinuous in $C([c,b])$. It follows from  Arzel\`a-Ascoli theorem that ${B}^{-}_{1}$ has compact closure in $C^{\alpha - \frac{1}{p}}([c,b])$. The proof is complete. 
    \end{proof}

    \begin{remark}
    	(a) We note that unlike the integer order case, we have proved the above 
    	   embedding results directly rather than relying on the infinite domain results and extension theorem. 
    	
        (b) From the above calculations we observe that when $c_{\pm}^{1-\alpha} = 0$, 
        the injection can be extended to the initial boundary so that ${^{\pm}}{W}{^{\alpha,p}}((a,b)) \hookrightarrow C^{\alpha - \frac{1}{p}}([a,b])$. Effectively,  $c^{1-\alpha}_{\pm} = 0$ implies that any singularity at the initial boundary is prevented; we denote this space by 
        \begin{align}\label{mathring}
        {^{\pm}}{\mathring{W}}{^{\alpha,p}}(\Omega) : = \{ u \in {^{\pm}}{W}{^{\alpha,p}}(\Omega) \, : \, c^{1-\alpha}_{\pm} = 0\}.
        \end{align}
    \end{remark}
    
The above embedding theorem motivates us to introduce the following definition of trace operators. 

    \begin{definition}
        We define trace operator ${^{-}}{T}: {^{-}}{W}{^{\alpha,p}}((a,b))\to \R$ by ${^{-}}{T}u={^{-}}{T}u|_{x=b} := u(b)$ and trace operator ${^{+}}{T}: {^{+}}{W}{^{\alpha,p}}((a,b))\to \R$ by ${^{+}}{T}u={^{+}}{T}u|_{x=a} := u(a)$.  
    \end{definition}
    
    It should be noted that the above proof demonstrates that we can confirm the following trace inequality: 
    \begin{align}\label{TraceInequality}
    	|{^{\pm}}{T} u | \leq C \|u\|_{{^{\pm}}{W}{^{\alpha,p}}(\Omega)}.
    \end{align}
    
    \subsubsection{\bf Zero Trace Spaces}\label{sec-5.6.1}
    With the help of the trace operators in spaces ${^{\pm}}{W}{^{\alpha,p}}(\Omega)$, we can define and 
    characterize different spaces with zero trace. First, we explicitly define the zero trace spaces.
    
    \begin{definition}
        Let $\Omega=(a,b)$, $0 < \alpha <1$ and $1 < p <\infty$. Suppose that  $\alpha p >1$.  Define  
        \begin{align*}
            {^{\pm}}{W}{^{\alpha,p}_{0}}(\Omega) &:= \{ u \in {^{\pm}}{W}{^{\alpha,p}}(\Omega) \,:\, {^{\pm}}{T} u = 0 \},\\
             {W}^{\alpha,p}_{0}(\Omega) &:= \{ u \in  {W}^{\alpha,p}(\Omega) : {^{-}}{T}u=0 \mbox{ and } {^{+}}{T}u = 0\}.
        \end{align*}
    \end{definition}

\begin{proposition}\label{ZeroTraceNorm}
	$\| u \|_{\leftidx{^{\pm}}{W}{^{\alpha,p}_{0}}(\Omega)}$ defines a norm.
\end{proposition}

\begin{proof}
The only thing we need to show is $0 = \| \leftidx{^{\pm}}{\mathcal{D}}{^{\alpha}} u\|_{L^{p}(\Omega)}$ if and only if $u =0$. The other properties are immediately clear by the properties of the weak fractional derivative and the $L^{p}$ norm. Of course, if $u= 0$, then as a direct consequence of the definition of weak fractional derivatives, $\leftidx{^{\pm}}{\mathcal{D}}{^{\alpha}} u = 0 $ and hence $\| \leftidx{^{\pm}}{\mathcal{D}}{^{\alpha}} u\|_{L^{p}(\Omega)} = 0$. To see the converse, assume $ \| \leftidx{^{\pm}}{\mathcal{D}}{^{\alpha}} u\|_{L^{p}(\Omega)} = 0$. Then $\leftidx{^{\pm}}{\mathcal{D}}{^{\alpha}} u = 0$ almost everywhere in $\Omega$, implying that $u$ must be in the kernel space of the derivative. Thus $u = C \kappa^{\alpha}_{\pm}$ for any $C \in \R$. Taking into consideration that $\leftidx{^{\pm}}{T}u = 0$, it follows that $u = 0$.
\end{proof}
  
    In an effort to characterize the above spaces, our goal is to link these spaces with the completion 
    spaces introduced in Section \ref{sec-5.2}. As our notion of traces is one-sided,  
    this makes the use of one-sided approximations spaces (i.e. ${^{\pm}}{C}{^{\infty}_{0}}(\Omega)$) 
    sensible. 
    
    \begin{lemma}
         Let $\Omega=(a,b)$, $0 < \alpha <1$ and $1 < p <\infty$. Suppose that  $\alpha p >1$. If $u \in {^{\pm}}{W}{^{\alpha,p}}(\Omega) \cap {^{\pm}}{C}{^{\infty}_{0}}(\Omega)$, then $u \in {^{\pm}}{\overline{W}}{^{\alpha,p}_{0}}(\Omega)$.
    \end{lemma}

    \begin{proof}
        Let $u \in {^{\pm}}{W}{^{\alpha, p}}(\Omega) \cap {^{\pm}}{C}{^{\infty}_{0}}(\Omega)$. 
        Consider the sequence $u_j := \eta_{\frac{1}{j}}*u$ with $\eta$ being the standard mollifier. Then $u_j \in {^{\pm}}{W}{^{\alpha,p}}(\Omega) \cap {^{\pm}}{C}{^{\infty}_{0}}(\Omega)$ and $u_j \rightarrow u$ in ${^{\pm}}{W}{^{\alpha,p}}(\Omega)$. Thus $u \in {^{\pm}}{\overline{W}}{^{\alpha,p}_{0}}(\Omega)$. 
    \end{proof}

The next two theorems give characterizations of the zero trace spaces.

    \begin{theorem}\label{ZeroTrace}
         Let $\Omega=(a,b)$, $0 < \alpha <1$ and $1 < p <\infty$. Suppose that  $\alpha p >1$.
         Then $ {^{\pm}}{\overline{W}}{^{\alpha , p}_{0}}(\Omega)= {^{\pm}}{W}{^{\alpha,p}_{0}}(\Omega)$ and ${\overline{W}}{^{\alpha,p}_{0}}(\Omega) = W^{\alpha,p}_{0}(\Omega)$.
    \end{theorem}

    \begin{proof}
        Let $u \in {^{\pm}}{\overline{W}}{^{\alpha,p}_{0}}(\Omega)$. Then there exists $\{u_j\}_{j=1}^{\infty} \subset {^{\pm}}{C}{^{\infty}_{0}}(\Omega)$ so that $u_j \rightarrow u$ in ${^{\pm}}{W}{^{\alpha,p}}(\Omega)$. 
        It follows that ${^{\pm}}{T} u_j = 0$ and $u_j \rightarrow u$ uniformly on $[c,b]$ or $[a,c]$ for every
        $c\in (a,b)$. Consequently, ${^{\pm}}{T} u = 0$. Thus, $ {^{\pm}}{\overline{W}}{^{\alpha , p}_{0}}(\Omega)\subset {^{\pm}}{W}{^{\alpha,p}_{0}}(\Omega)$

        Conversely, let $u \in {^{\pm}}{W}{^{\alpha,p}_{0}}(\Omega)$. We want to show that there exists $\{u^n\} \subset {^{\pm}}{C}{^{\infty}_{0}}(\Omega)$ such that $u^n \rightarrow u$ in ${^{\pm}}{W}{^{\alpha,p}}(\Omega)$. 
        For ease of presentation and without loss of the generality, let $\Omega = (0,1)$ and we only consider the left 
        space. Fix a function $\varphi \in C^{\infty}(\R)$ such that 
        \begin{align*}
            \varphi(x) : = \begin{cases} 0 &\text{if } |x| \leq 1, \\ 
            x &\text{if } |x| \geq 2,
            \end{cases}
        \end{align*}
        and $|\varphi(x)| \leq |x|$. Choose $\{u_j\}_{j=1}^{\infty} \subset C^{\infty}(\Omega)$ so that $u_j \rightarrow u$ in ${^{-}}{W}{^{\alpha,p}}(\Omega)$ and define the sequence $u^{n}_{j} : = (1/n) \varphi(nu_j)$. We can show that $u^{n}_{j} \rightarrow u^n$ in $L^{p}((0,1))$. Moreover, using the chain rule \eqref{chain_rule_weak} we get
        \begin{align*}
            &\|{^{-}}{\mathcal{D}}{^{\alpha}} u^{n}_{j}\|_{L^{p}((0,1))}
            = \dfrac{1}{n} \left\| {^{-}}{\mathcal{D}}{^{\alpha}} \varphi(nu_j) \right\|_{L^{p}((0,1))} \\ 
            &\qquad = \dfrac{1}{n} \left\| \dfrac{\varphi(nu_j)}{nu_j} {^{-}}{\mathcal{D}}{^{\alpha}} nu_j + {^{-}}{R}{^{\alpha}_{0}} \left(nu_j , \dfrac{\varphi(nu_j)}{nu_j}\right) \right\|_{L^{p}((0,1))}\\
            &\qquad \leq \dfrac{1}{n} \left\|\dfrac{\varphi(nu_j)}{nu_j} {^{-}}{\mathcal{D}}{^{\alpha}} nu_j \right\|_{L^{p}((0,1))}^{p} + \dfrac{1}{n} \left\| {^{-}}{R}{^{\alpha}_{0}} \left(nu_j , \dfrac{\varphi(nu_j)}{nu_j}\right) \right\|_{L^{p}((0,1))}\\
            &\qquad \leq \|{^{-}}{\mathcal{D}}{^{\alpha}} u_j \|_{L^{p}((0,1))}^{p} + \dfrac{1}{n} \left\| {^{-}}{R}{^{\alpha}_{0}} \left(nu_j , \dfrac{\varphi(nu_j)}{nu_j}\right) \right\|_{L^{p}((0,1))},
        \end{align*}
        where 
        \begin{align*}
            &  \left\| {^{-}}{R}{^{\alpha}_{0}} \left(nu_j , \dfrac{\varphi(nu_j)}{nu_j}\right) \right\|_{L^{p}((0,1))}^{p} \\
             &=  \int_{0}^{1} \left| \int_{0}^{x} \dfrac{nu_j(y)}{(x-y)^{1+\alpha}} \left( \dfrac{\varphi(nu_j)(x)}{nu_j(x)} - \dfrac{\varphi(nu_j)(y)}{nu_j(y)} \right)\,dy \right|^{p}\,dx\\
            &\quad \leq  \int_{0}^{1} \left( \int_{0}^{x} \dfrac{n|u_j(y)|}{(x-y)^{1+\alpha}} \left| \dfrac{\varphi(nu_j)(x)}{nu_j(x)} - \dfrac{\varphi(nu_j)(y)}{nu_j(y)}\right|\,dy  \right)^{p}\,dx\\
            &\quad \leq  \int_{0}^{1} \left( \int_{0}^{x} \dfrac{|u_j(y)|}{(x-y)^{1+\alpha}} \dfrac{2n|u_j(x) - u_j(y)|}{|u_j(y)|}\,dy \right)^{p}\,dx \\
            & \quad \leq 2^{p}n^{p} \int_{0}^{1} \left( \int_{0}^{x} \dfrac{|u_j(x) - u_j(y)|}{(x-y)^{1+\alpha}} \,dy \right)^{p}\,dx  \\
            &\quad \leq 2^p n^p \int_{0}^{1} \left( \int_{0}^{x} \dfrac{dy}{(x-y)^{\alpha}} \right)^{p}\,dx \leq 2^{p} n^{p}. 
        \end{align*}
        Hence, $\{u^{n}_{j}\}_{j=1}^{\infty}$ is a bounded sequence in ${^{-}}{W}{^{\alpha,p}}((0,1))$. Thus there exists $v^n \in L^{p}((0,1))$ so that ${^{-}}{\mathcal{D}}{^{\alpha}} u^n_j \rightharpoonup v^n$ in $L^{p}((0,1))$ as 
        $j\to \infty$. It can easily be shown using the weak derivative definition that $v^{n} = {^{-}}{\mathcal{D}}{^{\alpha}}u^{n}$. Hence $\{u^n\}_{n=1}^{\infty}$ belongs to ${^{-}}{W}{^{\alpha,p}}((0,1))$. On the other hand, since ${^{-}}{T}u = 0 $, 
        then $u_n \in {^{-}}{C}{^{\infty}_{0}}((0,1))$. Finally, it is a consequence of Lebesgue Dominated Convergence theorem that $u^n \rightarrow u$ in ${^{-}}{W}{^{\alpha,p}}((0,1))$. Thus, $u \in {^{-}}{\overline{W}}{^{\alpha,p}_{0}}((0,1))$. The proof is complete. 
    \end{proof}

    \begin{theorem}
          Let $\Omega=(a,b)$, $0 < \alpha <1$ and $1 < p <\infty$. Suppose that  $\alpha p >1$. Then $\overline{W}^{\alpha ,p}_{0}(\Omega)= W^{\alpha,p}_{0}(\Omega)$. 
    \end{theorem}
    
    \begin{proof}
    The same construction and proof used for the one-sided closure spaces in Theorem \ref{ZeroTrace} can be used for the symmetric result $W^{\alpha,p}_{0} = \overline{W}^{\alpha,p}_{0}$.
    \end{proof}
    
      At this point, we have gathered sufficient tools to prove a crucial characterization 
     result and a pair of integration by parts 
     formula for functions in the symmetric fractional Sobolev spaces $W^{\alpha,p}(\Omega)$. Similar integration by parts formula in ${^{\pm}}{W}{^{\alpha,p}}(\Omega)$ will be presented in a subsequent section. 
     
     \begin{proposition}\label{ConstantZero}
     	Let $\Ome=(a,b)$. If $u \in \leftidx{}{W}{^{\alpha,p}_{0}}(\Omega)$, then $\leftidx{^{+}}{T}\leftidx{^{-}}{I}{^{\alpha}} u = \leftidx{^{-}}{T} \leftidx{^{+}}{I}{^{\alpha}} u = 0$. That is, $c^{1-\alpha}_{+} = c^{1-\alpha}_{-} = 0.$
     \end{proposition}
     
     \begin{proof}
     	Let $u \in W^{\alpha,p}_{0}(\Omega)$. It follows that $u \in C(\overline{\Omega})$.  
     	Then we have 
     	\begin{align*}
     	\lim_{x \rightarrow a} \left|\int_{a}^{x} \dfrac{u(y)}{(x-y)^{1- \alpha}} \,dy \right| &= \lim_{x \rightarrow a} \left| \int_{0}^{x-a} \dfrac{u(x-z)}{z^{1-\alpha}}\,dz \right|\\
     	&\leq \lim_{x \rightarrow a} \|u\|_{L^{\infty}(\Omega)} \int_{0}^{x-a} \dfrac{dz}{z^{1-\alpha}} \\ 
     	&= \lim_{x\rightarrow a}\|u\|_{L^{\infty}(\Omega)} (x-a)^{\alpha} 
     	= 0.
     	\end{align*}
     	The other trace follows similarly. 
     \end{proof}
    
    \begin{proposition}
        Let $\Omega \subset \R$, $\alpha >0$ and $1 \leq p, q < \infty$. Suppose that $\alpha p> 1$ and $\alpha q> 1$. Then for any $u \in W^{\alpha,p}(\Omega)$ and $v \in W^{\alpha ,q}(\Omega)$,
        there holds the following identity: 
        \begin{align}\label{SymmetricSobolevIBP}
            \int_{\Omega} u {^{\pm}}{\mathcal{D}}{^{\alpha}} v\, dx = (-1)^{[\alpha]} \int_{\Omega} v {^{\mp}}{\mathcal{D}}{^{\alpha}} u\, dx.
        \end{align}
    \end{proposition}
    
    \begin{proof}
        We only give a proof for $0 < \alpha <1$ because the other cases follow similarly. By Theorem \ref{H=W} and Theorem \ref{CompactEmbedding}, there exist  $\{u_j\}_{j=1}^{\infty} \subset C^{\infty}(\Omega) \cap C(\overline{\Omega})$ and $\{v_k\}_{k=1}^{\infty}\subset C^{\infty}(\Omega) \cap C(\overline{\Omega})$ such that $u_j\to u$ in $W^{\alpha,p}(\Omega)$ and $v_k\to v$
        in $W^{\alpha,q}(\Omega)$. It follows by the classical integration by parts that 
        \begin{align*}
            \int_{\Omega} u {^{\mp}}{\mathcal{D}}{^{\alpha}} v \,dx= \lim_{j,k \rightarrow \infty} \int_{\Omega} u_j {^{\mp}}{D}{^{\alpha}} v_k\, dx = \lim_{j,k \rightarrow \infty} \int_{\Omega} v_k {^{\pm}}{D}{^{\alpha}} u_j \, dx
            = \int_{\Omega} v {^{\pm}}{\mathcal{D}}{^{\alpha}} u\, dx. 
        \end{align*}
        This completes the proof.
    \end{proof}
    
    \begin{remark}
        We have used the fact that $u$ and $v$ continue to the boundary of $\Omega$ in order to apply the classical integration by parts formula. Due to the inability to guarantee this for functions 
        in the one-sided spaces ${^{\pm}}{W}{^{\alpha,p}}(\Omega)$, we postpone 
        presenting a similar result in those spaces to Section \ref{sec-5.8.1}. 
    \end{remark}

    
    \subsection{Sobolev and Poincar\'e Inequalities}\label{sec-5.7}
   %
   The goal of this subsection is to extend the well known Sobolev and Poincar\'e  inequalities  
   for functions in $W^{1,p}(\Omega)$ to the fractional Sobolev spaces 
   ${^{\pm}}{W}{^{\alpha,p}}(\Omega)$. We shall present the extensions separately for  
   the infinite domain $\Omega=\R$ and the finite domain $\Omega=(a,b)$ because the kernel functions 
   have a very different boundary behavior in the two cases, which in turn results in different inequalities 
   in these two cases. 
   Two tools that will play a crucial role in our analysis are the $L^{p}$ mapping properties of the fractional 
   integral operators (cf. Theorem \ref{LpMappings}) and the Fundamental Theorem of Weak Fractional Calculus (cf.  Theorem
   \ref{FTWFC} and).

   \subsubsection{\bf The Infinite Domain Case: $\Omega=\R$}\label{sec-5.7.1}
   Due to the flexibility of the choice of $0<\alpha <1$, the validity of a Sobolev inequality in the fractional order
   case has more variations depending on the range of $p$. Precisely, we have

   \begin{theorem}\label{thm_Sobolev_inq}
   	Let $0<\alpha<1$ and $1 <  p < \frac{1}{\alpha}$. Then there exists a constant $C >0$ such that for any $u\in L^1(\R) \cap {^{\pm}}{W}{^{\alpha,p}}(\R)$.
   	\begin{align}\label{SobolevInequalityR}
   	\|u\|_{L^{p^*}(\R)} \leq C \|{^{\pm}}{\mathcal{D}}{^{\alpha}} u\|_{L^{p}(\R)}, \qquad p^* : = \frac{p}{1-\alpha p}.
   	\end{align}
   	$p^*$ is called the fractional Sobolev conjugate of $p$. 
   \end{theorem}
   
   \begin{proof}
   	It follows from the density/approximation theorem that there exists a sequence $\{ u_j \}_{j =  1}^{\infty} \subset C^{\infty}_{0}(\R)$ so that $u_j\rightarrow u$ in ${^{\pm}}{W}{^{\alpha}}(\R)$. Note that by 
   	construction, we also have $u_j \rightarrow u$ in $L^{1}(\R)$. Then by the 
   	FTcFC (cf. Theorem \ref{FTFCa}) we get    
   	\begin{align*}
   	u_j(x) = {^{\pm}}{I}{^{\alpha}} {^{\pm}}{D}{^{\alpha}} u_j(x)\qquad \forall x \in \R.
   	\end{align*}
   	It follows by Theorem \ref{LpMappings} (c) that
   	\begin{align*}
   	 \|u_j\|_{L^{p*}(\R)} &=  \|{^{\pm}}{I}{^{\alpha}} {^{\pm}}{D}{^{\alpha}} u_j \|_{L^{p*}(\R)}
   	\leq  C \|{^{\pm}}{D}{^{\alpha}} u_j \|_{L^{p}(\R)}
   	< \infty.
   	\end{align*}
   	Consequently,
   	\begin{align*}
   	\|u_m - u_n \|_{L^{p*}(\R)} \leq C \| {^{\pm}}{\mathcal{D}}{^{\alpha}} u_m - {^{\pm}}{\mathcal{D}}{^{\alpha}} u_n \|_{L^{p}(\R)} \to 0 \quad\mbox{as } m,n\to \infty.
   	\end{align*}
   	Hence, $\{u_j\}_{j=1}^{\infty}$ is a Cauchy sequence in $L^{p*}(\R)$. Therefore, there exists 
   	a function  $v \in L^{p^*}(\R)$ so that $u_j \rightarrow v$ in $L^{p^*}(\R)$.
   	Recall that we also have $u_j \rightarrow u$ in $L^{p}(\R)$. Moreover, 
   	for every $\varphi \in C^{\infty}_{0}(\R)$ 
   	\begin{align*}
   	\int_{\R} v {^{\mp}}{D}{^{\alpha}}\varphi\,dx &= \lim_{j \rightarrow \infty} \int_{\R} u_j {^{\mp}}{D}{^{\alpha}} \varphi \, dx
   	= \lim_{j \rightarrow \infty} \int_{\R} {^{\pm}}{\mathcal{D}}{^{\alpha}} u_j \varphi \,dx \\
   	&= \int_{\R} {^{\pm}}{\mathcal{D}}{^{\alpha}} u \varphi \, dx
   	= \int_{\R} u {^{\mp}}{D}{^{\alpha}} \varphi\, dx.
   	\end{align*}
   	Thus, $v=u$ almost everywhere and 
   	\begin{align*}
   	\|u \|_{L^{p^*}(\R)} \leq C \| {^{\pm}}{\mathcal{D}}{^{\alpha}} u\|_{L^p(\R)}.
   	\end{align*}
   	The proof is complete. 
   \end{proof}
   
   \begin{remark}
   	By the simple scaling argument, which considers the scaled function $u_{\lambda}(x) := u(\lambda x)$ for $\lambda >0$,
   	it is easy to verify that $\alpha p<1$ is a necessary condition for the inequality to hold in general. 
   	Similarly, the Poincar\'e inequality  does not hold in general, as in the integer order case, when $\Omega=\R$. 
   \end{remark}
   
   %
   

   \subsubsection{\bf The Finite Domain Case: $\Omega=(a,b)$}\label{sec-5.7.2}
   One key difference between the infinite domain case and the finite domain case is 
   that the domain-dependent kernel functions $\kappa^{\alpha}_{-}(x):=(x-a)^{\alpha-1}$ and 
   $\kappa^{\alpha}_{+}(x):=(b-x)^{\alpha-1}$ ($0<\alpha<1$) do not vanish in the 
   latter case. Since both kernel
   functions are singular now, they must be ``removed" from any function $u\in {^{\pm}}{W}{^{\alpha,p}}(\Omega)$
   in order to obtain the desired Sobolev and Poincar\'e inequalities for $u$.

   \begin{theorem}
   	Let $0<\alpha<1$ and $1 \leq  p < \frac{1}{\alpha}$.
   	Then there exists a constant $C = C(\Omega , \alpha, p) > 0$ such that
   	\begin{align}\label{FractionalSobolevInequalityOmega}
   	\|u - c_{\pm}^{1-\alpha} \kappa_{\pm}^{\alpha} \|_{L^{r}(\Omega)} \leq \|{^{\pm}}{\mathcal{D}}{^{\alpha}} u\|_{L^p(\Omega)} \qquad \forall \, 1 \leq r \leq p^*.
   	\end{align}
   \end{theorem}
   
   \begin{proof}
   	It follows by Theorem \ref{LpMappings} and the FTwFC (cf. Theorem \ref{FTWFC}) that 
   	\begin{align*}
   	\|u - c_{\pm}^{1-\alpha} \kappa^{\alpha}_{\pm} \|_{L^{p^*}(\Omega)} = \| {^{\pm}}{I}{^{\alpha}} {^{\pm}}{\mathcal{D}}{^{\alpha}} u\|_{L^{p^*}(\Omega)}  
   	\leq C \| {^{\pm}}{\mathcal{D}}{^{\alpha}} u \|_{L^{p}(\Omega)}.
   	\end{align*}
   	Since $\Omega=(a,b)$ is finite,  the desired inequality \eqref{FractionalSobolevInequalityOmega} follows from 
   	the above inequality and an application of H\"older's inequality. The proof is complete. 
   \end{proof}
   
   \begin{remark}
   	An important consequence of the above theorem is that it illustrates the need for $u \in L^{\mu}(\Omega)$ with $\mu > p^*$
   	in the extension theorem (cf. Theorem \ref{ExteriorExtension}) because the fractional Sobolev spaces 
   	${^{\pm}}{W}{^{\alpha,p}}(\Omega)$ may not embed into $L^{\mu}(\Omega)$ for $\mu > p^*$ in general.
   \end{remark}
   
   Repeating the first part of the above proof (with slight modifications),  we can easily show the 
   	following Poincar\'e inequality in fractional order spaces ${^{\pm}}{W}{^{\alpha,p}}(\Omega)$. 
   
   \begin{theorem}\label{thm_PoincareI}
   	Fractional Poincar\'e - Let $0 < \alpha <1$ and $1 \leq p < \infty$. Then there exists a constant $C = C(\alpha, \Omega)>0$ such that
   	\begin{align}\label{FractionalPoincare}
   	\|u - c_{\pm}^{1-\alpha}\kappa_{\pm}^{\alpha}\|_{L^{p}(\Omega)} \leq C\|{^{\pm}}{\mathcal{D}}{^{\alpha}} u\|_{L^{p}(\Omega)}
   	\qquad \forall u \in {^{\pm}}{W}{^{\alpha,p}}(\Omega).
   	\end{align}
   \end{theorem}
   
   It is worth noting that no restriction on $p$ with respect to $\alpha$ is imposed in 
   Theorem \ref{thm_PoincareI} because no embedding result for fractional integrals 
   is used in the proof. The equation \eqref{FractionalPoincare} is the fractional analogue to the well known Poincar\'e inequality 
   \begin{align}\label{Poincare}
       \|u - u_{\Omega}\|_{L^{p}(\Omega)} \leq C \| \mathcal{D} u \|_{L^{p}(\Omega)} \qquad \forall u \in W^{1,p}(\Omega)
   \end{align}
   where $u_{\Omega} : = |\Omega|^{-1}\int_{\Omega} u \,dx$ (\cite{Evans}). In the space $W^{1,p}(\Omega)$, a specific kernel function (a constant, i.e., $u_\Ome$), that depends on $u$, is subtracted from the function $u$. In (\ref{FractionalPoincare}), the analogue to this constant kernel 
   function, which must be subtracted from $u$, is $c_{\pm}^{1-\alpha} \kappa_{\pm}^{\alpha}$, where the   dependence on $u$ is hidden in $c_{\pm}^{1-\alpha}$. 
   
   Moreover, to obtain a fractional analogue to the traditional Poincar\'e inequality 
   \begin{align} 
   \|u \|_{L^{p}(\Omega)} \leq C \|\mathcal{D}u\|_{L^{p}(\Omega)} \qquad \forall u \in W^{1,p}_{0}(\Omega),
   \end{align}
   we have two options. The first one is to simply impose the condition $u \in {^{\pm}}{\mathring{W}}{^{\alpha,p}}(\Omega)$ (see \eqref{mathring}).
   It is an easy corollary of Theorem \ref{thm_PoincareI} that 
   \begin{align}\label{SimplePoincare}
       \|u\|_{L^{p}(\Omega)} \leq C \| {^{\pm}}{\mathcal{D}}{^{\alpha}} u \|_{L^{p}(\Omega)} \qquad \forall u \in {^{\pm}}{\mathring{W}}{^{\alpha,p}}(\Omega).
   \end{align}
   From the perspective of Poincar\'e inequalities, this condition is 
   comparable to a mean-zero condition imposed on the Sobolev space $W^{1,p}(\Omega)$. 
   In order to establish the second set of conditions under which the estimate \eqref{SimplePoincare}
   can hold, we first need to establish the following lemma.
  
   \begin{lemma}\label{NoConstants}
   Let $\Omega = (a,b)$ and $0 < \alpha <1$.
    If $u \in {W}{^{\alpha,p}}(\Omega)$, then $c^{1-\alpha}_{+} := {^{+}}{I}{^{\alpha}} u (b) = 0$ and $c^{1-\alpha}_{-} : = {^{-}}{I}{^{\alpha}} u (a) = 0$.
    \end{lemma}

\begin{proof}
    Let $u \in {W}^{\alpha,p}(\Omega)$. It follows that $u \in C(\overline{\Omega})$. Then a quick calculation yields  
    \begin{align*}
        c^{1-\alpha}_{-} =  \lim_{x \rightarrow a} \left|\int_{a}^{x} \dfrac{u(y)}{(x-y)^{1- \alpha}} \,dy \right|
        &\leq \lim_{x\rightarrow a}\|u\|_{L^{\infty}(\Omega)} (x-a)^{\alpha} 
        = 0.
    \end{align*}
    A similar calculation can be done for $c_{+}^{1-\alpha}$. The proof is complete.
\end{proof}

Now we can formalize the desired Poincar\'e inequality.

\begin{theorem}\label{SymmetricFractionalPoincare}
    Let $\Omega \subset \R$, $0 < \alpha <1$, and $1 < p < \infty$. Then there exists a constant $C = C(\alpha, \Omega)>0$ such that
    \begin{align}\label{SymmetricPoincare}
        \|u\|_{L^{p}(\Omega)} \leq C \| {^{\pm}}{\mathcal{D}}{^{\alpha}} u \|_{L^{p}(\Omega)} \qquad \forall u \in {W}^{\alpha,p}(\Omega).
    \end{align}
\end{theorem}

\begin{proof}
    The proof follows as a direct consequence of the FTwFC (cf. Theorem \ref{FTWFC}), Lemma \ref{NoConstants}, and the stability estimate for fractional integrals.
\end{proof}

The final question that may come to mind is whether such an estimate can be established in 
the one-sided zero-trace spaces ${^{\pm}}{W}{^{\alpha,p}_{0}}(\Omega)$. It was shown in 
Section \ref{sec-5.6} that functions belonging to ${^{\pm}}{W}{^{\alpha,p}_{0}}(\Omega)$ do not 
  guarantee that $c^{1-\alpha}_{\pm} = 0$. Hence, ${^{\pm}}{W}{^{\alpha,p}_{0}}(\Omega) \not\subset  {^{\pm}}{\mathring{W}}{^{\alpha,p}}(\Omega)$ in general and such an inequality does not hold in ${^{\pm}}{W}{^{\alpha,p}_{0}}(\Omega)$ in general.

\subsection{The Dual Spaces ${^{\pm}}{W}{^{-\alpha,q}}(\Ome)$ and 
	$W^{-\alpha ,q}(\Ome)$}\label{sec-4.5a}

In this subsection we assume that $1 \leq p < \infty$ and $1 < q \leq \infty$ 
so that $1/p + 1/q =1$. 
\begin{definition}
	We denote ${^{\pm}}{W}{^{-\alpha , q}}(\Omega)$ as the dual space of ${^{\pm}}{\mathring{W}}{^{\alpha,p}_{0}}(\Omega)$ and $W^{-\alpha , q}(\Omega)$ as the dual space of $W^{\alpha,p}_{0}(\Omega)$. When $p=2$, we set  ${^{\pm}}{H}{^{-\alpha}}(\Omega):={^{\pm}}{\mathring{W}}{^{-\alpha , 2}_{0}}(\Omega)$ and $H^{-\alpha}(\Omega):=W^{-\alpha , 2}(\Omega)$. 
\end{definition} 

It is our aim to fully characterize these spaces; as is well known in the case of integer order Sobolev dual spaces, $W^{-1,q}(\Omega)$ (cf. \cite{Brezis}), in particular, for $q=2$.

We will begin with the symmetric spaces since the presentation is more natural and easily understood than that for the one-sided spaces. It is a consequence of Proposition \ref{ConstantZero} and Proposition \ref{ZeroTraceNorm} that $$W^{\alpha,p}_{0}(\Omega) \subset L^{p}(\Omega) \subset W^{-\alpha , q}(\Omega)$$ where these injections are continuous for $1 \leq p < \infty$ and dense for $1 < p < \infty$ since $W^{\alpha,p}_{0}(\Omega)$ and $L^{p}(\Omega)$ are reflexive in this range. In order to formally characterize the elements of $W^{-\alpha, q}(\Omega)$, we present the following theorem. 

\begin{theorem}
	Let $F \in W^{-\alpha ,q}(\Omega)$. Then there exists three functions, $f_0, f_1, f_2 \in L^{q}(\Omega)$ such that 
	\begin{align}\label{SymmetricDualCharacterization}
	\langle F , u \rangle  = \int_{\Omega} f_0 u\, dx + \int_{\Omega} f_1 {^{-}}{\mathcal{D}}{^{\alpha}} u \,dx + \int_{\Omega} f_2 {^{+}}{\mathcal{D}}{^{\alpha}} u\, dx \qquad \forall \, u \in W^{\alpha , p }_{0}(\Omega)
	\end{align}
	and 
	\begin{align}\label{DualNorm}
	\|F\|_{W^{-\alpha, q}(\Omega)} = \max \Bigl\{ \|f_0\|_{L^{q}(\Omega)},\|f_1\|_{L^{q}(\Omega)}, \|f_2\|_{L^{q}(\Omega)} \Bigr\}.
	\end{align}
	When $\Omega \subset \R$ bounded, we can take $f_0 = 0$.
\end{theorem}

\begin{proof}
	Consider the product space $E = L^{p}(\Omega) \times L^{p}(\Omega) \times L^{p}(\Omega)$ equipped with the norm 
	\[
	\|h\|_{E} = \|h_0\|_{L^{p}(\Omega)} + \|h_1\|_{L^{p}(\Omega)} + \| h_2\|_{L^{p}(\Omega)},
	\]
	where $h = [ h_0 , h_1 , h_2]$. The map $T: W^{\alpha , p}_{0}(\Omega) \rightarrow E$ 
	defined by 
	\[
	T(u) = [u, {^{-}}{\mathcal{D}}{^{\alpha}} u , {^{+}}{\mathcal{D}}{^{\alpha}} u]
	\]
	is an isometry from $W^{\alpha , p }_{0}(\Omega)$ into $E$. Given the space 
	$(G, \|\cdot\|_{E})$ be the image of $W^{\alpha,p}_{0}$ under $T$ 
	($G= T(W^{\alpha,p}_{0}(\Omega)$)) and $T^{-1} : G \rightarrow W^{\alpha,p}_{0}(\Omega)$. 
	Let $F \in W^{-\alpha ,q}(\Omega)$ be a continuous linear functional on $G$ defined by 
	$F(h) = \langle F, T^{-1} h\rangle$. By the Hahn-Banach theorem, it can be extended to 
	a continuous linear functional $S$ on all of $E$ with $\|S\|_{E^{*}} = \|F\|$. 
	By the Riesz representation theorem, we know that there exists three functions 
	$f_0 , f_1 , f_2 \in L^{q}(\Omega)$ such that 
	\begin{align*}
	\langle S, h\rangle = \int_{\Omega} f_0 h_0 \, dx + \int_{\Omega} f_1 h_1\, dx + \int_{\Omega} f_2 h_2\, dx \qquad \forall \, h = [h_0 , h_1 , h_2] \in E.
	\end{align*}
	Moreover, we have 
	\begin{align*}
	\dfrac{|\langle S , h \rangle |}{\|h\|_{E}} &= \dfrac{1}{\|h\|_{E}}\left| \int_{\Omega} f_0 h_0\, dx + \int_{\Omega} f_1 h_1 \, dx + \int_{\Omega} f_2 h_2 \, dx \right| \\ 
	&\leq \dfrac{1}{\|h\|_{E}} \Bigl( \|f_{0}\|_{L^{q}(\Omega)} \| h_0 \|_{L^{p}(\Omega)} + \|f_{1}\|_{L^{q}(\Omega)} \|h_1\|_{L^{p}(\Omega)} + \|f_2\|_{L^{q}(\Omega)} \|h_2\|_{L^{p}(\Omega)} \Bigr)\\
	&\leq \max \Bigl\{ \|f_0\|_{L^{p}(\Omega)}, \|f_1\|_{L^{p}(\Omega)} , \|f_2\|_{L^{p}(\Omega)} \Bigr\}.
	\end{align*} 
	Upon taking the supremum, we are left with 
	\[ 
	\|S\|_{E^{*}} = \max \Bigl\{ \|f_0\|_{L^{p}(\Omega)}, \|f_1\|_{L^{p}(\Omega)} ,
	\|f_2\|_{L^{p}(\Omega)} \Bigr\}.
	\]
	Furthermore, we have 
	\begin{align*}
	\langle S , Tu \rangle = \langle F, u \rangle = \int_{\Omega} f_0 u\, dx + \int_{\Omega} f_1 {^{-}}{\mathcal{D}}{^{\alpha}} u \, dx + \int_{\Omega} f_2 {^{+}}{\mathcal{D}}{^{\alpha}} u\, dx \quad \forall\, u \in W^{\alpha, p}_{0}(\Omega).
	\end{align*}
	
	When $\Omega$ is bounded, recall that $\|u \|_{W^{\alpha,p}_{0}} = \bigl(\|{^{-}}{\mathcal{D}}{^{\alpha}} u \|_{L^{p}(\Omega)}^{p} + \| {^{+}}{\mathcal{D}}{^{\alpha}} u \|_{L^{p}(\Omega)}^{p} \bigr)^{1/p}$. 
	Then we can repeat the same argument with $E = L^{p}(\Omega) \times L^{p}(\Omega)$ and $T(u) = [ {^{-}}{\mathcal{D}}{^{\alpha}} u , {^{+}}{\mathcal{D}}{^{\alpha}} u ]$. The proof is complete.
\end{proof}

\begin{remark}(a) The functions $f_0 , f_1 , f_2$ are not uniquely determined by $F$. 
	
	(b) We write $F = f_0 + {^{+}}{\mathcal{D}}{^{\alpha}} f_1 + {^{-}}{\mathcal{D}}{^{\alpha}} f_2$ whenever (\ref{SymmetricDualCharacterization}) holds. Formally, this is a consequence of integration by parts in the right hand side of (\ref{SymmetricDualCharacterization}).
	
	(c) The first assertion of Proposition \ref{SymmetricDualCharacterization} also holds for continuous linear functionals on $W^{\alpha ,p}(\Omega)$ ($1 \leq p < \infty)$. That is, for every $F \in (W^{\alpha,p}(\Omega))^{*}$, 
	\[
	\langle F , u \rangle  = \int_{\Omega} f_0 u \,dx + f_1 {^{-}}{\mathcal{D}}{^{\alpha}} u \, dx
	+ f_2 {^{+}}{\mathcal{D}}{^{\alpha}} u\, dx \qquad \forall\, u \in W^{\alpha ,p}(\Omega)
	\]
	for some functions $f_0 , f_1 , f_2 \in L^{q}(\Omega)$. 
\end{remark}

Of course, the above results also hold for functions in $H^{-\alpha}(\Omega)$. However, 
in this case, the use of the inner product and Hilbert space structure allows for improved presentation and richer characterization. We state them in the following proposition. 

\begin{proposition}\label{SymmetricDualCharacterization1}
	Let $F \in H^{-\alpha}(\Omega)$. Then \begin{align}\label{DualNorm1}
	\|F\|_{H^{-\alpha}(\Omega)} = \inf \left\{ \left(\int_{\Omega} \sum_{i=0}^{2} |f_i|^{2} \,dx \right)^{\frac12};\, f_0 , f_1 , f_2 \in L^2(\Omega) \text{ satisfy } (\ref{SymmetricDualCharacterization}) \right\}.
	\end{align}
\end{proposition}

\begin{proof}
	We begin with an altered proof of (\ref{SymmetricDualCharacterization}) for the special case $p=2$. Not only is the proof illustrative, but we will also refer to components of it to prove necessary assertions of this proposition. 
	
	For any $u , v \in H^{\alpha}_{0}(\Omega)$, we define the inner product 
	\[
	(u,v) = \int_{\Omega} \bigl( uv + {^{-}}{\mathcal{D}}{^{\alpha}} u {^{-}}{\mathcal{D}}{^{\alpha}} v + {^{+}}{\mathcal{D}}{^{\alpha}} u {^{+}}{\mathcal{D}}{^{\alpha}} v \bigr)\,dx.
	\]
	Given $F \in H^{-\alpha}(\Omega)$, it follows from Riesz Representation theorem 
	that there exists a \textit{unique} $u \in H^{\alpha}_{0}(\Omega)$ so that 
	$\langle F , v \rangle = (u,v)$ for all $v \in H^{\alpha}_{0}(\Omega)$; that is 
	\begin{align}\label{Riesz1}
	\langle F , v \rangle =\int_{\Omega} \bigl( uv + {^{-}}{\mathcal{D}}{^{\alpha}} u {^{-}}{\mathcal{D}}{^{\alpha}} v + {^{+}}{\mathcal{D}}{^{\alpha}} u {^{+}}{\mathcal{D}}{^{\alpha}} v \bigr) \,dx \qquad \forall\, v \in H^{\alpha}_{0}(\Omega).
	\end{align}
	Taking 
	\begin{align}\label{Riesz2}
	f_0 = u,\qquad f_1 = {^{-}}{\mathcal{D}}{^{\alpha}} u,\qquad
	f_2 = {^{+}}{\mathcal{D}}{^{\alpha}} u,
	\end{align}
	then (\ref{SymmetricDualCharacterization}) holds. 
	
	It follows by (\ref{SymmetricDualCharacterization}) that there exists $g_0 , g_1 , g_2 \in L^{2}(\Omega)$ so that 
	\begin{align}\label{Riesz3}
	\langle F , v \rangle = \int_{\Omega}  \bigl(g_0 v + g_1 {^{-}}{\mathcal{D}}{^{\alpha}} v + g_2 {^{+}}{\mathcal{D}}{^{\alpha}} v \bigr)\, dx \qquad\forall v \in H^{\alpha}_{0}(\Omega).
	\end{align}
	Thus, taking $v = u$ in (\ref{Riesz1}) and combing that with (\ref{Riesz2}) and (\ref{Riesz3})
	yield 
	\begin{align*}
	\int_{\Omega} f_0^2 + f_1^2 + f_2^2 &= \int_{\Omega} \bigl(u^2 + ({^{-}}{\mathcal{D}}{^{\alpha}} u)^{2} + ( {^{+}}{\mathcal{D}}{^{\alpha}} u)^{2}\bigr)\,dx \\
	&= \int_{\Omega} \bigl( g_0 u + g_1 {^{-}}{\mathcal{D}}{^{\alpha}} u + g_2 {^{+}}{\mathcal{D}}{^{\alpha}} u \bigr) \, dx 
	\leq \int_{\Omega} \bigl( g_0^2 + g_1 ^2 + g_2^2\bigr)\, dx.
	\end{align*}
	It follows from (\ref{SymmetricDualCharacterization}) and the dual norm definition that for $\|v\|_{H^{\alpha}_{0}(\Omega)} \leq 1$, 
	\begin{align*}
	\|F\|_{H^{-\alpha}(\Omega)} \leq \left(\int_{\Omega} \bigl(f_0^2 + f_1^2 + f_2^2 \bigr)\, dx \right)^{\frac12}.
	\end{align*}
	Setting $v = u/\|u\|_{H^{\alpha}_{0}(\Omega)}$ in (\ref{Riesz1}), we deduce that 
	\begin{align*}
	\|F\|_{H^{-\alpha}(\Omega)} = \left( \int_{\Omega}  \bigl(f_0^2 + f_1^2 + f_2^2 \bigr)\, dx \right)^{\frac12}.
	\end{align*}
	Therefore, (\ref{DualNorm1}) must hold. The proof is complete.
\end{proof}

\begin{remark}
	Similar to the integer order case, we define the action of $v \in L^{2}(\Omega) \subset H^{-\alpha}(\Omega)$ on any $u \in H^{\alpha}_{0}(\Omega)$ by
	\begin{align}
	\langle v,u \rangle = \int_{\Omega} vu\,dx.
	\end{align}
	That is to say that given $v \in L^{2}(\Omega) \subset H^{-\alpha}(\Omega)$, we associate it with the bounded linear functional $v: H^{\alpha}_{0}(\Omega) \rightarrow \R$ defined by $\langle v , u \rangle = v(u) = \int_{\Omega} vu$. It is easy to check that this mapping is in fact continuous/bounded on $H^{\alpha}_{0}(\Omega)$. 
\end{remark}

\medskip
Now we turn our attention to dual spaces of one-sided Sobolev spaces. The situation in this case is more complicated. This is due to the fact that there are several variations of the parent spaces ${^{\pm}}{W}{^{\alpha,p}}$ of which we might consider. To be specific, we consider 
a space $W$ where 
\[
W \in \bigl\{ {^{\pm}}{W}{^{\alpha,p}}, {^{\pm}}{W}{^{\alpha,p}_{0}}, {^{\pm}}{\mathring{W}}{^{\alpha,p}}, {^{\pm}}{\mathring{W}}{^{\alpha,p}_{0}} \bigr\}.
\]
Thus, we want to know which of these spaces produces a dual space that can be characterized 
in similar fashion as for the symmetric space $W^{\alpha,p}_{0}$.

To answer this question, we first proposed that in order to prove a rich characterization of dual spaces, we must first have the continuous and dense inclusion $W \subset L^{p} \subset W^{*}$ for appropriate ranges of $p$. Effectively, it is necessary to have the inequality $\|u \|
_{L^{p}} \leq C \|u\|_{W}$ for every $u \in W$. More or less, this question is informed by the existence of fractional Poincar\'e inequalities in $W$. It is known that in general, $\|u \|_{L^{p}(\Omega)} \not\leq C \| {^{\pm}}{\mathcal{D}}{^{\alpha}} u \|_{L^{p}(\Omega)}$ 
for every $u \in {^{\pm}}{W}{^{\alpha,p}}(\Omega)$ and $u\in {^{\pm}}{W}{^{\alpha,p}_{0}}$, 
and note  ${^{\pm}}{W}{^{\alpha,p}_{0}}(\Omega) \not\hookrightarrow {^{\pm}}{\mathring{W}}{^{\alpha,p}}(\Omega)$). For these reasons, we are left to characterize the dual space ${^{\pm}}{W}{^{-\alpha, q}}(\Omega) : = ({^{\pm}}{\mathring{W}}{^{\alpha,p}}(\Omega))^{*}$.

It is easy to see that there holds 
\begin{align}
{^{\pm}}{\mathring{W}}{^{\alpha,p}}(\Omega) \subset L^{p}(\Omega) \subset {^{\pm}}{W}{^{-\alpha,q}}(\Omega),
\end{align}
where the injections are continuous for $1 \leq p < \infty$ and dense for $1 < p < \infty$ since ${^{\pm}}{\mathring{W}}{^{\alpha,p}}(\Omega)$ is reflexive in this range. 

Now we are well equipped to characterize ${^{\pm}}{W}{^{-\alpha,q}}(\Omega)$. For brevity, we will state the results and omit the proofs since each of them can be done using the same techniques as used in the symmetric case for the spaces $W^{-\alpha,q}(\Ome)$ and $H^{-\alpha}(\Ome)$. 

\begin{theorem}
	Let $F \in {^{\pm}}{W}{^{-\alpha,q}}(\Omega)$. Then there exists two functions, $f_0 , f_1 \in L^{q}(\Omega)$ such that 
	\begin{align}\label{DualCharacterization}
	\langle F, u \rangle _{\pm} = \int_{\Omega} \bigl( f_0 u + f_1 {^{\pm}}{\mathcal{D}}{^{\alpha}} u \bigr)\, dx \qquad \forall \, u \in {^{\pm}}{\mathring{W}}{^{\alpha,p}}(\Omega)
	\end{align}
	and 
	\begin{align}
	\|F\|_{{^{\pm}}{W}{^{-\alpha , q}}(\Omega)} = \max \Bigl\{ \|f_0\|_{L^{q}(\Omega)} , \|f_1\|_{L^{q}(\Omega)} \Bigr\}.
	\end{align}
\end{theorem}

\begin{proposition}
	Let $F \in {^{\pm}}{H}{^{-\alpha}}(\Omega)$. Then 
	\begin{align}
	\|F\|_{{^{\pm}}{H}{^{-\alpha}}(\Omega)} = \inf \left\{ \left( \int_{\Omega} \bigl(f_0^2 + f_1^2 \bigr)\, dx  \right)^{\frac12};\, f_0 , f_1 \in L^{2}(\Omega)\text{ satisfying } (\ref{DualCharacterization}) \right\}. 
	\end{align}
\end{proposition}

\begin{remark}
	Similar to the symmetric case, we define the action of $v \in L^{2}(\Omega) \subset {^{\pm}}{H}{^{-\alpha}}(\Omega)$ on any $u \in {^{\pm}}{\mathring{H}}{^{\alpha}}(\Omega)$ by 
	\begin{align}
	\langle v , u \rangle = \int_{\Omega} vu \,dx.
	\end{align}
\end{remark}
        
    \subsection{Relationships Between Fractional Sobolev Spaces}\label{sec-5.8}
    In this subsection we establish a few connections between the newly defined fractional Sobolev
    spaces ${^{\pm}}{W}{^{\alpha,p}}(\Omega)$ and ${W}^{\alpha,p}(\Omega)$ with some existing fractional Sobolev spaces recalled in Section \ref{sec-5.1}. Before doing that, we first address the issues of their consistency over subdomains, inclusivity across orders of differentiability, and their consistency with the existing integer order Sobolev spaces.
    
    \begin{proposition}
        Let $\Omega = (a,b)$, $0 < \alpha < \beta < 1$ and $1 \leq p < \infty$. If $u \in {^{\pm}}{W}{^{\beta,p}}(\Omega)$, then $ u\in {^{\pm}}{W}{^{\alpha,p}}(\Omega)$. 
    \end{proposition}
    
    \begin{proof}
        By Theorem \ref{FTWFC},
            $u = c^{1-\beta}_{\pm} \kappa^{\beta}_{\pm} + {^{\pm}}{I}{^{\beta}} {^{\pm}}{\mathcal{D}}{^{\beta}}u$
        and by Proposition \ref{properties} ${^{\pm}}{\mathcal{D}}{^{\alpha}}u$ exists and is given by 
        \begin{align*}
            {^{\pm}}{\mathcal{D}}{^{\alpha}} u &= c^{1-\beta}_{\pm} \kappa^{\beta - \alpha}_{\pm} + {^{\pm}}{I}{^{\beta - \alpha}} {^{\pm}}{\mathcal{D}}{^{\beta}} u  \\
            &= c^{1-\beta}_{\pm} \kappa^{\beta}_{\pm} \kappa^{-\alpha}_{\pm} + {^{\pm}}{I}{^{\beta - \alpha}}{^{\pm}}{\mathcal{D}}{^{\beta}} u \\ 
            &= ( u - {^{\pm}}{I}{^{\beta}} {^{\pm}}{\mathcal{D}}{^{\beta}} u ) \kappa^{-\alpha}_{\pm} + {^{\pm}}{I}{^{\beta - \alpha}} {^{\pm}}{\mathcal{D}}{^{\beta}} u.
        \end{align*}
        It follows by direct estimates that there exists $C  = C(\Omega , \alpha ,\beta , p)$ so that 
        \begin{align*}
            \|{^{\pm}}{\mathcal{D}}{^{\alpha}} u \|_{L^{p}(\Omega)} \leq C \| u \|_{{^{\pm}}{W}{^{\beta,p}}(\Omega)}. 
        \end{align*}
        The proof is complete. 
    \end{proof}
    
    \begin{remark}
        This inclusivity property is trivial in the integer order Sobolev spaces, but may not be so in 
        fractional Sobolev spaces, which may be a reason why it has not been discussed in the literature.  
        However, in our case, the proof is not difficult thanks to the FTwFC. 
    \end{remark}
    
    Unlike the integer order case, the consistency on subdomains is more difficult to establish 
    in the spaces ${^{\pm}}{W}{^{\alpha,p}}$. The following proposition and its accompanying proof 
    provide further insight to the effect of domain-dependent derivatives and their associated kernel functions.
    
    \begin{proposition}
        Let $\Omega =(a,b)$, $\alpha > 0$, $1 < p < \infty$, $\mu > p(1- \alpha p)^{-1}$.
        Suppose that $u \in {^{\pm}}{W}{^{\alpha,p}}(\Omega) \cap L^{\mu}(\Omega)$. Then for any $\Omega'=(c,d) \subset \Omega$, $u \in {^{\pm}}{W}{^{\alpha,p}}(\Omega')$.
    \end{proposition}
    
    \begin{proof}
         Since $(c,d) \subset (a,b)$, it is easy to see that $\|u \|_{L^{p}((c,d))} \leq \|u \|_{L^{p}((a,b))}$. Thus we only need to show that $u$ has a weak derivative on $(c,d)$ that belongs to $L^{p}((c,d))$. 
         
         Choose $\{u_j\}_{j=1}^{\infty} \subset C^{\infty}((a,b))$ so that $u_j \rightarrow u$ in ${^{\pm}}{W}{^{\alpha,p}}((a,b))$. It follows that $u_j \in C^{\infty}([c,d])$ and for any $\varphi \in C^{\infty}_{0}((c,d))$ there holds for the left derivative
        \begin{align*}
            \int_{c}^{d} u {^{+}}{D}{^{\alpha}} \varphi\, dx = \lim_{ j \rightarrow \infty} \int_{c}^{d} u_j {_{x}}{D}{^{\alpha}_{d}}\varphi\, dx 
            = \lim_{j \rightarrow \infty} \int_{c}^{d} {_{c}}{D}{^{\alpha}_{x}} u_j \varphi\, dx.
        \end{align*}
        Then we want to show that there exists $v \in L^{p}((c,d))$ such that 
        \begin{align*}
            \lim_{j \rightarrow \infty} \int_{c}^{d} {_{c}}{D}{^{\alpha}_{x}} u_j \varphi\, dx = \int_{c}^{d} v \varphi\, dx. 
        \end{align*}
        Note that 
        \begin{align*}
             {_{c}}{D}{^{\alpha}_{x}}u_j(x) = {_{a}}{D}{^{\alpha}_{x}} u_j(x) - {_{a}}{D}{^{\alpha}_{c}} u_j(x).
        \end{align*}
        Using this decomposition, we get
        \begin{align*}
            &\|{_{c}}{D}{^{\alpha}_{x}}u_j \|_{L^{p}((c,d))}^{p} = \| {_{a}}{D}{^{\alpha}_{x}} u_j - {_{a}}{D}{^{\alpha}_{c}} u_j\|_{L^{p}((c,d))}^{p}\\
            &\quad\leq \| {_{a}}{D}{^{\alpha}_{x}} u_j\|_{L^{p}((a,b))}^{p} + \| {_{a}}{D}{^{\alpha}_{c}} u_j\|_{L^{p}((c,d))}^{p} \\ 
            &\quad\leq \| {_{a}}{D}{^{\alpha}_{x}} u_j\|_{L^{p}((a,b))}^{p} + \int_{c}^{d} \left|\int_{a}^{c} \dfrac{u_j(y)}{(x-y)^{1+\alpha}}\,dy \right|^{p}\,dx \\ 
            &\quad\leq \| {_{a}}{D}{^{\alpha}_{x}} u_j\|_{L^{p}((a,b))}^{p}+  \|u_j\|_{L^{\mu}((a,c))}^{p} \int_{c}^{d} \left(\int_{a}^{c} \dfrac{dy}{(x-y)^{\nu (1+\alpha)}}\right)^{\frac{p}{\nu}} \,dx \\ 
            &\quad= \| {_{a}}{D}{^{\alpha}_{x}} u_j\|_{L^{p}((a,b))}^{p}+  \|u_j\|_{L^{\mu}((a,c))}^{p} \int_{c}^{d} \Bigl( (x-a)^{1 - \nu(1+\alpha)} - (x-c)^{1 - \nu(1 + \alpha)} \Bigr) dx \\ 
            &\quad\leq  \| {_{a}}{D}{^{\alpha}_{x}} u_j\|_{L^{p}((a,b))}^{p}+  \|u_j\|_{L^{\mu}((a,c))}^{p} \int_{c}^{d} (x-a)^{\frac{p}{\nu} - p (1+\alpha)} + (x-c)^{\frac{p}{\nu} -p(1+\alpha)}\,dx,
        \end{align*}
        which is bounded if and only if $\mu > p(1-\alpha p)^{-1}$. Choosing $j$ sufficiently large, we have that the sequence ${_{c}}{D}{^{\alpha}_{x}} u_j $ is bounded in $L^{p}((c,d))$. Therefore, there exists a function $v \in L^p((c,d))$ and a subsequence (still denoted by ${_{c}}{D}{^{\alpha}_{x}} u_j$) so that ${_{c}}{D}{^{\alpha}_{x}} u_j \rightharpoonup v$. It follows that 
        \begin{align*}
           \int_{c}^{d} u {^{+}}{D}{^{\alpha}}\varphi\, dx =  \lim_{j \rightarrow \infty} \int_{c}^{d} {_{c}}{D}{^{\alpha}_{x}} u_j \varphi \,d x= \int_{c}^{d} v \varphi\, dx.
        \end{align*}
        Hence $u \in {^{-}}{W}{^{\alpha,p}}((c,d))$. Similarly, we can prove
        that the conclusion also holds for the right derivative.   The proof is complete.
    \end{proof}
    
    \subsubsection{\bf Consistency with $W^{1,p}(\Omega)$}\label{sec-5.8.1}
         Our aim here is to show that there exists a consistency between our newly defined fractional Sobolev spaces 
         and the integer order Sobolev spaces. To the end,  we need to show that there is a consistency between 
         fractional order weak derivatives and integer order weak derivatives, which is detailed in the   
         lemma below. 
        
        \begin{lemma}\label{lem_consistency}
            Let $\Omega \subseteq \R$, $0 <\alpha<1$ and $1 \leq p < \infty$. Suppose $u \in W^{1,p}(\Omega)$. Then for every $\psi \in C^{\infty}_{0}(\Omega)$, ${^{\pm}}{\mathcal{D}}{^{\alpha}} \psi(u)= -{^{\pm}}{I}{^{1- \alpha}} [ \psi'(u) Du] \in L^{p}(\Omega)$. 
        \end{lemma}
    
        \begin{proof}
            Let $u \in W^{1,p}(\Omega) \cap {^{\pm}}{W}{^{\alpha,p}}(\Omega)$. By the density/approximation theorem, there exists $\{u_j\}_{j=1}^{\infty} \subset C^{\infty}(\Omega)$ 
            such that $u_j \rightarrow u$ in $W^{1,p}(\Omega)$. Then we have  
            \begin{align*}
                \int_{\Omega} \psi(u) {^{\mp}}{D}{^{\alpha}}\varphi\, dx = \lim_{j\rightarrow \infty} \int_{\Omega} \psi (u_j) {^{\mp}}{D}{^{\alpha}}\varphi \, dx
                &= \lim_{j \rightarrow \infty} (-1) \int_{\Omega} \psi' (u_j) Du_{j} {^{\mp}}{I}{^{1-\alpha}} \varphi\, dx \\ 
                &= \lim_{j \rightarrow \infty} (-1)\int_{\Omega}  {^{\pm}}{I}{^{1- \alpha}} [ \psi'(u_j) Du_j] \varphi \, dx.
            \end{align*}
           
           Next, we claim that ${^{\pm}}{I}{^{1-\alpha}} [ \psi ' (u_j) Du_j] \rightarrow {^{\pm}}{I}{^{1-\alpha}} [ \psi'(u)\mathcal{D}u]$ in $L^{p}(\Omega)$ where $\mathcal{D}$ denotes the integer weak derivative. Our claim follows because 
            \begin{align*}
                \|{^{\pm}}{I}{^{1-\alpha}} [\psi'(u_j) D u_j] - {^{\pm}}{I}{^{1-\alpha}} [ \psi'(u)\mathcal{D}u]\|_{L^{p}(\Omega)} &\leq C \| \psi'(u_j) Du_j - \psi'(u)\mathcal{D}u\|_{L^{p}(\Omega)}
            \end{align*}
            which converges to zero by the assumptions on $\psi$ and on $\{u_j\}_{j=1}^{\infty}$ and the chain rule in $W^{1,p}(\Omega)$.
            The proof is complete. 
        \end{proof}
        
        \begin{remark}
        	The identity ${^{\pm}}{\mathcal{D}}{^{\alpha}} \psi(u)= -{^{\pm}}{I}{^{1- \alpha}} [ \psi'(u) \mathcal{D}u] \in L^{p}(\Omega)$ can be regarded as a special fractional chain rule, which also explains why there is no clean fractional chain rule in general. 
        \end{remark}

        Our first consistency result will be one that allows us to make no assumption on the relationship 
        between $\alpha$ and $p$. However, a restriction on function spaces must be imposed, which will be shown 
        later to be a price to pay without imposing any restriction on the relationship between $\alpha$ and $p$.
        
        \begin{theorem}\label{TraceZeroConsistency}
             Let $\Omega \subset \R$, $0 <\alpha<1$ and $1 \leq p < \infty$. Then $W^{1,p}_{0}(\Omega) \subset {^{\pm}}{W}{^{\alpha,p}}(\Omega)$. Hence, $W^{1,p}_0(\Omega) \subset   {W}^{\alpha,p}(\Omega)$.
        \end{theorem}
    
        \begin{proof}
            Let $u \in W^{1,p}_{0}(\Omega)$. By the density/approximation theorem, there exists $\{u_j\}_{j=1}^{\infty} \subset C^{\infty}_{0}(\Omega)$ such that $u_j \rightarrow u$ in $W^{1,p}(\Omega)$. Then we have 
            \begin{align*}
                \int_{\Omega} u {^{\mp}}{D}{^{\alpha}}\varphi\, dx = \lim_{j \rightarrow \infty} \int_{\Omega} u_j {^{\mp}}{D}{^{\alpha}}\varphi \, dx
                &= \lim_{j\rightarrow \infty} (-1)\int_{\Omega}  Du_j {^{\mp}}{I}{^{1-\alpha}}\varphi \, dx \\
                &= \lim_{j \rightarrow \infty} (-1) \int_{\Omega}  {^{\pm}}{I}{^{1-\alpha}} Du_j \varphi \,dx .
            \end{align*}
            Next, by the boundedness of ${^{\pm}}{I}{^{1-\alpha}}$ we get 
            \begin{align*}
                \|{^{\pm}}{I}{^{1-\alpha}} Du_j - {^{\pm}}{I}{^{1-\alpha}} \mathcal{D}u\|_{L^{p}(\Omega)} \leq C \|Du_j - \mathcal{D}u \|_{L^{p}(\Omega)}, 
            \end{align*}
            which converges to zero by the choice of $\{u_j\}_{j=1}^{\infty}$. Setting $j\to \infty$ in the above 
            equation yields that ${^{\pm}}{\mathcal{D}}{^{\alpha}} u= - {^{\pm}}{I}{^{1-\alpha}} \mathcal{D}u$. 
            Thus, 
            \begin{align*}
                \|{^{\pm}}{\mathcal{D}}{^{\alpha}} u\|_{L^p(\Omega)} = \| {^{\pm}}{I}{^{1-\alpha}} \mathcal{D}u \|_{L^{p}(\Omega)} \leq C \|\mathcal{D}u\|_{L^p(\Omega)}<\infty.
            \end{align*}
            The proof is complete. 
        \end{proof}
        
        \begin{remark}
        	From the above proof we can see that $W^{1,p}(\R) \subset {^{\pm}}{W}{^{\alpha,p}}(\R)$
        	for all $0 <\alpha<1$ and $1 \leq p < \infty$ with the replacing of $\{u_j\}_{j=1}^{\infty} \subset C^{\infty}_{0}(\R)$.
        \end{remark}
        
        To see that the need for zero boundary traces is a necessary condition, we consider  the function 
        $u \equiv 1$. With $\Omega = (a,b)$ is a finite domain, it is easy to check that $u$ is weakly differentiable with 
        the weak derivative coinciding with the Riemann-Liouville derivative, that is,  ${^{-}}{\mathcal{D}}{^{\alpha}} 1 = \Gamma(1-\alpha)^{-1} (x-a)^{-\alpha}$ 
        and a similar formula holds for the right weak derivative. It is easy to show that $\|{^{\pm}}{\mathcal{D}}{^{\alpha}} 1 \|_{L^{p}((a,b))} < \infty$ if and only if $\alpha p <1$. 
        Therefore, the inclusion $W^{1,p}((a,b)) \subset {^{\pm}}{W}{^{\alpha,p}}((a,b))$ may not 
        hold in general. However, the next theorem shows that the inclusion does hold in general
        provided that $\alpha p <1$. 
          
        \begin{theorem}
           Let $\Omega =(a,b)$, $0 <\alpha<1$ and $1\leq p<\infty$. Suppose that $\alpha p <1$. Then $W^{1,p}(\Omega) \subset {^{\pm}}{W}{^{\alpha,p}}(\Omega)$. Hence, $W^{1,p}(\Omega) \subset   {W}{^{\alpha,p}}(\Omega)$ when $\alpha p <1$.
        \end{theorem}
        
        \begin{proof}
            We only give a proof for $W^{1,p}(\Omega) \subset {^{-}}{W}{^{\alpha,p}}(\Omega)$ because 
             the inclusion $W^{1,p}(\Omega)$ $\subset {^{+}}{W}{^{\alpha,p}}(\Omega)$ can be proved similarly.
            
            Let $u \in W^{1,p}((a,b))$. By the density/approximation theorem, there exists a sequence $\{u_j\}_{j=1}^{\infty} \subset C^{\infty}((a,b))\cap C([a,b])$ so that 
            $u_j \rightarrow u$ in $W^{1,p}((a,b))\cap C([a,b])$. 
            Then for any $\varphi \in C^{\infty}_{0}((a,b))$, using the integration by parts formula and the 
            relationship between the Riemann-Liouville and Caputo derivatives, we get 
            \begin{align*}
              \int_{a}^{b} u_j(x) {_{x}}{D}{^{\alpha}_{b}} \varphi(x)\,dx 
                &=\int_{a}^{b} {_{a}}{D}{^{\alpha}_{x}}u_j(x) \varphi (x)\,dx \\
                &= \int_{a}^{b} \left(\dfrac{u_{j}(a)  }{\Gamma(1- \alpha) (x-a)^{\alpha}} + {_{a}}{I}{^{1-\alpha}_{x}} Du_j(x)\right) \varphi(x)\,dx.
            \end{align*}
            Taking the limit $j\to \infty$ on both sides yields 
            \begin{align*}
                \int_{a}^{b} u(x) {_{x}}{D}{^{\alpha}_{b}} \varphi(x)\, dx 
                = \int_{a}^{b} \biggl( \dfrac{u(a) }{\Gamma(1-\alpha) (x-a)^{\alpha}} + {_{a}}{I}{^{1-\alpha}_{x}} \mathcal{D}u(x)\biggr) \varphi(x)\,dx.
            \end{align*}
            Hence, ${^{-}}{\mathcal{D}}{^{\alpha}} u$ almost everywhere in $(a,b)$ and is given by
            \begin{equation}\label{weak_Caputo}
            {^{-}}{\mathcal{D}}{^{\alpha}} u = \dfrac{u(a) }{\Gamma(1-\alpha) (x-a)^{\alpha}} 
            + {_{a}}{I}{^{1-\alpha}_{x}} \mathcal{D}u(x).
            \end{equation}
            It remains to verify that ${^{-}}{\mathcal{D}}{^{\alpha}} u \in L^{p}((a,b))$, which 
            can be easily done for $\alpha p<1$ using the formula above for the weak derivative and 
            the mapping properties of the fractional integral operators. The proof is complete.
        \end{proof}
    
    \begin{remark}
    \eqref{weak_Caputo} suggests the following definitions of the weak fractional Caputo derivatives for any $u\in W^{1,1}(\Omega)$:
    \begin{align}\label{weak_Caputo_derivativeL}
    {^{-}_{C}}{\mathcal{D}}{^{\alpha}} u(x) &: = {_{a}}{I}{^{1-\alpha}_{x}} \mathcal{D}u(x) \qquad\mbox{a.e. in }  \Omega, \\
    {^{+}_{C}}{\mathcal{D}}{^{\alpha}} u(x) &: =  {_{x}}{I}{^{1-\alpha}_{b}} \mathcal{D}u(x) \qquad\mbox{a.e. in }  \Omega, \label{weak_Caputo_derivativeR}
    \end{align} 
    and then we have almost everywhere in $\Omega$
    \begin{align}\label{C1}
     {^{-}_{C}}{\mathcal{D}}{^{\alpha}} u (x)&: =  {^{-}}{\mathcal{D}}{^{\alpha}} u(x)  - \dfrac{u(a)  }{\Gamma(1-\alpha) (x-a)^{\alpha}},\\
      {^{+}_{C}}{\mathcal{D}}{^{\alpha}} u(x) &: = {^{+}}{\mathcal{D}}{^{\alpha}} u(x)   - \dfrac{u(b)  }{\Gamma(1-\alpha)(b-x)^{\alpha} }.  \label{C2}
    \end{align} 
    \end{remark}
    
    We conclude this section with an integration by parts formula for functions in one-sided Sobolev spaces. The need to wait until now for such a formula will be evident in the assumptions.
    
    \begin{proposition}
        Let $\Omega \subset \R$, $\alpha >0$, $1 \leq p < \infty$. Suppose that $u \in {^{\pm}}{W}{^{\alpha,p}}(\Omega)$, $v \in W^{1,q}_{0}(\Omega)$, and $w \in W^{1,q}(\Omega)$. Then there holds 
        \begin{align}\label{Sobolev0IBP}
            \int_{\Omega} v{^{\pm}}{\mathcal{D}}{^{\alpha}} u\, dx 
             = (-1)^{[\alpha]}\int_{\Omega} u {^{\mp}}{\mathcal{D}}{^{\alpha}}v \, dx. 
        \end{align}
        Moreover,  if $\alpha q < 1$, there holds
        \begin{align}\label{SobolevIBP}
            \int_{\Omega} w{^{\pm}}{\mathcal{D}}{^{\alpha}} u \,dx = (-1)^{[\alpha]} \int_{\Omega} u {^{\mp}}{\mathcal{D}}{^{\alpha}} w\, dx.
        \end{align}
    \end{proposition}
    
    \begin{proof}
        We only give a proof for \eqref{Sobolev0IBP} when $0 < \alpha <1$. The other cases and (\ref{SobolevIBP}) can be showed similarly. 
        Choose $\{u_j\}_{j=1}^{\infty} \subset {}{C}^{\infty}(\Omega)$
        	and $\{v_k\}_{k=1}^{\infty} \subset C^{\infty}_0(\Omega)$ 
        such that $u_j\to u$ in ${^{\pm}}{W}{^{\alpha,p}}(\Omega)$ and  
          $v_k \rightarrow v$ in $W^{1,q}(\Omega)$. By Theorem \ref{TraceZeroConsistency} we have 
        $v \in {^{\mp}}{W}{^{\alpha,q}}(\Omega)$. It follows that 
        \begin{align*}
            \int_{\Omega} u {^{\mp}}{\mathcal{D}}{^{\alpha}} v\, dx = \lim_{j,k \rightarrow \infty} \int_{\Omega} u_j {^{\mp}}{D}{^{\alpha}} v_k\, dx 
            = \lim_{j,k \rightarrow \infty} \int_{\Omega} v_k {^{\pm}}{D}{^{\alpha}} u_j \, dx
            = \int_{\Omega} v {^{\pm}}{\mathcal{D}}{^{\alpha}} u_j\ dx.
        \end{align*}
        This completes the proof.
    \end{proof}

    \subsubsection{\bf The Case $p=1$ and $\Omega = \R$}\label{sec-5.8.2}
       First, by doing a change of variables we get for any $\varphi \in C^{\infty}_{0}(\R)$ 
        \begin{align*}
                {^{-}}{D}{^{\alpha}} \varphi(x) &= \dfrac{1}{\Gamma(1- \alpha)} \dfrac{d}{dx} \int_{-\infty}^{x} \dfrac{\varphi(y)}{(x-y)^{\alpha}} \,dy 
                = \dfrac{1}{\Gamma(1 - \alpha)} \dfrac{d}{dx} \int_{0}^{\infty} t^{-\alpha} \varphi(x-t)\,dt\\ 
                &= \dfrac{1}{\Gamma(1- \alpha)} \int_{0}^{\infty} t^{-\alpha} \varphi'(x-t)\,dt  
                = \dfrac{\alpha}{\Gamma(1- \alpha)} \int_{-\infty}^{x} \dfrac{\varphi(x) - \varphi(t)}{(x-t)^{1+\alpha}}\,dt.
        \end{align*}
        Similarly, 
        \begin{align*}
            {^{+}}{D}{^{\alpha}}\varphi(x) &= \dfrac{\alpha}{\Gamma(1 - \alpha)} 
            \int_{x}^{\infty} \dfrac{ \varphi(t) - \varphi(x) }{(t-x)^{1 + \alpha}}\,dt.
        \end{align*}
        These equivalent formulas will be used in the proof of the next theorem. 

        \begin{theorem}
            Let $0<\alpha<1$. Then $\widetilde{W}^{\alpha,1}(\R) \subseteq {^{\pm}}{W}{^{\alpha,1}}(\R)$.
        \end{theorem}
        
        \begin{proof}
            Let $u \in \widetilde{W}^{\alpha,1}(\R)$. Recall that $C^{\infty}_{0}(\R)$ is dense in $\widetilde{W}^{\alpha,1}(\R)$. Then 
            there exists a sequence $\{u_j\}_{j=1}^{\infty} \subset C^{\infty}_{0}(\R)$ such that $u_j \rightarrow u$ 
            in $\widetilde{W}^{\alpha,1}(\R)$ as $j \rightarrow \infty$. We only give a proof of the inclusion 
            for the left fractional Sobolev space because the proof for the other case follows similarly.  
            
            Using the above equivalent formula for left derivatives, we get 
            \begin{align*}
                \left\| {^{-}}{\mathcal{D}}{^{\alpha}}u_j \right\|_{L^{1}(\R)} &= \left\|{^{-}}{D}{^{\alpha}} u_j \right\|_{L^{1}(\R)} 
                = C_\alpha \int_{\R} \left|\int_{-\infty}^{x} \dfrac{u_j (x) - u_j(y)}{(x-y)^{1+\alpha}}\,dy \right|dx \\ 
                &\leq C_\alpha \int_{\R} \int_{-\infty}^{x} \dfrac{|u_j(x) - u_j(y)|}{|x-y|^{1+\alpha}}\,dydx \\
                &\leq C_\alpha \int_{\R} \int_{\R} \dfrac{|u_j(x) - u_j(y)|}{|x-y|^{1+\alpha}}\,dydx  \\
                &= C_\alpha \left[u_j \right]_{\widetilde{W}^{\alpha,1}(\R)}.
            \end{align*}
            By the property of $\{u_j\}_{j=1}^{\infty}$, we conclude that $\left[ u_j \right]_{\widetilde{W}^{\alpha,1}(\R)} \rightarrow [u]_{\widetilde{W}^{\alpha,1}(\R)} < \infty$. 

            Let $\eps > 0$, for sufficiently large $m,n \in \N$, we have
            \begin{align*}
                \left\|{^{\pm}}{\mathcal{D}}{^{\alpha}} u_m - {^{\pm}}{\mathcal{D}}{^{\alpha}} u_n \right\|_{L^{1}(\R)} = \left\|{^{\pm}}{\mathcal{D}}{^{\alpha}} \left[u_m - u_n \right]\right\|_{L^{1}(\R)}  
                \leq \left[u_m - u_n \right]_{\widetilde{W}^{\alpha,1}(\R)}  
                <\eps.
            \end{align*}
            Hence, $\{u_j\}_{j=1}^{\infty}$ is a Cauchy sequence in ${^{\pm}}{W}{^{\alpha,1}}(\R)$. Thus, 
            there exists $v \in {^{\pm}}{W}{^{\alpha,1}}(\R)$ such that $u_j \rightarrow v$ 
            in ${^{\pm}}{W}{^{\alpha,1}}(\R)$. By the property of $\{u_j\}_{j=1}^{\infty}$, there holds 
            $u_j \rightarrow u$ in $L^{1}(\R)$. On the other hand, the convergence in ${^{\pm}}{W}{^{\alpha,1}}(\R)$ implies that $u_j \rightarrow v$ in $L^{1}(\R)$.
            Thus, $u=v$ almost everywhere in $\R$ and yielding that $u \in {^{\pm}}{W}{^{\alpha,1}}(\R)$. 
        \end{proof}

    \subsubsection{\bf The Case $p=2$ and $\Omega = \R$}\label{sec-5.8.3}
        This section extends the above equivalence result of two fractional Sobolev spaces to the case when $p=2$. 
        As we will see, $p=2$ is special in the sense that it is the only case in which the equivalence of the space $\widehat{H}^{\alpha}(\R)$ defined by the Fourier transform (and its inverse) and the space $\widetilde{H}^{\alpha}(\R)$ holds. Recall that $\widehat{W}^{\alpha,p}(\R) \neq \widetilde{W}^{\alpha,p}(\R)$ for $p\neq2$ (cf. \cite{Adams,Brezis}).

        \begin{lemma}\label{equivalence_seminorms}
           Let $0 < \alpha <1$ and $\varphi \in C^{\infty}_{0}(\R)$. Then $\|{^{\mathcal{F}}}{D}{^{\alpha}} \varphi \|_{L^{2}(\R)} \cong [ \varphi ]_{\widetilde{H}^{\alpha}(\R)}$.
        \end{lemma}
        
        \begin{proof}
            Let $\hat{\varphi}=\mathcal{F}(\varphi)$. It follows from Plancherel theorem and (\ref{SeminormRelation}) that
            \begin{align*}
                \|{^{\mathcal{F}}}{D}{^{\alpha}} \varphi \|_{L^{2}(\R)}^{2} &= \|\mathcal{F}^{-1} [ (i\xi)^{\alpha} \hat{\varphi} ] \|_{L^{2}(\R)}^{2} 
                = \|(i\xi)^{\alpha} \hat{\varphi} \|_{L^{2}(\R)}^{2}  \\ 
                & = \int_{\R} |i\xi|^{2\alpha} |\hat{\varphi}(\xi)|^{2} \,d\xi 
                = \int_{\R} |\xi|^{2\alpha} |\hat{\varphi}(\xi)|^{2}\,d\xi  \\
                &\cong
                [u]_{\widetilde{H}^{\alpha}(\R)}.
            \end{align*}
            Taking the square root of each side, we obtain the desired result. 
        \end{proof}

        \begin{theorem}
            Let $0 < \alpha <1$. Then ${^{\pm}}{{H}}{^{s}} (\R) = \widetilde{H}^{s}(\R)$. 
        \end{theorem}
        
        \begin{proof}
            Step 1: Suppose $u\in{^{\pm}}{{H}}{^{\alpha}}(\R)$. Since $C^{\infty}_{0}(\R)$ is dense in ${^{\pm}}{{H}}{^{\alpha}}(\R)$, then there exists a sequence $\left\{u_j \right\}_{j=1}^{\infty} \subset C^{\infty}_{0}(\R)$ such that $ u_j \rightarrow u$ in ${^{\pm}}{{H}}{^{\alpha}}(\R)$. 
            Then by Lemma \ref{equivalence_seminorms} we get 
            \begin{align*}
                \left\|u _j \right\|_{\widetilde{H}^{\alpha}(\R)}^{2} &= \left\| u _j \right\|_{L^{2}(\R)}^{2} + \left[ u_j \right]_{\widetilde{H}^{\alpha}(\R)}^{2} 
                \leq  \left\|u_j \right\|_{L^{2}(\R)}^{2} + C \| {^{\mathcal{F}}}{D}{^{\alpha}} u_j \|_{L^{2}(\R)}^{2}   \\ 
                &=    \left\|u _ j\right\|_{L^{2}(\R)}^{2} + C\left\|{^{\pm}}{\mathcal{D}}{^{\alpha}} u_j \right\|_{L^{2}(\R)}^{2}  
                \leq  C \left\|u_j \right\|_{{^{\pm}}{{H}}{^{\alpha}}(\R)}^{2}.
            \end{align*}
            Consequently,
            \begin{align*}
                \|u_m - u_n\|_{\widetilde{H}^{\alpha}(\R)} \leq C \|u_m - u_n \|_{{^{\pm}}{H}{^{\alpha}}(\R)}\to 0 \quad\mbox{as } m,n\to \infty.
            \end{align*}
             Thus, $\{u_j\}_{j=1}^{\infty}$ is a Cauchy sequence in $\widetilde{H}^{\alpha}(\R)$. Since $\widetilde{H}^{\alpha}(\R)$ is a Banach space, there exists $v \in \widetilde{H}^{\alpha}(\R)$ so that $u_j \rightarrow v$ in $\widetilde{H}^{\alpha}(\R)$; in particular, $u_j \rightarrow v$ in $L^{2}(\R)$. By assumption, $u_j \rightarrow u$ in $L^{2}(\R)$. Therefore, $v = u$ a.e. in $\R$ and $u \in \widetilde{H}^{\alpha}(\R)$.

            Step 2: Let $u \in \widetilde{H}^{\alpha}(\R)$. By the approximation theorem, there exists
            a sequence $\{u_j\}_{j=1}^{\infty} \subset C^{\infty}_{0}(\R)$ 
            such that $u_j \rightarrow u$ in $\widetilde{H}^{\alpha}(\R)$. Then by Lemma \ref{equivalence_seminorms} we get 
            \begin{align*}
                \|u_j\|_{{^{\pm}}{H}{^{\alpha}}(\R)}^{2} &= \|u_j\|_{L^{2}(\R)}^{2} + \|{^{\pm}}{\mathcal{D}}{^{\alpha}} u_j \|_{L^{2}(\R)}^{2} \\ 
                &= \|u_j\|_{L^{2}(\R)}^{2} + \|{^{\mathcal{F}}}{D}{^{\alpha}}u_j \|_{L^{2}(\R)}^{2} \\ 
                &\leq \|u_j\|_{L^{2}(\R)}^{2} + C [ u_j]_{\widetilde{H}^{\alpha} (\R)}^{2}.
            \end{align*}
            It implies that 
            \begin{align*}
                \|u_m - u_n \|_{{^{\pm}}{H}{^{\alpha}}(\R)} \leq C \| u_m - u_n \|_{\widetilde{H}^{\alpha}(\R)}  \to 0 \quad\mbox{as } m,n\to \infty.
            \end{align*}
            Hence $\{u_j\}_{j=1}^{\infty}$ is a Cauchy sequence in ${^{\pm}}{H}{^{\alpha}}(\R)$. Since ${^{\pm}}{H}{^{\alpha}}(\R)$ is a Banach space, there exists $v \in {^{\pm}}{H}{^{\alpha}}(\R)$ so that $u_j \rightarrow v$ in ${^{\pm}}{H}{^{\alpha}}(\R)$; in particular, $u_j \rightarrow v$ in $L^{2}(\R)$. By assumption $u_j \rightarrow u$ in $L^{2}(\R)$. Therefore, $v=u$ a.e. and $u \in {^{\pm}}{H}{^{\alpha}}(\R)$. 
        \end{proof}

  \begin{remark}
        (a) The above result immediately infers that the equivalences \newline 
         ${^{\pm}}{H}{^{\alpha}}(\R) = \widetilde{H}^{\alpha}(\R) = \widehat{H}^{\alpha}(\R)$.
        
      (b) We note that  ${^{-}}{{H}}{^{s}} (\R) = {^{+}}{{H}}{^{s}} (\R)$, 
          however, this does not means 
        that the left and right weak derivatives of the same function are the same 
        or equivalent but rather two spaces contain the same set of functions. 
    
      (c) We conjecture that ${^{\pm}}{W}{^{\alpha,p}}(\R) \neq \widehat{W}^{\alpha,p}(\R)$, but 
      ${^{\pm}}{W}{^{\alpha,p}}(\R) = {\tW}^{\alpha,p}(\R)$ for $p\neq2$ and $0<\alpha<1$.
      
      (d) It can easily be shown that the equality ${^{\pm}}{W}{^{\alpha,p}}(\Omega) = {\tW}^{\alpha,p}(\Omega)$ cannot hold in general. It was proved that when $\alpha p > 1$, ${^{\pm}}{\mathcal{D}}{^{\alpha}} C \notin {^{\pm}}{W}{^{\alpha,p}}(\Omega)$. However, constant functions always belong to ${\tW}^{\alpha,p}(\Omega)$. In general, ${\tW}^{\alpha,p}(\Omega) \not\subset {^{\pm}}{W}{^{\alpha,p}}(\Omega)$. For the same reason, ${\tW}^{\alpha,p}(\Omega) \not\subset  {W}^{\alpha,p}(\Omega)$ when $\alpha p >1$. This simple example shows that the fractional derivative definition is fundamentally different from the (double) integral term resembling a difference quotient in the seminorm of ${\tW}^{\alpha,p}(\Omega)$. If an equivalence exists on the finite domain, it is our conjecture that for $\alpha p <1$, the spaces $ {W}^{\alpha , p}(\Omega)$ and ${\tW}^{\alpha,p}(\Omega)$ are the two spaces that should be comparable. 
      
  \end{remark}


\section{Weak Fractional Derivatives of Distributions} \label{sec-6}
The aim of this section is to introduce some weak fractional derivative notions for distributions. 
Like in the integer order case, such a notion  is necessary in order to define fractional order 
weak derivatives for ``all functions" including very rough ones and will also provide a useful tool 
for studying  fractional order differential equations (cf. \cite{Feng_Sutton2, Guo, Meerschaert}).   

The main difficulty for doing so is caused by the 
pollution effect of fractional order derivatives (and integrals), as a result,  the standard test space $\mathscr{D}(\Omega):=C^\infty_0(\Omega)$ is 
not invariant under the mappings ${^{\pm}}{D}{^{\alpha}}$, instead, ${^{\pm}}{D}{^{\alpha}}
(\mathscr{D}(\Omega)) \subset {^{\pm}}{\mathscr{D}(\Omega)}:= {^{\mp}}{C}{^{\infty}_{0}}(\Omega)$ 
(see the definitions below). 
Hence, ${^{\pm}}{D}{^{\alpha}} \varphi$ become invalid test functions (or inputs) for a distribution 
$u\in \mathscr{D}'(\Omega)$  although $\varphi \in \mathscr{D}(\Omega)$ is.  To circumvent this difficulty, 
there are two approaches used in the literature. The first one, which is most popular \cite{Samko}, is 
to use different test spaces which are larger than the standard test space $\mathscr{D}(\Omega)$ so that the chosen test space is invariant under  the mappings ${^{\pm}}{D}{^{\alpha}}$, and then to consider generalized functions (still called distributions) as continuous linear functionals on the chosen 
test space. The second approach is to extend the domain of a distribution  
$u\in \mathscr{D}'(\Omega)$ without changing the standard test space $\mathscr{D}(\Omega)$
so that the extended distribution $\tilde{u}$ can take the inputs ${^{\pm}}{D}{^{\alpha}} \varphi$. 
In this section, we use both approaches although we give more effort to the second one because it 
covers general distributions in $\mathscr{D}'(\Omega)$,  not just a subclass of $\mathscr{D}'(\Omega)$.   

\subsection{Test Spaces, Distributions and One-sided Distributions}\label{sec-6.1}
  We first recall some of the necessary function spaces and notions of convergence that are inherent to 
  constructing a fractional derivative for distributions. We also introduce two new  spaces of 
  one-side compactly supported functions and establish some properties of the weak fractional derivative operators ${^{\pm}}{\mathcal{D}}{^{\alpha}}$ 
  on the new spaces. Unless stated otherwise, in this section $\Omega$ denotes  
  either a finite interval $(a,b)$ or the real line $\R$.  
  
    \begin{definition}
        Let $\mathscr{D}(\Omega): =C^{\infty}_{0}(\Omega)$ which is equipped with the following 
        topology (sequential convergence): given a sequence $\{\varphi_k\}_{k=1}^{\infty} \in \mathscr{D}(\Omega)$ is said to converge to $\varphi \in \mathscr{D}(\Omega)$ if 
        \begin{itemize}
            \item[(a)] there exists a compact subset $K\subset \Omega$ such that $\supp(\varphi_k) \subset K$ for every $k$,
            \item[(b)] $D^{m} \varphi_k \rightarrow D^{m}\varphi$ uniformly in $K$ for each $m\geq 0$.
        \end{itemize}
    Let $\mathscr{D}^{\prime}(\Omega)$ denote the space of continuous linear functionals on 
    $\mathscr{D}(\Omega)$, namely the dual space. 
    Every functional in $\mathscr{D}^{\prime} (\Omega)$ is called a distribution. 
    \end{definition}

    \begin{definition}
        Define the following two spaces of one-side compactly supported functions: 
        \begin{align*}
        {^{-}}{\mathscr{D}(\Omega)} &:= \{\varphi\in C^{\infty}(\Omega)\,:\, \exists x_0 \in \Omega, \varphi(x) \equiv 0 \, \forall x \leq x_0\}, \\
         {^{+}}{\mathscr{D}(\Omega)} &:= \{\varphi\in C^{\infty}(\Omega)\,:\, \exists x_0 \in \Omega, \varphi(x) \equiv 0 \, \forall x \geq x_0\},
         \end{align*} 
        which are equipped with the following topology: given a sequence $\{\varphi_k\}_{k=1}^{\infty} \in {^{\pm}}{\mathscr{D}(\Omega)}$, it is said to converge to 
        $\varphi \in  {^{\pm}}{\mathscr{D}(\Omega)}$ if 
        \begin{itemize}
            \item[(a)] there exists an $x_0 \in \Omega$ such that $\varphi_k (x) \equiv 0$ for all 
            $x \leq x_0$ (or $x \geq x_0$ in the case of the right space) for  $k\geq  1$, 
            \item[(b)] $D^m \varphi_k \rightarrow D^m \varphi$ uniformly in $\Omega$ for every $m\geq 0$.
        \end{itemize}
    Let ${^{\pm}}{\mathscr{D}^{\prime}(\Omega)} $ denote respectively the spaces of continuous linear functionals on 
    ${^{\pm}}{\mathscr{D}(\Omega)}$,  namely the dual spaces of ${^{\pm}}{\mathscr{D}}(\Omega)$. 
    Every functional in ${^{\pm}}{\mathscr{D}^{\prime}}(\Omega) $ is called a one-sided distribution. 
    \end{definition} 
    
    One should note the the spaces ${^{\pm}}{\mathscr{D}}(\Omega) = {^{\mp}}{C}{^{\infty}_{0}}(\Omega)$ as defined in Section \ref{sec-5.4}.
    
   \begin{lemma}
      $\mathscr{D}(\Omega)$ and ${^{\pm}}{\mathscr{D}}(\Omega) $ are complete topological vector spaces
      and $\mathscr{D}(\Omega)\subset {^{\pm}}{\mathscr{D}(\Omega)} $.
   \end{lemma}

Recall that it was proved in Section \ref{sec-2.7} that 
${^{\pm}}{D}{^{\alpha}}(\mathscr{D}(\Omega)) \subset {^{\pm}}{\mathscr{D}}(\Omega) $,  
Below we show that this inclusion is continuous.

\begin{proposition}\label{continuity}
	 ${^{\pm}}{\mathcal{D}}{^{\alpha}} : \mathscr{D}(\Omega)\to 
	 {^{\pm}}{\mathscr{D}}(\Omega)$ are continuous. 
\end{proposition}

\begin{proof}
We only give a proof for the left derivative ${^{-}}{\mathcal{D}}{^{\alpha}} = {^{-}}{D}{^{\alpha}}$ because the 
other case follows similarly.  

Let $\varphi_k \rightarrow \varphi$ in $\mathscr{D}(\Omega)$,  we want to show that ${^{-}}{{D}}{^{\alpha}} \varphi_k \rightarrow {^{-}}{{D}}{^{\alpha}} \varphi$ in ${^{-}}{\mathscr{D}}(\Omega) $. To the end, let $K \subset\subset  \Omega$ be a compact subset 
so that $\supp(\varphi_k) \subset K$ for every $k \geq 0$ with $\varphi_0\equiv \varphi$, without loss of the generality, assume 
$K=[x_0 , x_1] \subset\subset  \Omega$. Then we have ${^{-}}{D}{^{\alpha}} \varphi_k \equiv 0$ 
for every $x \leq x_0$ and $k \geq 0$, and for any integer $m\geq 0$ and $x > x_0$
        \begin{align*}
            &\bigl|D^m ({^{-}}{D}{^{\alpha}} \varphi_k) (x) - D^m ({^{-}}{D}{^{\alpha}} \varphi)(x)\bigr|\\
            &\quad = \biggl| \dfrac{d^m}{dx^m} \left[ C_\alpha \dfrac{d}{dx} \int_{x_0}^{x} \dfrac{\varphi_k(y)}{(x-y)^{\alpha}} \,dy \right] - \dfrac{d^m}{dx^m} \left[ C_\alpha \dfrac{d}{dx} \int_{x_0}^{x} \dfrac{\varphi(y)}{(x-y)^{\alpha}} \,dy \right] \biggr| \\ 
            & \quad = \left|C_\alpha \dfrac{d^{m+1}}{dx^{m+1}} \int_{x_0}^{x} \dfrac{\varphi_k(y) - \varphi(y)}{(x-y)^{\alpha}}\,dy\right| \\ 
            &\quad = \left|C_\alpha \int_{x_0}^{x} \dfrac{\varphi_k^{(m+1)}(y) - \varphi^{(m+1)}(y)}{(x-y)^{\alpha}}\,dy\right|\\
            &\quad \leq C_\alpha \int_{x_0}^{x_1} 
            \dfrac{ \bigl|\varphi_k^{(m+1)}(y) - \varphi^{(m+1)}(y)\bigr|}{|x-y|^\alpha}\,dy\\
            &\quad \leq \dfrac{|K|^{1-\alpha}C_\alpha}{1-\alpha} \sup_{x\in K}\Bigl|\varphi_k^{(m+1)}(x)- \varphi^{(m+1)}(x)\Bigr|.
        \end{align*}
         
        It follows by the uniform convergence of $\{\varphi_k\}_{k=1}^{\infty}$ that 
        $D^m {^{-}}{{D}}{^{\alpha}}\varphi_k \rightarrow D^m {^{-}}{{D}}{^{\alpha}} \varphi$ uniformly in $\Omega$ for every $m$.  The proof is complete.
    \end{proof}

The above proof also infers that  the  spaces ${^{\pm}}{\mathscr{D}}(\Omega)$ are invariant 
under the mapping ${^{\pm}}{D}{^{\alpha}}$, respectively. 

\begin{proposition}\label{invarance}
	${^{\pm}}{\mathcal{D}}{^{\alpha}} ({^{\pm}}{\mathscr{D}}(\Omega) )\subset 
	{^{\pm}}{\mathscr{D}}(\Omega)$, respectively. Moreover, the inclusion is continuous. 
\end{proposition}

\begin{remark}
	Without any added integrability condition (i.e. decay at $x=\pm \infty$), the inclusions of Proposition \ref{invarance} may not be 
	true when $\Omega=\R$.  The smoothed (in an $\eps$-neighborhood of $x=0$) Heaviside functions $H_\eps(x)$ and $H_\eps (-x)$ are two counterexamples.  In fact, ${^{\pm}}{{D}}{^{\alpha}} \varphi$ may even not exist for some $\varphi\in  
	{^{\pm}}{\mathscr{D}}(\R)$. 
\end{remark}

Let $\mathcal{S}$ denote the space of Schwartz rapidly decaying functions defined in $\R$ (see \cite{Rudin} for the precise definition). Then we have

\begin{lemma}\label{invaranceF}
The space $\mathcal{S}$ is invariant under the Fourier fractional order derivative operator, namely, 	
$ {^{\mathcal{F}}}{D}{^{\alpha}}(\mathcal{S}) \subset \mathcal{S}$.  Moreover, the 
inclusion is continuous. 
\end{lemma}

\begin{proof}
	Let $\varphi\in \mathcal{S}$, it is well known \cite{Adams, Rudin} that $\hat{\varphi}:=\mathcal{F}[\varphi] \in \mathcal{S}$.
	Then $(i\xi)^\alpha \hat{\varphi} \in \mathcal{S}$, so is 
	$ {^{\mathcal{F}}}{D}{^{\alpha}}(\varphi):=\mathcal{F}^{-1} [(i\xi)^\alpha \hat{\varphi}]$. 
	The continuity of the inclusion can be proved in the same way as that in Proposition \ref{continuity}.
\end{proof}

\begin{remark}
	It is easy to check that the Schwartz space $\mathcal{S}$  is not invariant under the mappings
	${^{\pm}}{\mathcal{D}}{^{\alpha}}$, nor is it under ${^{\pm}}{D}{^{\alpha}}$.  Consequently,
	the Fourier fractional derivatives and the Riemann-Liouville fractional derivatives may not coincide 
	for functions in $\mathcal{S}$ in general. 	On the other hand, they do coincide for functions
	in $\mathscr{D}$ (see  Proposition \ref{EquivalencesonR}).  This fact is a main reason for and also validates the choice 	of test functions in the definition of weak fractional derivatives in Section \ref{sec-4.1}. 
\end{remark}

 \subsection{Weak Fractional Derivatives for Compactly Supported Distributions}\label{sec-6.2}
The goal of this subsection is to extend the notion of the   weak fractional derivatives to 
distributions in $\mathscr{D}'$ with compact supports.  First, we recall the definition of
supports for distributions. 

\begin{definition}\label{supportD}
	Let  $u\in \mathscr{D}'(\Omega)$, $u$ is said to vanish on an open subset $O\subset \Omega$
	 if  $u(\varphi)=0$ for all $\varphi\in C^\infty (\Omega)$ with $\supp(\varphi)\subset  O$.  
	 Let $O_{max}$ be a maximal open subset of $\Omega$ on which the distribution $u$ vanishes. 
	 The support of $u$   is defined as the complement of $O_{max}$ in $\Omega$, that is,  
	 $\supp(u):=\Omega\setminus O_{max}$. 
	 Moreover, $u$ is said to be compactly supported if $\supp(u)$  is a compact set. 
\end{definition} 

The best known compactly supported distribution is the Dirac delta function $\delta_0$ which is defined
by $\delta_0(\varphi)=\varphi(0)$ for any $\varphi\in \mathscr{D}(\R)$. $\delta_0$ has the one point
support $\{x=0\}$ and zero order. 

Given a compact subset $K\subset\subset \Omega$, we also define the space 
\[
\mathscr{D}^\prime_K(\Omega):=\bigl\{ u\in \mathscr{D}'(\Omega):\, \supp(u)\subseteq K \bigr\}.
\]

    \begin{lemma}\label{cut_off}
        Let $0 < \alpha < 1$ and $\psi , \varphi \in \mathscr{D}(\Omega) $. Then $\psi {^{\pm}}{D}{^{\alpha}} \varphi \in \mathscr{D}(\Omega)$. Moreover, if $\varphi_k \rightarrow \varphi$ in $\mathscr{D}(\Omega)$, then $\psi {^{\pm}}{D}{^{\alpha}} \varphi_k \rightarrow \psi {^{\pm}}{D}{^{\alpha}} \varphi$ in $\mathscr{D}(\Omega)$.
    \end{lemma}

    \begin{proof}
    	Let  $\psi , \varphi \in \mathscr{D}(\Omega)$.
        Recall that ${^{\pm}}{D}{^{\alpha}} \varphi \in C^{\infty}(\Omega)$. Then, 
         $\psi {^{\pm}}{D}{^{\alpha}} \varphi \in  \mathscr{D}(\Omega)$. 
         It remains to show  the desired convergence result. Again, we only give a proof for 
         the left space because the other case follows similarly. 
    
    Suppose  that $\varphi_k \rightarrow \varphi$ in $\mathscr{D}(\Omega)$, then 
        there exists a compact subset $K \subset \Omega $ so that $\supp( \varphi_k) \subset K$ 
        for all $k\geq 1$.  Without
        loss of the generality, assume $K=[x_0,x_1]$ and $K\cap \supp(\psi)\subset [x_0, x_2]$ for some 
        $x_2> x_0$.  Then, ${^{-}}{D}{^{\alpha}}
        \varphi_k \equiv 0$ for $x\leq x_0$ and all $k\geq 1$ and $\psi\equiv 0$ for all $x>x_2$. 
        Thus,  for any integer $m\geq 1$  and $x_0<x<x_2$
        \begin{align*}
            &\bigl|D^m (\psi {^{-}}{D}{^{\alpha}} \varphi_k (x) )- D^m(\psi {^{-}}{D}{^{\alpha}} \varphi (x) ) \bigr|\\
            &\qquad = \biggl| \sum_{j=0}^{m} \binom{m}{j} \psi^{(m-j)} D^j {^{-}}{D}{^{\alpha}} \varphi_k (x) - \sum_{j=0}^{m} \binom{m}{j} \psi^{(m-j)} D^j {^{-}}{D}{^{\alpha}} \varphi (x) \biggr| \\ 
            &\qquad \leq \sum_{j=0}^{m} \binom{m}{j} \left|\psi^{(m-j)}(x) D^j {^{-}}{D}{^{\alpha}} (\varphi_k - \varphi)(x)\right|\\
            &\qquad = \sum_{j=0}^{m} \binom{m}{j} \left|\psi^{(m-j)}(x) {^{-}}{I}{^{1-\alpha}}(\varphi_{k}^{(j+1)} - \varphi^{(j+1)})(x)\right|\\
            &\qquad \leq \dfrac{C_m C_\alpha}{1-\alpha} \sup_{x_0\leq x \leq x_2 \atop 1\leq j\leq m}  \Bigl( \bigl|\psi^{(m-j)}(x)\bigr|\cdot \bigl|\varphi_k^{(j+1)}(x) - \varphi^{(j+1)}(x) \bigr|  \Bigr).
        \end{align*}
        By the uniform convergence of $\varphi_k$ to $\varphi$ in $\mathscr{D}(\Omega)$, we obtain the desired result. 
    \end{proof}

 The above lemma guarantees that  $\psi {^{\pm}}{D}{^{\alpha}} \varphi$ belongs to the 
 standard test space $ \mathscr{D}(\Omega)$  which removes most pollution contribution in  
 ${^{\pm}}{D}{^{\alpha}} \varphi$ by using a compactly supported smooth (cutoff) function $\psi$. 

For compactly supported distributions, there holds the following result, its proof can be 
found in \cite[Theorem 6.24]{Rudin}.

\begin{theorem}\label{extension_rudin}
	Let $u\in \mathscr{D}^\prime_K(\Omega)$. 
	Then $u$ has a finite (integer) order $N (\geq 0)$ and can be uniquely extended to a 
	continuous  linear functional on $C^\infty(\Omega)$ which is given by 
	\[
	\tilde{u}(\varphi) := u(\psi \varphi) \qquad \forall \varphi\in C^\infty(\Omega),
	\] 
	where $\psi\in C^\infty_0(\Omega)$ satisfying $\psi\equiv 1$ in $K$ is a partition of unity.
\end{theorem}

We note that the extension $\tilde{u}$ as a functional does not depend on the choice of 
the cut-off function $\psi$ (see \cite{Rudin} for a proof). 

We now are ready to define weak fractional derivatives for compactly supported distributions in $\mathscr{D}'(\Omega)$.

    \begin{definition}
        Let $ \alpha >0$ and  $u \in \mathscr{D}^\prime_K(\Omega)$. 
        Define ${^{\pm}}{\mathcal{D}}{^{\alpha}} u:\mathscr{D}(\Omega)\to \R$  respectively by  
        \begin{align*}
           {^{\pm}}{\mathcal{D}}{^{\alpha}} u  (\varphi):= (-1)^{[\alpha]} \tilde{u} ({^{\mp}}{D}{^{\alpha}} \varphi)   
                    = (-1)^{[\alpha]} u(\psi {^{\mp}}{D}{^{\alpha}} \varphi)   
           \qquad \forall \varphi \in \mathscr{D}(\Omega),
        \end{align*}
        where $\tilde{u}$ and $\psi$ are the same as in Theorem \ref{extension_rudin}. 
        
    \end{definition}
 
 The next theorem shows that a compactly supported distribution $u\in \mathscr{D}^\prime_K(\Omega)$ has 
 any order  weak fractional derivative ${^{\pm}}{\mathcal{D}}{^{\alpha}} u$ which belongs to $\mathscr{D}'(\Omega)$.
 
    \begin{theorem}\label{existence}
        Let $\alpha>0$ and suppose $u \in \mathscr{D}^\prime_K(\Omega)$. Then 
        \begin{itemize}
       \item [{\rm (i)}] ${^{\pm}}{\mathcal{D}}{^{\alpha}} u \in \mathscr{D}'(\Omega)$. Moreover, if $K \subseteq [c,d]\subset\subset\Omega$, then $\supp({^{-}}{\mathcal{D}}{^{\alpha}} u)\subseteq(-\infty, d]\cap \Omega$ and $\supp({^{+}}{\mathcal{D}}{^{\alpha}} u)\subseteq[c, \infty)\cap \Omega$.
       \item[{\rm (ii)}] Suppose that $\{ u_j\}_{j=1}^{\infty} \in \mathscr{D}^\prime_K(\Omega)$
       such that 
        $u_j \rightarrow u$ in $\mathscr{D}^\prime_K(\Omega)$, then ${^{\pm}}{\mathcal{D}}{^{\alpha}} u_j \rightarrow {^{\pm}}{\mathcal{D}}{^{\alpha}} u$ in $\mathscr{D}'(\Omega)$.
        \end{itemize}
    \end{theorem}

    \begin{proof}
       (i) The linearity of ${^{\pm}}{\mathcal{D}}{^{\alpha}} u$ is trivial. 
       To show the continuity, it suffices to show that ${^{\pm}}{\mathcal{D}}{^{\alpha}} u$ 
       is sequentially continuous at zero. To the end, let $\{\varphi_k\}
       {k=1}^{\infty} \in \mathscr{D}(\Omega)$ so that 
       $\varphi_k \rightarrow 0$ in $\mathscr{D}(\Omega)$. It follows by Lemma \ref{cut_off} that 
        \begin{align*}
           {^{\pm}}{\mathcal{D}}{^{\alpha}}u(\varphi_k) = (-1)^{[\alpha]} u(\psi {^{\mp}}{D}{^{\alpha}} \varphi_k) \to   u(\psi {^{\mp}}{D}{^{\alpha}} (0)) = 0
           \qquad\mbox{as }k\to \infty.
        \end{align*}
        
        Since for any $\varphi \in  \mathscr{D}(\Omega)$, ${^{\pm}}{\mathcal{D}}{^{\alpha}}\varphi
        \in {^{\pm}}{C}{^{\infty}_{0}}(\Omega)$, then the supports of ${^{\pm}}{\mathcal{D}}{^{\alpha}}u$
        pollute that of $u$ to the right/left accordingly.  
        
        (ii) Suppose that $u_j \to u$ in $\mathscr{D}'(\Omega)$, we have for any $\varphi\in \mathscr{D}(\Omega)$
        \begin{align*}
        {^{\pm}}{\mathcal{D}}{^{\alpha}} u_j (\varphi) : = (-1)^{[\alpha]} u_j (\psi {^{\mp}}{D}{^{\alpha}} \varphi) 
        \underset{j\to \infty}{\longrightarrow}  (-1)^{[\alpha]} u( \psi {^{\mp}}{D}{^{\alpha}} \varphi) 
        ={^{\pm}}{\mathcal{D}}{^{\alpha}} u(\varphi).
        \end{align*}
        Thus, ${^{\pm}}{\mathcal{D}}{^{\alpha}} u_j \rightarrow {^{\pm}}{\mathcal{D}}{^{\alpha}} u$ in $\mathscr{D}'(\Omega)$ as $j\to\infty$.
        The proof is complete.
    \end{proof}

    \begin{proposition}
        Let $\alpha>0$ and suppose $u \in \mathscr{D}^\prime_K(\Omega)$. Then ${^{\pm}}{\mathcal{D}}{^{\alpha}} u  \to \mathcal{D}u $ in $\mathscr{D}'(\Omega)$ as $\alpha \rightarrow 1^-$ and ${^{\pm}}{\mathcal{D}}{^{\alpha}} u  \to \mathcal{D}u $ in $\mathscr{D}'(\Omega)$ as $\alpha \rightarrow 1^+$.
    \end{proposition}

    \begin{proof}
        For any $\varphi \in C^{\infty}_{0}(\Omega)$, we have 
        \begin{align*}
             {^{\pm}}{\mathcal{D}}{^{\alpha}} u (\varphi) = (-1)^{[\alpha]}   u(\psi {^{\mp}}{D}{^{\alpha}} \varphi) \underset{\alpha\to 1^-}{\longrightarrow}  -u(\psi D \varphi)= -u(D\varphi) = \mathcal{D}u(\varphi),\\
              {^{\pm}}{\mathcal{D}}{^{\alpha}} u (\varphi) = (-1)^{[\alpha]}   u(\psi {^{\mp}}{D}{^{\alpha}} \varphi) \underset{\alpha\to 1^+}{\longrightarrow}  - u(\psi D \varphi)=-u(D\varphi) =\mathcal{D}u(\varphi). 
        \end{align*}
        Hence, the assertions hold.
    \end{proof}

    \begin{proposition}
        Let $\Omega = (a,b)$ and $0<\alpha<1$. Suppose that $ u \in \mathscr{D}^\prime_K(\Omega)$ and $\eta \in C^{\infty}(\Omega)$, 
        then there holds the following product rule: 
        \begin{align}\label{product_rule_dist}
         {^{\pm}}{\mathcal{D}}{^{\alpha}} (\eta u) 
         = \eta {^{\pm}}{\mathcal{D}}{^{\alpha}} u -  \sum_{k=1}^{m} D^k \eta \, {^{\pm}}{I}{^{k-\alpha}} u
         - C_{m,\alpha}  (\eta^{(m+1)} * \mu_{\pm}) \,  (\kappa_{\pm}^{-\alpha} *\psi u),
        \end{align}
        where
        \begin{align} \label{C_kalpha}
        C_{k,\alpha} &:= \dfrac{\Gamma(1+\alpha)}{\Gamma(k+1) \Gamma(1-k + \alpha)},\\
        {^{\pm}}{I}{^{k-\alpha}} u (\varphi) &:= C_{k,\alpha} u\bigl(\psi {^{\mp}}{I}{^{k-\alpha}} \varphi \bigr)
        \qquad \forall \varphi\in \mathscr{D}(\Omega). \label{fractional_int_dist}
        \end{align}
       
    \end{proposition}

    \begin{proof} By the fractional order product rule, we have 
        \begin{align*}
            {^{\pm}}{\mathcal{D}}{^{\alpha}} (\eta u)(\varphi) 
            :& = \eta u ( \psi {^{\mp}}{D}{^{\alpha}} \varphi ) 
            =   u( \eta \psi {^{\mp}}{D}{^{\alpha}} \varphi ) 
            = (u , \psi \eta {^{\mp}}{D}{^{\alpha}} \varphi ) \\
            &= u\bigl(\psi {^{\mp}}{D}{^{\alpha}} (\varphi \eta) \bigr)  - u \Bigl( \psi \sum_{k=1}^{m} C_{k,\alpha} {^{\mp}}{I}{^{k-\alpha}} \varphi D^{k} \eta \Bigr) 
            - u \bigl(\psi {^{\mp}}{R}{^{\alpha}}(\varphi, \eta) \bigr) \\
            &=: I - II - III
        \end{align*}
        where 
        \begin{align*}
        {^{+}}{R}{^{\alpha}_{m}}(\varphi, \eta) &: = \dfrac{(-1)^{m+1}}{m! \Gamma(-\alpha)} \int_{x}^{b} \dfrac{\varphi(y)}{(y-x)^{1+\alpha}}\,dy \int_{x}^{y} \eta^{(m+1)}(z) (z-x)^{m} \,dz 
        \end{align*}
        with a similar formula for ${^{-}}{R}{^{\alpha}_{m}} (\varphi, \eta)$.  
        
        For terms $I$ and $II$  we have 
        \begin{align*}
            I:& =  u\Bigl(\psi {^{\mp}}{D}{^{\alpha}} (\varphi \eta ) \Bigr)  
            ={^{\pm}}{\mathcal{D}}{^{\alpha}}u(\eta \varphi) 
            = \eta {^{\pm}}{\mathcal{D}}{^{\alpha}} u(\varphi),\\
            II:&= u\Bigl( \psi \sum_{k=1}^{m} C_{k,\alpha} {^{\mp}}{I}{^{k-\alpha}} \varphi D^{k} \eta \Bigr)
            = \sum_{k=1}^{m} C_{k,\alpha} u\bigl(\psi {^{\mp}}{I}{^{k-\alpha}} \varphi D^{k} \eta \bigr) \\ 
            &= \sum_{k=1}^{m} C_{k,\alpha} D^{k} \eta \, u\bigl(\psi {^{\mp}}{I}{^{k-\alpha}} \varphi \bigr) 
           = \sum_{k=1}^{m} D^k \eta \, {^{\pm}}{I}{^{k-\alpha}} u (\varphi).
        \end{align*}

        Finally, to simplify term $III$, we rewrite the remainder formula as follows: 
        \begin{align*}
            {^{+}}{R}{^{\alpha}_{m}}(\varphi, \eta) 
            &: = \dfrac{(-1)^{m+1}}{m! \Gamma(-\alpha)} \int_{x}^{b} \int_{x}^{y} \dfrac{\varphi(y)}{(y-x)^{1+\alpha}} \eta^{(m+1)}(z)(z-x)^{m} \,dzdy \\ 
            &= \dfrac{(-1)^{m+1}}{m!\Gamma(-\alpha)} \int_{x}^{b} \dfrac{\varphi(y)}{(y-x)^{1+\alpha}} (\eta^{(m+1)} * \mu_{+})(y)\,dy \\ 
            &= C_{m , \alpha} \bigl( \varphi (\eta^{(m+1)} * \mu_{+}) * \kappa_+^{-\alpha} \bigr)(x).
        \end{align*}
        Then we have 
        \begin{align*}
            III :&= u\bigl( \psi {^{+}}{R}{^{\alpha}}(\varphi, \eta) \bigr) \\ 
            &=  u \bigl( C_{m,\alpha} \psi (\varphi (\eta^{(m+1)} * \mu_{+} ) *\kappa^{-\alpha}_{+}) \bigr) \\
            &= C_{m,\alpha} u \bigl(\psi (\varphi (\eta^{(m+1)} * \mu_{+} ) * \kappa^{-\alpha}_{+}) \bigr) \\ 
            &=  C_{m,\alpha}  (\kappa_{+}^{-\alpha} * \psi u) \bigl( \varphi (\eta^{(m+1)} *\mu_+) \bigr)  \\
            &=C_{m,\alpha}  (\eta^{(m+1)} * \mu_{+}) \,  (\kappa_{+}^{-\alpha} * \psi u) ( \varphi).  
        \end{align*}
   The desired formula \eqref{product_rule_dist} follows from combining the above three identities. The proof is complete. 
    \end{proof}

  \subsection{Weak Fractional Derivatives for Distributions on Finite Intervals}\label{sec-6.3}
  In the previous subsection we introduce a fractional order derivative notion for compactly supported 
  distributions in $\mathscr{D}_K^\prime(\Omega)$. The aim of this subsection is to introduce a fractional 
  derivative notion for general distributions in $\mathscr{D}^\prime(\Omega)$ when $\Omega=(a,b)$ is finite. 
  We shall address the case $\Omega=\R$ in the next subsection. 
  
First, we consider the class of one sided generalized functions in 
${^{\pm}}{\mathscr{D}}^\prime (\Omega):= ({^{\pm}}{\mathscr{D}} (\Omega ))^\prime$,
which are proper subspaces of ${\mathscr{D}}^\prime (\Omega)$. By Proposition \ref{invarance} we know that
${^{\pm}}{\mathscr{D}}(\Omega)$ are respectively invariant under the mappings ${^{\pm}}{D}{^{\alpha}}$. 
This fact then makes defining ${^{\pm}}{\mathcal{D}}{^{\alpha}}u$
for $u\in {^{\pm}}{\mathscr{D}}^\prime (\Omega)$ a trivial task.

 \begin{definition}
	Let $ \alpha >0$ and  $u \in {^{\pm}}{\mathscr{D}}^\prime (\Omega) $. 
	Define ${^{\pm}}{\mathcal{D}}{^{\alpha}} u: {^{\pm}}{\mathscr{D}}(\Omega) \to \R$ respectively by  
	\begin{align}\label{eq6.3}
	{^{\pm}}{\mathcal{D}}{^{\alpha}} u  (\varphi):=  
	 (-1)^{[\alpha]} u({^{\mp}}{D}{^{\alpha}} \varphi)   
	\qquad \forall \varphi \in {^{\pm}}{\mathscr{D}}(\Omega) .
	\end{align}
  
\end{definition}

Clearly, ${^{\pm}}{\mathcal{D}}{^{\alpha}} u$ is well defined and ${^{\pm}}{\mathcal{D}}{^{\alpha}} u
\in {^{\pm}}{\mathscr{D}}^\prime (\Omega) $, respectively. It also can be shown that many other properties 
hold for the fractional order derivative operators ${^{\pm}}{\mathcal{D}}{^{\alpha}}$ on the one sided 
generalized function spaces ${^{\pm}}{\mathscr{D}}^\prime (\Omega)$. We leave the verification to 
the interested reader.   

To define fractional order derivatives for distributions in ${\mathscr{D}}^\prime (\Omega)\setminus {^{-}}{\mathscr{D}}^\prime (\Omega)\cup {^{+}}{\mathscr{D}}^\prime (\Omega)$, we need to 
construct``good" extensions for any distribution $u\in {\mathscr{D}}^\prime (\Omega)$ to 
${^{-}}{\mathscr{D}}^\prime (\Omega)$ and  ${^{+}}{\mathscr{D}}^\prime (\Omega)$. 
This will be done below by using the partition of unity theorem to define 
$u (\varphi):= \sum_{j = 1}^{\infty} u( \psi_j  \varphi)$ for any $\varphi\in {^{\pm}}{\mathscr{D}}(\Omega)$.

Let $\{I_\beta\}$ be a family of open subintervals of $(a,b)$ which forms a covering of $\Omega$. 
By the  partition of unity theorem (cf. \cite{Rudin}),  there exists a subsequence $\{I_j\}_{j =1}^{\infty} \subset \{I_\beta\}$ and a partition  of the unity $\{\psi_j\}_{j = 1}^{\infty}$
subordinated to $\{I_j\}_{j= 1}^{\infty}$, namely, $\psi_j\in C^\infty_0(I_j)$ for $j\geq 1$ and $\sum \psi_j(x)\equiv 1$ 
on every compact subset $K$ of $\Omega$ and the sum is a finite sum for every $x\in K$.

\begin{definition}\label{def-partition}
	Let $ \alpha >0$ and  $u \in  {\mathscr{D}}^\prime (\Omega) $. 
	Define ${^{\pm}}{\mathcal{D}}{^{\alpha}} u:  {\mathscr{D}}(\Omega) \to \R$ respectively by  
	\begin{align}\label{eq6.4}
	{^{\pm}}{\mathcal{D}}{^{\alpha}} u  (\varphi):=  
	(-1)^{[\alpha]} \sum_{j = 1}^{\infty} u( \psi_j {^{\mp}}{D}{^{\alpha}} \varphi)   
	\qquad \forall \varphi \in  {\mathscr{D}}(\Omega) .
	\end{align}

\end{definition}

We claim that ${^{\pm}}{\mathcal{D}}{^{\alpha}} u$ is well defined and ${^{\pm}}{\mathcal{D}}{^{\alpha}} u \in  {\mathscr{D}}^\prime (\Omega)$, respectively.
We leave the verification to the interested reader.

  \subsection{Weak Fractional Derivatives for Distributions on $\R$ }\label{sec-6.4}
  To define Riemann-Liouville fractional order derivatives for distributions in ${\mathscr{D}}^\prime (\R)$ is more 
  complicated than in ${^{\pm}}{\mathscr{D}}^\prime (\Omega)$; the complication is due
  to the fact that the kernel functions $\kappa_{\pm}^{\alpha} \not\in L^1(\R)$ and the pollutions of ${^{\pm}}{\mathcal{D}}{^{\alpha}}\varphi(x)$ for $\varphi\in \mathscr{D}(\R)$ do not decay fast enough 
  when $x\to \pm \infty$. 
  
  We first consider the simpler case of Fourier fractional order derivatives for tempered distributions 
  in $\mathcal{S}^\prime(\R)$. By Proposition \ref{invaranceF} we know that the Schwartz space $\mathcal{S}(\R)$ 
  is invariant under the Fourier derivative operator ${^{\mathcal{F}}}{D}{^{\alpha}}$. This allows us easily 
  to define  Fourier fractional derivatives for tempered distributions as follows.
  
  \begin{definition}
 	Let $ \alpha >0$ and  $u \in  {\mathcal{S}}^\prime (\R) $. 
 	Define ${^{\mathcal{F}}}{\mathcal{D}}{^{\alpha}} u:  \mathcal{S}(\R) \to \R$ by  
 	\begin{align}\label{eq6.5}
 	 {^{\mathcal{F}}}{\mathcal{D}}{^{\alpha}} u  (\varphi):=  
 	(-1)^{[\alpha]} u \bigl( {^{\mathcal{F}}}{D}{^{\alpha}} \varphi \bigr)  
 	\qquad \forall \varphi \in  \mathcal{S}(\R).
 	\end{align} 
\end{definition}
  
It is easy to verify that ${^{\mathcal{F}}}{\mathcal{D}}{^{\alpha}} u$ is well defined and 
${^{\mathcal{F}}}{\mathcal{D}}{^{\alpha}} u \in \mathcal{S}^\prime(\R)$. It also can be shown that 
many other properties hold for the fractional order derivative operator ${^{\mathcal{F}}}{\mathcal{D}}{^{\alpha}}$ 
on the space of tempered distributions ${\mathcal{S}}^\prime (\R)$. We leave the verification to 
the interested reader.

To define fractional order derivatives for distributions in
${\mathscr{D}}^\prime (\R)\setminus {\mathcal{S}}^\prime (\R)$, we need to extend the domain 
of $u\in {\mathscr{D}}^\prime (\R)$ from ${\mathscr{D}}(\R)$ to ${\mathcal{S}}(\R)$ (or  ${^{\pm}}{\mathscr{D}}(\R)$). 
Again, this will be done by using the partition of the unity theorem as seen above 
to define 
$u (\varphi):= \sum_{j = 1}^{\infty} u( \psi_j  \varphi)$ for any $\varphi \in  {^{\pm}}{\mathscr{D}}(\R)$.

Let $\{I_\beta\}$ be a family of open finite subintervals of $\R$ which forms a covering of $\R$. 
By the  partition of unity theorem (cf. \cite{Rudin}),  there exists a subsequence $\{I_j\}_{j =1}^{\infty} \subset \{I_\beta\}$ and a partition  of the unity $\{\psi_j\}_{j = 1}^{\infty}$
subordinated to $\{I_j\}_{j= 1}^{\infty}$, namely, $\psi_j\in C^\infty_0(I_j)$ 
for $j\geq 1$ and $\sum \psi_j(x)\equiv 1$ on every compact subset $K$ of $\R$ 
and the sum is a finite sum for every $x\in K$.

	\begin{definition}\label{def-partition_R}
		Let $ \alpha >0$ and  $u \in  {\mathscr{D}}^\prime (\R) $. 
		Define ${^{\pm}}{\mathcal{D}}{^{\alpha}} u:  {\mathscr{D}}(\R) \to \R$ respectively by  
		\begin{align}\label{eq6.6}
		{^{\pm}}{\mathcal{D}}{^{\alpha}} u  (\varphi):=  
		(-1)^{[\alpha]} \sum_{j = 1}^{\infty} u( \psi_j {^{\mp}}{D}{^{\alpha}} \varphi)   
		\qquad \forall \varphi \in  {\mathscr{D}}(\R) .
		\end{align}

	\end{definition} 

 We claim that ${^{\pm}}{\mathcal{D}}{^{\alpha}} u$ is well defined and ${^{\pm}}{\mathcal{D}}{^{\alpha}} u \in  {\mathscr{D}}^\prime (\R) $, respectively.
Again, we leave the verification to the interested reader.

\section{Conclusion}\label{sec-7}
   In this paper we first gave a holistic review/survey of the classical fractional
   calculus theory. We also gave some new interpretations of the classical theory from different 
   perspectives, and especially emphasized the importance of the Fundamental 
   Theorem of Classical Fractional Calculus (FTcFC) and its ramifications in the classical theory. 
   We then presented a complete and self-contained new theory of weak fractional differential 
   calculus and fractional Sobolev spaces in one-dimension. The crux of this new theory 
   is the introduction of a weak fractional derivative notion which is a natural 
   generalization of integer order weak derivatives; it also helps to unify multiple existing 
   fractional derivative definitions and has the potential to be easily extended to higher dimensions. 
   Various calculus rules including a fundamental theorem of weak fractional calculus, product and chain rules,
   and integration by parts formulas were established for weak fractional derivatives 
   and relationships with existing classical 
   fractional derivatives were also obtained. Based on the weak fractional derivative 
   notion, new fractional order Sobolev spaces were introduced in the exact same manner as was done for the integer order Sobolev spaces. Many important theorems and properties, such as 
   density theorem, extension theorems, trace theorem, various embedding theorems, and 
   integration by parts formulas in those Sobolev spaces were established. 
   Moreover, a few relationships, including equivalences and differences, with existing 
   fractional Sobolev spaces were also established.
   Furthermore, the notion of weak fractional derivatives was systematically extended 
   to general distributions instead of only to some special distributions as done in the literature. 
   It is expected (and our hope, too) that these newly developed theories of weak fractional 
   differential calculus and fractional order Sobolev spaces will lay down a solid theoretical  
   foundation for systematically and rigorously developing a fractional calculus of variations 
   theory and a fractional PDE theory as well as their numerical solutions. Moreover, we hope this work 
   will stimulate more research on and applications of fractional calculus and 
   fractional differential equations in the near future.


\end{document}